\documentclass[11pt]{article}
\usepackage{amsthm,amsmath,amssymb}
\usepackage[numbers,sort]{natbib}
\usepackage{multirow}
\usepackage{caption}
\usepackage{subcaption}
\usepackage{float}
\usepackage{makecell}
\usepackage{booktabs}
\usepackage{array}
\usepackage{fullpage}
\usepackage{url}
\usepackage{algorithm}
\usepackage{algorithmic}
\usepackage{bm}
\usepackage{authblk}
\usepackage{cancel}
\usepackage{bbm}
\usepackage{smile}
\usepackage{mathtools}
\usepackage{lipsum}
\usepackage{mathrsfs}
\usepackage{dsfont}
\usepackage{titling}
\usepackage{epstopdf}
\usepackage{caption}
\usepackage{marvosym}
\usepackage{diagbox}
\usepackage{slashbox}
\usepackage{fullpage}

\usepackage{epsfig}
\usepackage{pst-grad} 
\usepackage{pst-plot} 
\usepackage[space]{grffile} 
\usepackage{etoolbox} 
\makeatletter 
\patchcmd\Gread@eps{\@inputcheck#1 }{\@inputcheck"#1"\relax}{}{}
\makeatother

\makeatletter
\DeclareRobustCommand\widecheck[1]{{\mathpalette\@widecheck{#1}}}
\def\@widecheck#1#2{%
    \setbox\z@\hbox{\m@th$#1#2$}%
    \setbox\tw@\hbox{\m@th$#1%
       \widehat{%
          \vrule\@width\z@\@height\ht\z@
          \vrule\@height\z@\@width\wd\z@}$}%
    \dp\tw@-\ht\z@
    \@tempdima\ht\z@ \advance\@tempdima2\ht\tw@ \divide\@tempdima\thr@@
    \setbox\tw@\hbox{%
       \raise\@tempdima\hbox{\scalebox{1}[-1]{\lower\@tempdima\box
\tw@}}}%
    {\ooalign{\box\tw@ \cr \box\z@}}}
\makeatother

\newcommand{\T}{^\top}

\usepackage{natbib}
\usepackage{multirow}

\usepackage[colorlinks=true,
            linkcolor=blue,
            urlcolor=blue,
            citecolor=blue]{hyperref}

\usepackage[utf8]{inputenc}
\definecolor{fondpaille}{cmyk}{0,0,0.1,0}

\newcommand\independent{\protect\mathpalette{\protect\independenT}{\perp}}
\def\independenT#1#2{\mathrel{\rlap{$#1#2$}\mkern2mu{#1#2}}}

\newcommand*{\diam}{\operatorname{diam}}
\newcommand*{\TV}{\operatorname{TV}}

\newcommand*{\Tr}{\operatorname{Tr}}

\newcommand*{\supp}{\mathrm{supp}}

\begin{document}

\title{\huge Minimax Optimal Conditional Independence Testing}
\author{Matey Neykov \and Sivaraman Balakrishnan \and Larry Wasserman}
\date{Department of Statistics \& Data Science\\ Carnegie Mellon University\\Pittsburgh, PA 15213\\[2ex]\texttt{\{mneykov, siva, larry\}@stat.cmu.edu}}

\maketitle

\begin{abstract} 
We consider the problem of conditional independence testing of $X$ and $Y$ given $Z$ where $X,Y$ and $Z$ are three real random variables and $Z$ is continuous. 
We focus on two main cases -- when $X$ and $Y$ are both discrete, and when $X$ and $Y$ are both continuous. In view of recent results on 
conditional independence testing \citep{shah2018hardness}, one cannot hope to design non-trivial tests, which control the type I error for all absolutely continuous conditionally independent 
distributions,
while
still ensuring power against interesting alternatives.
Consequently, we identify various, natural 
smoothness assumptions on the conditional distributions of $X,Y|Z=z$ as $z$ varies in the support of $Z$, and study the hardness of conditional independence testing under these smoothness
assumptions.
We derive matching lower and upper bounds on the critical radius of separation between the null and alternative hypotheses in the total variation metric. The tests we consider are easily implementable and rely on binning the support of the continuous variable $Z$. To complement these results, we provide a new proof of the hardness result of \citet{shah2018hardness}. 
\end{abstract}

\section{Introduction}
Conditional independence (CI) testing is a fundamental problem, with widespread applications throughout statistics. 
From being a foundation of basic concepts such as sufficiency and ancillarity~\cite{dawid1979conditional}, to its applications in 
estimation and inference for graphical models \cite{margaritis2005distribution,Koller2009Probabilistic} and in 
causal inference and causal discovery \cite{zhang2012kernel, spirtes2000causation, pearl2014probabilistic}, the concept of conditional independence and conditional independence testing play a central role in the fields of statistics, machine learning and related areas. 
A large body of work has focussed on CI testing under the assumption of joint Gaussianity. In this setting 
CI testing corresponds to testing whether certain partial correlations between the variables are zero. 
Since partial correlations are (relatively) easy to estimate, the Gaussian assumption gives a shortcut to CI testing, but if the model is non-Gaussian this can lead to misleading conclusions as variables could be conditionally dependent even with zero partial correlation. 
In practice, the Gaussian assumption is unlikely to hold exactly and many applications call for the additional flexibility provided by nonparametric CI testing.



In this paper we consider CI testing from a nonparametric perspective. Following \citet{dawid1979conditional}, given three random vectors $(X,Y,Z) \in \RR^{d_X + d_Y + d_Z}$ we will denote the 
CI of $X$ and $Y$ given $Z$ by $X \independent Y | Z$. In the case when $d_X = d_Y = d_Z = 1$ and $Z$ is a continuous random variable supported on $[0,1]$, we construct nonparametric tests which are capable of testing the null hypothesis $X \independent Y | Z$ versus the alternative $X \not \independent Y | Z$. The variables $X$ and $Y$ are allowed to be either both discrete or both continuous supported on $[0,1]$. It was recently argued in a precise mathematical sense \cite{shah2018hardness} that CI testing is a statistically hard task for absolutely continuous (with respect to the Lebesgue measure) random variables --- namely if one wants to have a test that controls the type I error for all absolutely continuous triplets $(X,Y,Z)$ such that $X \independent Y | Z$, such a test cannot have power against any alternative. This discouraging result demystified the fact that despite a large body of literature on the subject, no fully satisfactory CI tests had been developed for continuous random variables. 

Concurrently with the paper of Shah and Peters \cite{shah2018hardness}, the work of \citet{canonne2018testing} constructed tests for CI of \emph{discrete} distributions $(X,Y,Z)$ which are minimax optimal in certain regimes. 
Part of the effort of this paper is devoted to extending the ideas of \citet{canonne2018testing} to the case when $Z$ is an absolutely continuous random variable on $[0,1]$.

In order to characterize the difficulty of CI testing in this setting we adopt the minimax perspective \cite{ingster1982minimax,ingster2003nonparametric}. Naturally, 
if an alternative distribution is 
very close to a null distribution (in a certain metric such as the total variation metric) 
it will be very difficult to test for CI given a finite number of $n$ samples. 
By discarding distributions under the alternative that are ``$\varepsilon_n$-close'' 
to the null hypothesis we are able to set up a well-defined testing problem. The goal in minimax hypothesis testing is then to characterize the optimal ``critical radius'' $\varepsilon_n$, i.e. the smallest $\varepsilon_n$ 
at which it is possible to reliably distinguish the null from the $\varepsilon_n$-separated alternative, as a function of the sample size $n$. 
This standard step of discarding ``near-null distributions'' is insufficient 
as one cannot hope to design a non-trivial test which controls the type I error for \emph{all} conditionally independent absolutely continuous triplets \cite{shah2018hardness}. 
In order to make the problem of CI testing well-posed we further impose certain natural 
smoothness assumptions on the conditional distributions of $X,Y|Z=z$ as $z$ varies in the support of $Z$, and establish upper and lower bounds on the critical radius of 
conditional independence testing under these smoothness
assumptions. 

\subsection{Related Work}


As we mentioned earlier, there is a large body of work on independence and CI testing. We focus our review on the literature most relevant to our approach. 
It is worth noting that almost all relevant works considered here, with the notable exception of \citet{canonne2018testing} who consider minimax CI testing for discrete distributions, do not take a minimax perspective to the problem. 
We are not aware of tests that achieve the minimax rates for testing CI with a continuous random variable $Z$ other than the ones that we develop in this paper. In addition, we would like to note that the ideas introduced by Cannone et al. \cite{canonne2018testing} are instrumental in the development of the minimax rates in the present work. In particular \cite{canonne2018testing} offer a variety of results in the discrete $X,Y,Z$ conditional independence testing, including lower and upper bounds on the sample complexity. We borrow key constructs from this work, particularly an unbiased estimator of the $L^2_2$ distance, and tools to analyze its variance and expectation under Poisson sampling in order to come up with upper bounds for our estimators.  

Given knowledge of the conditional distribution of $X | Z$, \citet{berrett2018conditional} develop a permutation-based test for testing the null hypothesis of CI. 
We note that from a minimax perspective knowing $X|Z$ changes the problem of CI testing significantly and we do not address this CI testing variant here. 
The works \cite{bergsma2010nonparametric, bergsma2004testing} propose a partial copula approach, which needs estimators of the conditional distributions of $X|Z$ and $Y|Z$. Since estimation is typically more costly than testing, we anticipate that such a procedure does not attain minimax optimal rates for the critical radius. In a setting different from the present paper, \citet{song2009testing} proposes a CI test for two variables given a single index of a random vector via ``Rosenblatt transforms'', which are multivariate extensions of the probability integral transform. The techniques in this work also involve estimation of certain conditional distributions via kernel smoothing. Huang \cite{huang2010testing} proposes a nonparametric CI test using the so called maximal nonlinear conditional correlation. The author proves that under the null hypothesis given that certain conditions hold, the test achieves asymptotic normality. This work once again requires kernel smoothed estimates of certain conditional expectations and is therefore unlikely to result in minimax optimal tests of CI. In an interesting paper, Gy\"{o}rfy and Walk \cite{gyorfi2012strongly}, extend the independence testing results of \citet{gretton2010consistent} to the CI case, and propose strongly consistent nonparametric tests. We believe however that there is a gap 
in one of the proofs of this work, which would otherwise seem to contradict the CI hardness results of \citet{shah2018hardness}. In particular, in the proof of Theorem 1 of \citet{gyorfi2012strongly}, it is claimed that the following expression is $0$
\begin{align*}
\bigg|\PP(X \in A, Y \in B, Z \in C) - \frac{\PP(X \in A, Z \in C)\PP(Y \in A, Z \in C)}{\PP(Z \in C)}\bigg|,
\end{align*}
under the null hypothesis of independence, where $A,B,C$ are elements of a partition of the domains of $X,Y,Z$ respectively. Note that this need not hold in general for conditionally independent distributions since averaging over $Z$ does not necessarily preserve independence. In fact, this is one of the major complications that we have to deal with in our proofs.

\citet{patra2016nonparametric} design a novel nonparametric residual between a random variable and a random vector and use it to develop tests of CI with the help of the bootstrap. 
 An innovative approach to nonparametric CI testing using a nearest neighbor bootstrap and converting the testing problem to a classification problem was recently proposed by \citet{sen2017model}. \citet{fukumizu2008kernel} give a measure of CI of random variables, based on normalized cross-covariance operators on reproducing kernel Hilbert spaces. Different reproducing kernel based methods were proposed by \citet{zhang2012kernel} and \citet{doran2014permutation} respectively. The recent work of Shah and Peters \cite{shah2018hardness}, along with the hardness result, proposes CI tests based on the so called generalized covariance measure which is a measure related to the normalized residuals of regressing $X$ and $Y$ on $Z$. In another recent paper, \citet{azadkia2019simple} propose a novel measure of CI which takes values in $[0,1]$, where the measure takes the value $0$ when the variables are conditionally independent, and is equal to $1$ when the $Y$ is a measurable function of $X$ given $Z$. 
 
There is also a significant amount of work on CI testing in the econometrics literature (see for instance~\cite{su2007consistent, su2008nonparametric, su2014testing, wang2018characteristic}).
Su and White \cite{su2008nonparametric} give a Hellinger distance based approach to CI testing, which employs a plug in based estimate using kernel smoothed estimates of the joint and conditional densities of $X,Y,Z$. In follow-up work, Su and White \cite{su2007consistent} propose estimating a functional involving the difference of two conditional characteristic functions. They show asymptotic normality under the null hypothesis and explore the power of the test based on this estimator under local alternatives. In the work \cite{su2014testing} the authors propose an empirical likelihood based approach to CI testing. \citet{wang2018characteristic} develop a new test based on characteristic functions, which achieves faster rates against certain local alternatives in comparison to the test developed by \citet{su2007consistent}. 

So far we have discussed works which focus on nonparametric CI testing in the continuous case. It is noteworthy that there are also numerous CI tests in the discrete case as well. See for example the works 
\cite{rosenbaum1984testing, agresti1992survey, yao1993exact, canonne2018testing} as well as references therein.

\subsection{Summary of Results}
\label{sec:summary}

We will now informally summarize the main findings of our work. For the most part this paper is focused on the following two cases:
\begin{enumerate}
\item When $X$ and $Y$ are discrete supported on $[\ell_1]\times [\ell_2]$ for some integers $\ell_1, \ell_2$ (here $[\ell_1] = \{1,2,\ldots,\ell_1\}$ and similarly for $[\ell_2]$), and when $Z$ has an 
absolutely continuous (with respect to the Lebesgue measure) distribution supported on $[0,1]$, 
\item When all three variables $(X,Y,Z)$ have an absolutely continuous (with respect to the Lebesgue measure) distribution supported on $[0,1]$. 
\end{enumerate}
We study the minimax rate for the critical radius $\varepsilon_n$ which we define as  the separation between the null and alternative hypothesis, in the total variation (TV) distance, required to reliably distinguish them. Formally, we consider 
distinguishing,
\begin{align*}
H_0: &~~p_{X,Y,Z}~\mbox{ s.t. }~X \independent Y | Z~~\text{versus}~~ \\
H_1: &~~p_{X,Y,Z}~\mbox{ s.t. }~\inf_{q \mbox{ \scriptsize in } H_0}\|p_{X,Y,Z}-q\|_1 \geq \varepsilon_n. 
\end{align*}
In addition, we remove distributions under $H_0$ and $H_1$ which are not smooth enough, i.e., $p_{X,Y| Z = z}$ is not a smooth function of $z$ (for precise definitions refer to Section~\ref{sec:smoothness}). Given this set-up, our interest is in 
finding the smallest possible $\varepsilon_n$ such that even in the worst-case scenario for distributions under $H_0$ and under $H_1$ the sum of the type I and type II errors can be controlled under a pre-specified threshold.

\begin{enumerate}
\item Let us first discuss the case when $X$ and $Y$ are discrete on $[\ell_1]\times [\ell_2]$ where $\ell_1$ and $\ell_2$ are fixed integers which are not allowed to scale with $n$. In this setting we show that 
\begin{align*}
\varepsilon_n \asymp n^{-2/5}.
\end{align*}
That is we show matching minimax lower and upper bounds at the optimal rate of the critical radius which is given by $n^{-2/5}$. Here we use $\asymp$ to mean equal up to a positive absolute constant. 

\item Next, consider the more general case when $\ell_1$ and $\ell_2$ are allowed to scale with $n$. Then we are able to show that 
\begin{align*}
\varepsilon_n \gtrsim \frac{(\ell_1 \ell_2)^{1/5}}{n^{2/5}} \wedge 1,
\end{align*}
and we have a matching upper bound (i.e., a test) whenever, for $\ell_1 \geq \ell_2$, we have $\frac{\ell_1^{4}}{\ell_2} \lesssim n^{3}$. 
We further show that this latter condition holds whenever $\ell_1 \asymp \ell_2$. Here, and throughout this paper, $\gtrsim$ and $\lesssim$ mean inequalities up to a positive absolute constant.

\item Finally in the fully continuous case we show that 
\begin{align*}
\varepsilon_n \asymp n^{-2s/(5s + 2)}, 
\end{align*}
where $s$ denotes the H\"{o}lder smoothness parameter of the conditional density $p_{X,Y|Z}$ under the alternative hypothesis. 
\end{enumerate}
\begin{center}
\begin{table}
\begin{tabular}{ rc|c|c| }& \multicolumn{3}{c}{$X,Y$}\\
\hline
\multicolumn{1}{|r|}{}& discrete on $[\ell_1] \times [\ell_2]$, $\ell_1,\ell_2$ fixed &  discrete on $[\ell_1] \times [\ell_2]$ & continuous \\ 
  \hline
 \multicolumn{1}{|r|}{$\varepsilon_n$-Upper Bounds} & $n^{-2/5}$ & $\frac{(\ell_1\ell_2)^{1/5}}{n^{2/5}}$, given $\frac{\ell_1^4}{\ell_2} \lesssim n^3$ & $n^{-2s/(5s + 2)}$ \\ 
 \multicolumn{1}{|r|}{$\varepsilon_n$-Lower Bounds} & $n^{-2/5}$ & $\frac{(\ell_1\ell_2)^{1/5}}{n^{2/5}}$ & $n^{-2s/(5s + 2)}$\\ 
 \hline
\end{tabular}
 \caption{This is a summary of the minimax results obtained in the main text of our paper. }
 \label{table:minimax:rates}
\end{table}
\end{center}

Our results are also summarized in Table \ref{table:minimax:rates}. The tests used to achieve the upper bounds for the above minimax rates, are computationally tractable and we implement them and provide some numerical results. Our tests do not require kernel smoothing. They are rather calculated based on binning the support of $Z$ (and $X$ and $Y$ when they are continuous) into a certain sample-size dependent number of bins. For each $Z$-bin, a (weighted) U-statistic is calculated and the resulting statistics are summed up according to appropriate weighting across the $Z$-bins. Roughly, the U-statistics target the $L_2^2$ distance between $p_{X,Y|Z}$ and $p_{X|Z}p_{Y|Z}$ within each of the $Z$-bins (or in the weighted U-statistic case a distance similar to the chi-square distance between $p_{X,Y|Z}$ and $p_{X|Z}p_{Y|Z}$). This strategy also reveals the need to impose certain smoothness assumptions on the conditional distribution of $p_{X,Y|Z = z}$ in $z$ since otherwise the binning may result in unreliable estimates of the $L_2^2$ distance. 

Along with the aforementioned results we also provide a new proof of the hardness result of Shah and Peters \cite{shah2018hardness}. Our proof is based on a coupling between an arbitrary absolutely continuous distribution and a statistically independent distribution, which bears some resemblance to the coupling used in Lemma 14 of \cite{shah2018hardness}. We use this coupling to show the fact that conditionally independent distributions are Wasserstein dense in the set of all absolutely continuous distributions of bounded support. 

\subsection{Organization}

The paper is structured as follows. We present some basic background in Section~\ref{sec:background}. We revisit the hardness results of \citet{shah2018hardness} in Section \ref{hardness:section}. Minimax lower bounds on the critical radius are given in Section \ref{minimax:lower:bounds:section}. Section \ref{upper:bound:section} is devoted to developing tests of CI which match the lower bounds of Section \ref{minimax:lower:bounds:section}. Section \ref{examples:section} gives examples for distributions satisfying the smoothness assumptions we impose in Sections \ref{minimax:lower:bounds:section} and \ref{upper:bound:section}. Section \ref{numerical:experiments:section} provides a brief numerical study, which is meant to show that our nonparametric tests are in fact readily implementable and perform well in practice. Finally a discussion is provided in Section \ref{discussion:section}. 

\section{Background}
\label{sec:background}
In this section, following some basic notation, we present some background on minimax testing and briefly introduce the various smoothness conditions we use in our minimax upper and lower bounds.
\subsection{Notation}
We make extensive usage of metrics on probability distributions in this paper. The total variation (TV) metric between two distributions $p,q$ on a measurable space $(\Omega, \cF)$ is defined as
\begin{align*}
d_{\TV}(p,q) = \sup_{A \in \cF} |p(A) - q(A)| = \frac{1}{2}\|p-q\|_1 = \frac{1}{2}\int \bigg|\frac{dp}{d\nu} - \frac{dq}{d\nu}\bigg|d\nu,
\end{align*}
where the last identity assumes $\nu$ is a common dominating measure of $p$ and $q$, i.e., $p \ll \nu$ $q \ll \nu$ and $\frac{d p}{d \nu}, \frac{dq}{d\nu}$ denote the densities of $p$ and $q$ with respect to $\nu$ (note here that $\nu$ can always be taken as $\nu = p + q$). Under the latter assumption one can also define the $L_2$ distance between $p $ and $q$ as
\begin{align*}
\|p - q\|_2 = \bigg[\int\bigg|\frac{dp}{d\nu} - \frac{dq}{d\nu}\bigg|^2 d\nu\bigg]^{1/2}. 
\end{align*}
Assuming that $p \ll q$ we may define the $\chi^2$-divergence between $p$ and $q$ as
\begin{align*}
d_{\chi^2}(p,q) = \int \bigg(\frac{d p}{d q} - 1\bigg)^2 dq.
\end{align*}
If $p \ll q$ fails to hold then we take $d_{\chi^2}(p,q) = \infty$. 

Next we formalize our notation for conditional distributions. If the triplet $(X,Y,Z)$ has a distribution $p_{X,Y,Z}$ we will use $p_{X,Y|Z = z}$ to denote the conditional joint distribution of $X,Y | Z = z$. Additionally $p_{X|Z = z}$ and $p_{Y|Z = z}$ will denote the marginal conditional distributions of $X|Z = z$ and $Y|Z = z$ respectively. The marginal distributions will be denoted with $p_X, p_Y, p_Z$ and joint marginal distributions will be denoted with $p_{X,Y}, p_{Y,Z}, p_{X,Z}$. Furthermore, with a slight abuse of notation, $p_{X,Y|Z}(x,y|z)$ and $p_{X | Z}(x|z)$ and $p_{Y| Z }(y|z)$ will denote the densities of these distributions evaluated at the points $x,y$ and $z$ (or the corresponding probability mass functions when $X$ and $Y$ are discrete).

In addition we will use $\lesssim$ and $\gtrsim$ to mean $\leq$ and $\geq$ up to positive universal constants (which may be different from place to place). If both $\lesssim$ and $\gtrsim$ hold we denote this as $\asymp$. For an integer $n \in \NN$ we use the convenient shorthand $[n] = \{1,2,\ldots,n\}$. 

\subsection{Minimax Testing}

In order to characterize the complexity of CI testing we use the minimax testing framework, introduced in the work of Ingster and co-authors \cite{ingster1982minimax,ingster2003nonparametric}, and which has since then been 
considered by many authors (see for instance \cite{balakrishnan2017hypothesis,balakrishnan2018hypothesis,baraud2002non,diakonikolas2016new,canonne2018testing,valiant2017automatic,ery2018remember, canonne2015survey}).
Formally, consider the testing problem
\begin{align}
\label{eqn:main}
H_0: p_{} \in \cH_0 \mbox{ vs } H_1: p_{} \in \cS_1(\varepsilon),
\end{align}
where $\cS_1(\varepsilon) := \{p \in \cH_1: \inf_{q \in \overline \cH_0} \|p - q\|_1 \geq \varepsilon\}$, and $\cH_0\subseteq \overline \cH_0$ and $\cH_1$ are pre-specified sets of distributions. We define the minimax risk of testing as 
\begin{align}\label{minimax:risk}
R_n(\cH_0,\overline \cH_0, \cH_1, \varepsilon) = \inf_{\psi}\bigg\{\sup_{p \in \cH_0} \EE_p[\psi(\cD_n)] + \sup_{p \in \cS_1(\varepsilon)} \EE_p [1 - \psi(\cD_n)]\bigg\}\footnotemark,
\end{align}
\footnotetext{Here and throughout with a slight abuse of notation we use $\EE_p$ to denote expectation under i.i.d. data $\cD_n  = \{(X_1,Y_1,Z_1), \ldots, (X_n, Y_n, Z_n)\}$ where each observation is drawn from $p$.}where the infimum is taken over all Borel measurable test functions $\psi: \supp(\cD_n) \mapsto [0,1]$ (which gives the probability of rejecting the null hypothesis), and $\supp(\cD_n)$ is the support of the random variables $\cD_n = \{(X_1,Y_1, Z_1),\allowbreak \ldots (X_n, Y_n, Z_n)\}$. 
We note that it is common to choose the sets $\cH_0$ and $\overline \cH_0$ to be identical. However, as will be clearer hereafter, in the setting of CI testing we will choose $\cH_0$ to be a subset of distributions which are conditionally independent \emph{and} appropriately smooth, while we will choose $\overline \cH_0$ to be the set of \emph{all} conditionally independent distributions. 

In the minimax framework our goal is to study the \emph{critical radius} of testing defined as
\begin{align}\label{critical:radius}
\varepsilon_n(\cH_0,\overline \cH_0, \cH_1) = \inf\bigg\{\varepsilon : R_n(\cH_0, \overline \cH_0, \cH_1, \varepsilon) \leq \frac{1}{3}\bigg\}. 
\end{align}
The constant $\frac{1}{3}$ above is arbitrary, and can be chosen as any small constant. 
The minimax testing radius or the critical radius, corresponds to the smallest radius $\varepsilon$ at which there exists \emph{some test} which distinguishes distributions in $\cH_0$ from those in $\cH_1$ which
are appropriately far from $\cH_0$. The critical radius provides a fundamental characterization of the statistical difficulty of the hypothesis testing problem in~\eqref{eqn:main}.

\subsection{Smoothness Conditions}
\label{sec:smoothness}
In Sections~\ref{minimax:lower:bounds:section} and~\ref{upper:bound:section} we derive upper and lower bounds on the minimax critical radius for conditional independence testing. However, in view of the
results of \citet{shah2018hardness}, and our own results in Section~\ref{hardness:section}, we must impose some restrictions on the distributions under consideration in order to obtain non-trivial minimax rates.
Broadly, we restrict our attention to settings where the conditional distributions are appropriately smooth.

We focus on two main settings in our work, the setting where $X$ and $Y$ are discrete but $Z$ is continuous and when all three are continuous.
For the case when $X$ and $Y$ are discrete and $Z$ is continuous, we consider $Z$ that is supported on $[0,1]$. 
Define the set of distributions $\cE'_{0,[0,1]}$ as distributions whose generating mechanism of the triple $(X,Y,Z)$ supported on $\RR^3$ is as follows: first a $Z$ from the distribution $p_Z$ (which is absolutely continuous with respect to the Lebesgue measure) with support $[0,1]$ is generated. Next, $X$ and $Y$ are generated from the distribution $p_{X,Y|Z}$ which is supported on\footnote{It is not crucial here that $X,Y \vert Z$ is supported on $[\ell_1]\times [\ell_2]$. It could be supported on any set $\cX \times \cY$ with $|\cX| = \ell_1$ and $|\cY| = \ell_2$. Here for the sake of simplicity of presentation we focus only on the case $[\ell_1] \times [\ell_2]$. } $[\ell_1] \times [\ell_2]$ for (almost) all $Z$. Denote by $\cP_{0,[0,1]}' \subset \cE'_{0,[0,1]}$ the set of null distributions (i.e. distributions such that $X \independent Y | Z$) and let $\cQ_{0, [0,1]}' = \cE'_{0,[0,1]} \setminus \cP_{0,[0,1]}'$. 
Similarly, in the case when $X, Y$ and $Z$ are continuous, we 
let $\cP_{0,[0,1]^3} \subset \cE_{0, [0,1]^3}$ be the set of distributions for which $X \independent Y | Z$ and let $\cQ_{0,[0,1]^3} = \cE_{0, [0,1]^3} \setminus \cP_{0,[0,1]^3}$.

With these preliminaries in place we can define the various smoothness classes that we work with in this paper:
\begin{definition}[Null Lipschitzness] \leavevmode 
\begin{enumerate}
\item {Null TV Lipschitzness: } Let $\cP_{0,[0,1], \TV}'(L) \subset \cP_{0, [0,1]}'$ (analogously $\allowbreak \cP_{0,[0,1]^3, \TV}(L) \subset \cP_{0, [0,1]^3}$) 
be the collection of distributions $p_{X,Y,Z}$ such that for all $z,z' \in [0,1]$ we have:
\begin{align*}
\|p_{X | Z = z} - p_{X | Z = z'}\|_1 \leq L |z - z'| \mbox{ and } \|p_{Y | Z = z} - p_{Y | Z = z'}\|_1 \leq L |z - z'|,
\end{align*}
where $p_{X| Z = z}$ and $p_{Y | Z = z}$ denote the conditional distributions of $X | Z = z$ and $Y | Z = z$ under $p_{X,Y,Z}$  respectively. 
\item {Null $\chi^2$ Lipschitzness: } Let $\cP_{0,[0,1], \chi^2}'(L) \subset \cP_{0, [0,1]}'$ (analogously $\cP_{0,[0,1]^3, \chi^2}(L)\allowbreak \subset \cP_{0, [0,1]^3}$)
be the collection of distributions $p_{X,Y,Z}$ such that for all $z,z' \in [0,1]$ we have:
\begin{align*}
d_{\chi^2}(p_{X | Z = z}, p_{X | Z = z'}) \leq L |z - z'| \mbox{ and } d_{\chi^2}(p_{Y | Z = z}, p_{Y | Z = z'}) \leq L |z - z'|,
\end{align*}
where $p_{X| Z = z}$ and $p_{Y | Z = z}$ denote the conditional distributions of $X | Z = z$ and $Y | Z = z$ under $p_{X,Y,Z}$ respectively. The distance $d_{\chi^2}(p_{X | Z = z}, p_{X | Z = z'})$ is considered $\infty$ if $p_{X | Z = z} \ll p_{X | Z = z'}$ is violated. 
\item {Null H\"{o}lder Lipschitzness: } Let $\cP_{0,[0,1], \TV^2}'(L) \subset \cP_{0, [0,1]}'$ be the collection of distributions $p_{X,Y,Z}$ such that for all $z,z' \in [0,1]$ we have:
\begin{align*}
\|p_{X | Z = z} - p_{X | Z = z'}\|_1 \leq \sqrt{L |z - z'|} \mbox{ and } \|p_{Y | Z = z} - p_{Y | Z = z'}\|_1 \leq \sqrt{L |z - z'|},
\end{align*}
where $p_{X| Z = z}$ and $p_{Y | Z = z}$ denote the conditional distributions of $X | Z = z$ and $Y | Z = z$ under $p_{X,Y,Z}$  respectively. 
\end{enumerate}
\label{def:nullsmoothness}
\end{definition}
Under the alternative we consider slightly different classes in the discrete and continuous cases. Formally, we define the following class for the discrete $X$ and $Y$ setting:
\begin{definition}[Alternative TV Lipschitzness]\label{alternative:TV:smoothness:def} Let $\cQ_{0,[0,1], \TV}'(L) \subset \cQ_{0,[0,1]}'$ be the collection of distributions 
$p_{X,Y,Z}$ such that for all $z, z' \in [0,1]$ we have
\begin{align*}
\|p_{X,Y | Z= z} - p_{X,Y | Z = z'}\|_1 \leq L |z - z'|, 
\end{align*}
where $p_{X,Y| Z = z}$ denotes the conditional distribution of $X,Y | Z = z$ under $p_{X,Y,Z}$. 
\end{definition}
In the continuous case, we will further restrict our attention to distributions which in addition to being TV smooth (as above), also have smooth conditional density $p_{X,Y|Z}$. In order for us to impose proper smoothnes on the density $p_{X,Y|Z}$ we will first define a H\"{o}lder smoothness class.

\begin{definition}[H\"{o}lder Smoothness] Let $s > 0$ be a fixed real number, and let $\lfloor s \rfloor$ denote the maximum integer strictly smaller than  $s$. Denote by $\cH^{2,s}(L)$, the class of functions $f:[0,1]^2 \mapsto \RR$, which posses all partial derivatives up to order $\lfloor s \rfloor$ and for all $x,y,x',y' \in [0,1]$ we have
\begin{align}\label{holder:class}
\sup_{k \leq \lfloor s \rfloor}\bigg|\frac{\partial^k}{\partial x^k}\frac{\partial^{\lfloor s \rfloor - k}}{\partial y^{\lfloor s \rfloor - k}}f(x,y) - \frac{\partial^k}{\partial x^k}\frac{\partial^{\lfloor s \rfloor - k}}{\partial y^{\lfloor s \rfloor - k}}f(x',y') \bigg| \leq L((x-x')^2 + (y - y')^2))^{\frac{s - \lfloor s \rfloor}{2}},
\end{align}
and in addition
\begin{align*}
\sup_{k \leq \lfloor s \rfloor}\bigg|\frac{\partial^k}{\partial x^k}\frac{\partial^{\lfloor s \rfloor - k}}{\partial y^{\lfloor s \rfloor - k}}f(x,y)\bigg| \leq L.
\end{align*}
\end{definition}
In the above assumption in the addition to the usual H\"{o}lder smoothness assumption we assume that there is a uniform bound on all derivatives of lower than $\lfloor s \rfloor$ order. When $s = 1$ the above is simply the class of $L$-Lipschitz functions.

\begin{definition}[Alternative Lipschitzness]\label{alternative:smoothness:holder} Let $\cQ_{0,[0,1]^3, \TV}(L,s) \subset \cQ_{0,[0,1]^3}$ be the collection of distributions $p_{X,Y,Z}$ 
such that for all $z, z' \in [0,1]$ we have
\begin{align*}
\|p_{X,Y | Z= z} - p_{X,Y | Z = z'}\|_1 \leq L |z - z'|, 
\end{align*}
where $p_{X,Y| Z = z}$ denotes the conditional distribution of $X,Y | Z = z$ under $p_{X,Y,Z}$. In addition we assume that for all $z,x,y \in [0,1]$: $p_{X,Y|Z}(x,y|z) \in \cH^{2,s}(L)$. 
\end{definition}
We devote Section~\ref{examples:section} to investigating various relationships between these different Lipschitzness assumptions, 
as well as to constructing broad nonparametric classes of distributions which satisfy these Lipschitzness conditions.

\section{The Hardness of CI Testing Revisited}\label{hardness:section}

In this section we revisit the recent work of \citet{shah2018hardness}. In order for us to review their results, and to build upon them, we will recall their notation. Let $\cE_0$ denote the set of all distributions for $(X,Y,Z)$ on $\RR^{d_X + d_Y + d_Z}$ which are absolutely continuous with respect to the Lebesgue measure. Define the set of conditionally independent distributions, i.e., distributions such that $X \independent Y | Z$,  as $\cP_0 \subset \cE_0$. Let $\cE_{0,M} \subseteq \cE_0$ be the set of distributions whose support is contained within an $L_\infty$ ball of radius $M$. Define the set of alternative distributions as $\cQ_0 = \cE_0 \setminus \cP_0$ and $\cP_{0,M} = \cE_{0,M} \cap \cP_0$ and $\cQ_{0,M} = \cE_{0,M} \cap \cQ_0$. 

In their Proposition 5, Shah and Peters argue that the null and alternative sets of distributions $\cP_{0,M}$ and $\cQ_{0,M}$ are separated in TV distance. Here separated is meant in the sense that there exists a distribution from $\cQ_{0,M}$ which is at least $1/24$ apart in TV distance from any distribution in $\cP_{0,M}$. Similarly, in Proposition 16, Shah and Peters argue that the sets of distributions $\cP_0$ and $\cQ_0$ are separated in KL divergence (in this proposition they consider only the case $(X,Y,Z) \in \RR^3$). In contrast, the first result of this section will show that when the Wasserstein distance is considered, the set of distributions $\cP_{0,M}$ is dense in the set $\cQ_{0,M}$. Let us first define the Wasserstein distance. 

\begin{definition}[Wasserstein Distance] Let $p \geq 1$ be a real number. Let $\cP_p(\RR^d)$ denote the set of measures $\mu$ on $(\RR^d, \|\cdot\|_2)$, such that there exists $x_0 \in \RR^d$ for which
\begin{align*}
\int_{\RR^d} \|x - x_0\|_2^p d \mu(x) < \infty.
\end{align*}
For two probability measures, $\mu$ and $\nu$ in $\cP_p(\RR^d)$ the $p$\textsuperscript{th} Wasserstein distance between $\mu$ and $\nu$ is defined as 
\begin{align*}
W_p(\mu, \nu) = \bigg(\inf_{\gamma \in \Gamma(\mu,\nu)} \int_{\RR^d \times \RR^d} \|x- y\|^p_2 d \gamma(x,y)\bigg)^{1/p},
\end{align*}
where $\Gamma(\mu,\nu)$ is set of all couplings between the measures $\mu$ and $\nu$, i.e., all probability measures on $\RR^d \times \RR^d$, with marginals $\mu$ and $\nu$. 
\end{definition}

We are now ready to state the first result of this section.

\begin{lemma}[Wasserstein Denseness]\label{Wasserstein:dense:lemma} Take any distribution $P \in \cE_{0,M}$ for some $M > 0$. Then for any $p \geq 1$ and any $\varepsilon > 0$ there exists a distribution $Q \in \cP_{0,M}$ such that
$$
W_p(P, Q) \leq \varepsilon. 
$$
\end{lemma}

\begin{proof} For simplicity, we will prove this result for the one dimensional case $d_X = d_Y = d_Z = 1$. The proof extends trivially to the more general case. First note that since both $P, Q \in \cE_{0,M} \subseteq \cP_p(\RR^3)$, the Wasserstein distance between $P$ and $Q$ is well defined. We will now construct $Q$ from $P$ by describing a coupling between the two distributions. 

\begin{figure}
\centering
\includegraphics[scale=.4]{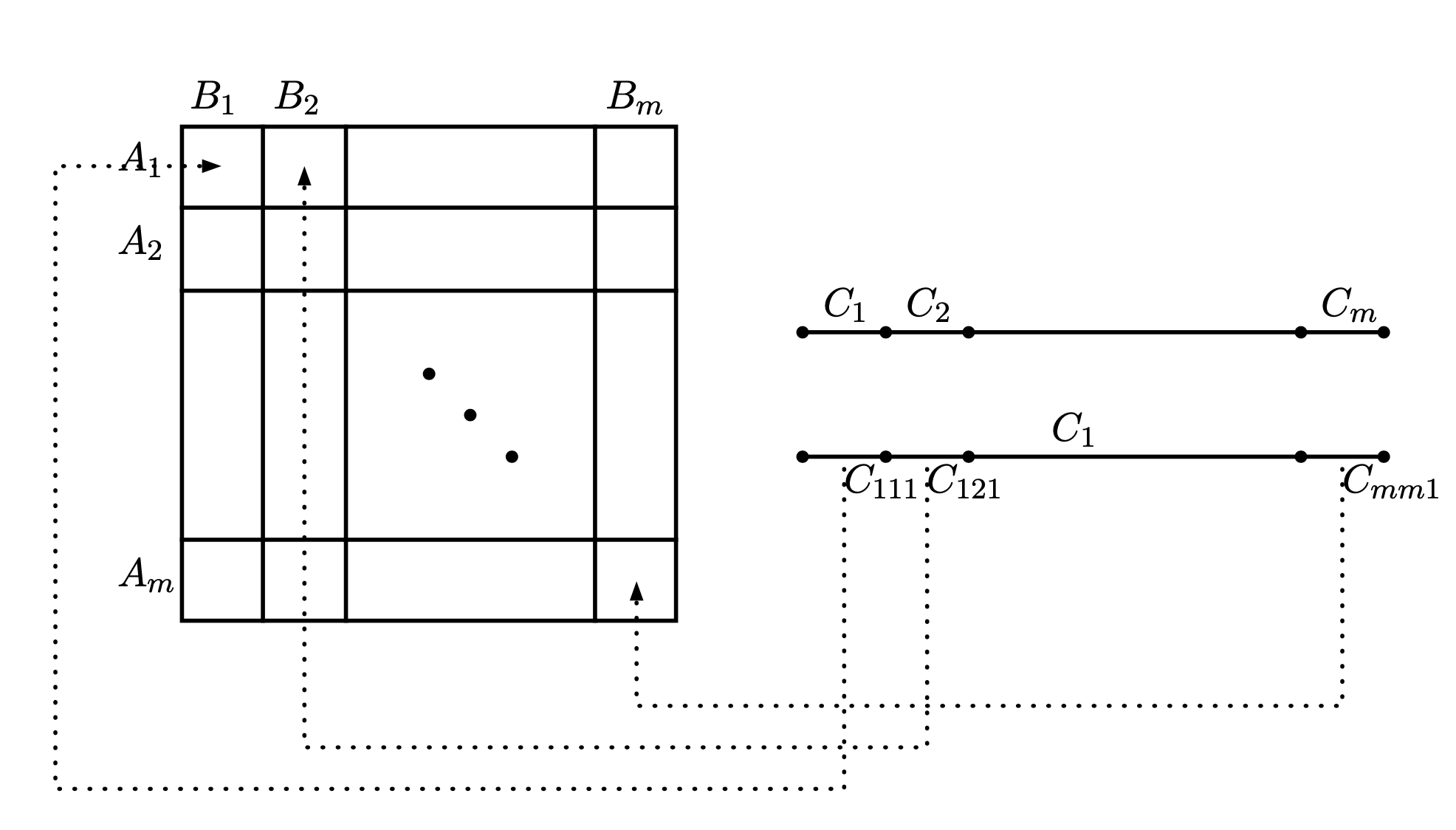}
\caption{This schematic describes the construction of $Q$ from $P$. $[-M,M]$ is divided in intervals $\{A_1, \ldots, A_m\}$, $\{B_1, \ldots,  B_m\}$ and $\{C_1, \ldots, C_m\}$. Next each interval $C_k$ is sub-divided into $m^2$ smaller sub-intervals. The interval $C_1$ is displayed along with its sub-divisions of $C_{ij1}$ for $i,j \in [m]$. Each little interval $C_{ij1}$ corresponds to a pair $(A_i, B_j)$ or equivalently to a cell $A_i \times B_j$ in $[-M,M]^2$.}\label{viz:smart:construction}
\end{figure}

Let $\{A_1, \ldots, A_{m}\}$ denote an equi-partition of $[-M,M]$ in intervals. Similarly let $\{B_1,\ldots, B_{m}\}$ and $\{C_1, \ldots, C_m\}$ be equi-partitions of $[-M,M]$. Divide each $C_k$ further in $m^2$ sub-intervals of equal length denoted by $C_{ijk}$, so that each of these small intervals corresponds to a pair $(A_i, B_j)$. Refer to Figure \ref{viz:smart:construction} for a visualization of this construction. The lengths of each interval $A_i$, $B_i$ or $C_i$ is $\frac{2M}{m}$, while the length of an interval $C_{ijk}$ is $\frac{2M}{m^3}$. Given a draw $(X,Y,Z) \sim P$ we construct $(\tilde X, \tilde Y, \tilde Z) \sim Q$ as follows. Suppose that $X \in A_i$, $Y \in B_j$ and $Z \in C_k$. Then we generate uniformly $\tilde Z \in C_{ijk}$ and $(\tilde X, \tilde Y)$ uniformly in $A_i \times B_j$. By definition then $\tilde X \independent \tilde Y | \tilde Z$, and further $\tilde X, \tilde Y, \tilde Z \in [-M,M]$. Hence $Q \in \cP_{0,M}$\footnote{For a precise expression of the density of $Q$ refer to Appendix \ref{hardness:appendix}.}. Furthermore we can bound the Wasserstein distance for this particular coupling as 
$$
W_p(P,Q)^p \leq \EE \EE_{(\tilde X, \tilde Y, \tilde Z) | (X,Y,Z)} \|(X,Y,Z) - (\tilde X,\tilde Y, \tilde Z)\|_2^p \leq \bigg(\sqrt{3}\frac{2  M}{m}\bigg)^p.
$$
Since $m$ can be selected arbitrarily large the above can be made smaller than $\varepsilon^p$. This completes the proof.
\end{proof}


The construction used to obtain $Q$ from $P$ in the above result captures intuitively the essence of the ``hardness'' of CI testing with continuous $Z$. The set $\cP_{0,M}$ contains distributions which allow the conditional distributions of $X,Y | Z = z$ to be ``wildly discontinuous'' as functions of $z$. This in turn allows for the existence of distributions in $\cP_{0,M}$ capable of approximating any distribution in $\cE_{0,M}$ in the Wasserstein metric. Later in this paper we will see that, if we disallow distributions in $\cE_{0,M}$ whose conditional distributions can be wildly variable in $z$, CI testing becomes possible.  We would also like to point out the intuition why the Wasserstein distance yields a result like Lemma \ref{Wasserstein:dense:lemma} in contrast to using TV distance or KL divergence. The Wasserstein distance is based on a metric (in our case the $L_2$ metric) on the underlying sample space, and as a consequence has the critical feature (unlike the KL divergence or TV distance) that it is robust to 
small perturbations in the sample space (on the other hand, metrics like the TV metric are typically stable to small perturbations in the probability space). Indeed, the heart of the construction of \citet{shah2018hardness} and of our own result, is the idea that given a sample
from a conditionally dependent distribution, one can perturb it slightly (effectively ``encoding'' the value of $X$ in $Z$) to create a conditionally independent distribution. This operation, effectively a small perturbation in the sample space, does not change the Wasserstein distance much, but can have a large effect on the TV distance or KL divergence. 

Lemma \ref{Wasserstein:dense:lemma} suggests, but does not imply that CI testing is ``hard''. Building on the construction of Lemma \ref{Wasserstein:dense:lemma}, we give a new simpler proof of the ``no-free-lunch'' theorem of Shah and Peters \citep[see Theorem 2]{shah2018hardness}. For convenience of the reader we restate the no-free-lunch theorem below, and  give a complete proof in Appendix~\ref{hardness:appendix}. Let $d = d_X + d_Y + d_Z$, and suppose that we observe $n$ observations $\cD_n = \{(X_1,Y_1, Z_1), \ldots, (X_n, Y_n, Z_n)\}$. 

\begin{theorem}[No-Free-Lunch]\label{no-free-lunch:thm} Given any $n \in \NN$, $\alpha \in (0,1)$, $M \in (0, \infty]$ and a potentially randomized test $\psi_n :\RR^{nd}\times[0,1] \mapsto \{0,1\}$, that has valid level $\alpha$ for the null hypothesis $\cP_{0,M}$, we have that $\PP_Q(\psi_n = 1) \leq \alpha $ for all $Q \in \cQ_{0,M}$. 
\end{theorem}

As stated, Theorem \ref{no-free-lunch:thm} assumes that $(X,Y,Z)$ have a distribution which is continuous with respect to the Lebesgue measure.  Suppose that $\ell_1$ and $\ell_2$ are two fixed and finite integers. We assume that $X$ and $Y$ are supported on $[\ell_1]$ and $[\ell_2]$ respectively, and that $Z$ is supported on $[-M,M]^{d_Z}$ and has a continuous density with respect to the Lebesgue measure. The generating mechanism of the triple $(X,Y,Z)$ is as follows: first a $Z$ from the distribution $P_Z$ is generated. Next, $X$ and $Y$ are generated from the distribution $P_{X,Y|Z}$ which is supported on $[\ell_1] \times [\ell_2]$ for (almost) all $Z$. Denote the set of all such distributions with $\cE_{0,M}'$ (where we omit the dependence of $\cE_{0,M}'$ on $d_Z, \ell_1, \ell_2$ for simplicity). Let $\cP_{0,M}'\subset \cE_{0,M}'$ be the subset of $\cE_{0,M}'$ consisting of distributions such that $X \independent Y | Z$ and $\cQ_{0,M}' = \cE_{0,M}' \setminus \cP_{0,M}'$. Again as before we assume that we observe $n$ observations $\cD_n$. We have the following simple Corollary to Theorem \ref{no-free-lunch:thm} which was alluded to by Shah and Peters \cite{shah2018hardness}. 
\begin{corollary}[Discrete No-Free-Lunch]\label{counting:meas:XY:cont:Z} Given $n \in \NN$, $\alpha \in (0,1)$, $M \in (0, \infty)$ and a potentially randomized test $\psi_n$, that has a valid level $\alpha$ for the null hypothesis $\cP_{0,M}'$, we have that $\PP_Q(\psi_n = 1) \leq \alpha$ for all $Q \in \cQ_{0,M}'$. 
\end{corollary}
Intuitively, Corollary \ref{counting:meas:XY:cont:Z} reveals that it is the continuity of $Z$ that makes CI testing ``hard'', and not the continuity of $X$ and $Y$.

\section{Minimax Lower Bounds}\label{minimax:lower:bounds:section}

In this section we present our minimax lower bounds on the critical radius for conditional independence testing in various settings. 
Our first main result (Theorem~\ref{first:lower:bound}) develops a lower bound on the critical radius 
in the case when $X$ and $Y$ are discrete and $Z$ is continuous. Our next main result (Theorem~\ref{lower:bound:continuous:case}) develops an analogous bound for the setting when $X$, $Y$ and $Z$ all have continuous distributions.


\subsection{$X$ and $Y$ Discrete, $Z$ Continuous Case}\label{discrete:case:lower:bound:section}

We begin by recalling the Lipschitzness classes $\cP_{0,[0,1], \TV}'(L), \cP_{0,[0,1], \TV^2}'(L)$ and $\cP_{0,[0,1], \chi^2}'(L)$ introduced in Definition~\ref{def:nullsmoothness}, 
and $\cQ_{0,[0,1], \TV}'(L)$ introduced in Definition~\ref{alternative:TV:smoothness:def}.
In this section we develop a lower bound on the critical radius for distinguishing the conditionally independent distributions in any one of the null classes $\cP_{0,[0,1], \TV}'(L), \cP_{0,[0,1], \TV^2}'(L)$ and $\cP_{0,[0,1], \chi^2}'(L)$
from the alternative class of conditionally dependent distributions $\cQ_{0,[0,1], \TV}'(L)$. Formally, we have the following result:

\begin{theorem}[Critical Radius Lower Bound]\label{first:lower:bound} Let $\overline \cH_0 = \cP_{0,[0,1]}'$. Suppose that $\cH_0$ is either of $\cP_{0,[0,1], \TV}'(L) $, $\cP_{0,[0,1], \TV^2}'(L) $ or $\cP_{0,[0,1], \chi^2}'(L)$, while $\cH_1 = \cQ_{0,[0,1], \TV}'(L)$ for some fixed $L \in \RR^+$. Then for some absolute constant $c_0 > 0$ the critical radius defined in \eqref{critical:radius} is bounded as
\begin{align*}
\varepsilon_n(\cH_0,\overline \cH_0, \cH_1) \geq c_0 \bigg(\frac{(\ell_1 \ell_2)^{1/5}}{n^{2/5}} \wedge 1\bigg).
\end{align*}
\end{theorem}
\noindent {\bf Remarks: } 
\begin{itemize}
\item In the case when $\ell_1$ and $\ell_2$ are constant, our lower bound on the critical 
radius \begin{align*}
\varepsilon_n(\cH_0,\overline \cH_0, \cH_1) \geq c_0 n^{-2/5},
\end{align*} scales as the familiar rate for 
goodness-of-fit testing in the nonparametric setting of $\varepsilon_n \asymp n^{-2s/(4s + d)}$ \cite{ery2018remember, balakrishnan2017hypothesis,ingster2003nonparametric} (where in our setting we take $d = 1$ and $s = 1$, corresponding to the one-dimensional Lipschitz smooth component $Z$). 

We note that as is typical in hypothesis testing problems this rate is faster than the $n^{-1/3}$ rate that we would expect for estimating a univariate Lipschitz smooth density, highlighting the fact that in many cases, from a statistical perspective, hypothesis testing is easier than estimation.

\item On the other hand the scaling of $\varepsilon_n$ with $\ell_1$ and $\ell_2$ has a typical square-root dependence seen in \emph{parametric} hypothesis testing problems \cite{valiant2017automatic,ingster2003nonparametric}, where roughly we see that the critical radius shrinks provided that $\sqrt{\ell_1 \ell_2}/n \rightarrow 0$. Once again this is in contrast to the linear dependence we would expect in estimating a multinomial distribution on $\ell_1 \times \ell_2$ categories, which would require
$\ell_1 \ell_2/n \rightarrow 0$ for consistent estimation.

Thus we see that the lower bound we obtain for CI testing in the setting where $X$ and $Y$ are discrete, and $Z$ is continuous blends parametric and nonparametric hypothesis testing rates. In Section~\ref{upper:bound:section} we develop matching upper bounds in various settings. 
\item We note in passing that our lower bound applies when the null distribution is restricted to belong to any of the three Lipschitzness classes introduced in Definition~\ref{def:nullsmoothness}.
\item We give the proof of Theorem \ref{first:lower:bound} in Appendix \ref{app:lb}. We note that at a high-level we follow the strategy of \citet{ingster1982minimax} of creating a carefully chosen 
collection of possible densities under the alternative, and lower bounding the performance
of the (optimal) likelihood ratio test in distinguishing a fixed null distribution against a uniform mixture of the selected distributions under the alternative. However, in our setting additional care is needed when perturbing the $X$ and $Y$ components in order to ensure that they remain valid discrete distributions (see Figure~\ref{tilde:Delta:figure}, and the associated construction), and to characterize the distance of
our perturbed distributions from the manifold of conditionally independent distributions.
\end{itemize}

\subsection{$X$, $Y$ and $Z$ Continuous Case}\label{continuous:case:lower:bound:section}

We first recall 
the Lipschitzness classes $\cP_{0,[0,1]^3, \TV}(L)$ and $\cP_{0,[0,1]^3, \chi^2}(L)$ introduced in Definition~\ref{def:nullsmoothness}, 
and $\cQ_{0,[0,1]^3, \TV}(L,s)$ introduced in Definition~\ref{alternative:smoothness:holder}.
We derive a lower bound on the critical radius for distinguishing the conditionally independent distributions in either of the null classes $\cP_{0,[0,1]^3, \TV}(L)$ and $\cP_{0,[0,1]^3, \chi^2}(L)$
from the alternative class of conditionally dependent distributions $\cQ_{0,[0,1]^3, \TV}(L,s)$. Formally, we have the following result:

 \begin{theorem}[Critical Radius Lower Bound]\label{lower:bound:continuous:case} Let $\overline \cH_0 = \cP_{0,[0,1]^3}$. Suppose that $\cH_0$ is either $\cP_{0,[0,1]^3, \TV}(L)$ or $\cP_{0,[0,1]^3, \chi^2}(L)$, and $\cH_1 = \cQ_{0,[0,1]^3, \TV}(L,s)$ for some fixed $L \in \RR^+$. Then we have that for some absolute constant $c_0 > 0$,
 \begin{align*}
\varepsilon_n(\cH_0,\overline \cH_0, \cH_1) \geq \frac{c_0}{n^{2s/(5s + 2)}}.
\end{align*}
 \end{theorem}

\noindent {\bf Remark: } 
\begin{itemize}

\item We note that our lower bound applies when the null distribution is restricted to belong to either of the classes $\cP_{0,[0,1]^3, \TV}(L)$ and $\cP_{0,[0,1]^3, \chi^2}(L)$.
Our proof in this setting builds on that of Theorem~\ref{first:lower:bound}. In this case, to create a collection of distributions under the alternative 
we perturb the null distribution by smooth, infinitely differentiable bumps along all three coordinates
in a carefully constructed fashion. By an appropriate choice of 
various parameters, we ensure that the distributions we construct satisfy the Lipschitzness and H\"{o}lder smoothness conditions required by the class $\cQ_{0,[0,1]^3, \TV}(L,s)$, while still remaining
sufficiently far from the conditional independence manifold. We provide the details of our construction, as well as the subsequent analysis of the likelihood ratio test in Appendix~\ref{app:lb}.
\end{itemize}

\section{Minimax Upper Bounds}\label{upper:bound:section}
In this section we provide matching (in certain regimes) upper bounds to the lower bounds given in Section \ref{minimax:lower:bounds:section}. 

\subsection{Upper Bound with Finite Discrete $X$ and $Y$}\label{fixed:l1:l2:section}

In this section we will suggest a conditional independence test to match the lower bound of Section \ref{discrete:case:lower:bound:section} when $\ell_1, \ell_2 = O(1)$ are not allowed to scale with $n$. In this case the bound of Theorem \ref{first:lower:bound} simply states that the critical radius is bounded from below by $c n^{-2/5}$, for some sufficiently small constant $c > 0$. To start the preparation for our test statistic we will first re-introduce certain unbiased estimators from the work of \citet{canonne2018testing}. Our exposition and treatment of their estimators is novel, and builds on classical work on U-statistics \cite{hoeffding1948class,Serfling2001Approximation}.

Suppose we observe $\sigma \geq 4$ observations of two discrete covariates $X'$ and $Y'$ taking values in $[\ell_1]$ and $[\ell_2]$\footnote{As in the lower bound, it is not crucial that the supports of $X'$ and $Y'$ are $[\ell_1]$ and $[\ell_2]$. We focus on this case simply for the sake of clarity.}. Denote the joint distribution of $(X',Y')$ by $p_{X',Y'}$. As usual we denote the marginals as $p_{X'}$ and $p_{Y'}$ (i.e., $p_{X'}(x) = \sum_{y \in [\ell_2]} p_{X',Y'}(x,y)$ and similarly for $p_{Y'}$). We are interested in finding an unbiased estimate of the following expression
\begin{align}\label{l2:distance:p:prod:marg}
\|p_{X',Y'} - p_{X'} p_{Y'}\|_2^2 = \sum_{x \in [\ell_1], y\in[\ell_2]} (p_{X',Y'}(x,y) - p_{X'}(x)p_{Y'}(y))^2. 
\end{align}
The above expression is nothing but the $L^2_2$ distance between $p_{X',Y'}$ and the product of the marginals $p_{X'} p_{Y'}$. In order for us to unbiasedly estimate this quantity we will use a U-statistic, and at least $4$ observations. Before we define the U-statistic, let us define its kernel. Let $i, j \in [\sigma]$ be two observations. Define 
\begin{align}\label{phi:def}
\phi_{ij}(xy) & =  \mathbbm{1}(X'_i = x, Y'_i = y) - \mathbbm{1}(X'_i = x) \mathbbm{1}(Y'_j = y).
\end{align}
Next, take $4$ distinct observations $i,j,k,l \in [\sigma]$, and define the kernel function
\begin{align*}
h_{ijkl} =  \frac{1}{4!} \sum_{\pi \in [4!]} \sum_{x \in [\ell_1], y \in [\ell_2]} \phi_{\pi_1\pi_2}(xy)\phi_{\pi_3 \pi_4}(xy),
\end{align*}
where $\pi$ is a permutation of $i,j,k,l$. Clearly since $i,j,k,l \in [\sigma]$ are distinct, the above is an unbiased estimate of \eqref{l2:distance:p:prod:marg}. Next, we construct the U-statistic
\begin{align}\label{U:stat}
U(\cD) := \frac{1}{{\sigma \choose 4}} \sum_{i < j < k < l : (i,j,k,l) \in  [\sigma]} h_{ijkl},
\end{align}
where we denoted $\cD = \{(X_1',Y_1'), \ldots, (X'_\sigma, Y'_{\sigma}) \}$. The U-statistic \eqref{U:stat} is an unbiased estimate of the $L_2^2$ distance in~\eqref{l2:distance:p:prod:marg}. It is not obvious that this estimator is the same as the one defined in equation (18) of \citet{canonne2018testing}. However, using Proposition 4.2 of \cite{canonne2018testing} and the fact that the U-statistic in~\eqref{U:stat} is a symmetric estimator, one can deduce that the two estimators must coincide. 

In order to analyze our hypothesis test we will appropriately bound the mean and variance of our test statistic under the null and under the alternative. Since our test is based on the U-statistic in~\eqref{U:stat} we will need to bound its variance. In principle, one can directly reuse the bound on the variance of the U-statistic in~\eqref{U:stat} given in~\cite{canonne2018testing}. Since the original derivation of this bound is complicated, we give a novel derivation starting from first principles, building on the extensive theory for U-statistics. We have the following result:
\begin{lemma}[Variance Upper Bound]\label{U:stat:variance:lemma}
There exists some absolute constant $C$ such that 
\begin{align*}
\MoveEqLeft \Var[ U(\cD)] \leq C\bigg(\frac{\EE[U(\cD)]\max(\|p_{X',Y'}\|_2, \|p_{X'}p_{Y'}\|_2)}{\sigma} \\
& + \frac{\max(\|p_{X',Y'}\|^2_2, \|p_{X'}p_{Y'}\|^2_2)}{\sigma^2}\bigg).
\end{align*}
\end{lemma}
\noindent Now that we have defined the statistic $U$ and have bounded its variance, we are ready to introduce our test statistic. Before that we include a randomization device in the test:

Draw $N \sim Poi(\frac{n}{2})$. If $N > n$ accept the null hypothesis. If $N \leq n$ take arbitrary $N$ out of the $n$ samples and work with them. The next step is to discretize the variable $Z$ into $d$ bins of equal size. Denote those bins with $\{C_1, \ldots, C_d\}$, so that $\cup_{i \in [d]} C_i = [0,1]$, and each $C_i$ is an interval of length $\frac{1}{d}$. Next construct the datasets $\cD_{m} = \{(X_i, Y_i) : Z_i \in C_m, i \in [N]\}$. Let $\sigma_m = |\cD_m|$ be the sample size in each set $\cD_m$, so that $\sum_{m \in [d]} \sigma_m = N$. For bins $\cD_m$ with at least $\sigma_m \geq 4$ observations, let for brevity $U_m = U(\cD_m)$. Each $U_m$ can be thought of as a local test of independence within the bin $C_m$ --- if the value of $U_m$ is close to $0$ then intuitively independence holds within that bin, while if the value of $U_m$ is large, independence is potentially violated within that bin. In order to combine these different statistics we follow \citet{canonne2018testing} and consider the following test statistic
\begin{align}
\label{eqn:fixedstat}
T = \sum_{m \in [d]} \mathbbm{1}(\sigma_m \geq 4) \sigma_m U_m. 
\end{align}
We will prove that under the null hypothesis the value of $T$ is likely to be below a threshold $\tau$ (to be specified), while under the alternative hypothesis $T$ will likely exceed the value $\tau$. Define the test 
$$\psi_\tau(\cD_N) = \mathbbm{1}(T \geq \tau),$$ 
where $\cD_N = \{(X_1,Y_1,Z_1), \ldots, (X_N,Y_N,Z_N)\}$. Recall the definitions of the null Lipschitzness classes 
$\cP_{0,[0,1],\TV}'(L), \cP_{0,[0,1],\chi^2}'(L), \cP_{0,[0,1],\TV^2}'(L)$ and the alternative Lipschitzness classes $\cQ_{0,[0,1],\TV}'(L)$ (see Definitions \ref{def:nullsmoothness} and~\ref{alternative:TV:smoothness:def} in Section \ref{sec:smoothness}). We are now ready to state the main result of this section.
\begin{theorem}[Finite Discrete $X$, $Y$ Upper Bound] \label{main:theorem:finite:discrete:XY} Set $d = \lceil n^{2/5} \rceil$ and let $\tau = \zeta n^{1/5}$ for a sufficiently large absolute constant $\zeta$ (depending on $L$). 
Finally, suppose that $\varepsilon \geq c n^{-2/5}$, for a sufficiently large constant $c$ (depending on $\zeta$, $L$, $\ell_1,\ell_2$).
Then we have that 
\begin{align*}
\sup_{p \in \cP_{0,[0,1],\TV^2}'(L) \cup \cP_{0,[0,1],\TV}'(L) \cup \cP_{0,[0,1],\chi^2}'(L)} \EE_p[\psi_\tau(\cD_N)] & \leq \frac{1}{10},\\ 
\sup_{p \in \{p \in \cQ_{0,[0,1],\TV}'(L): \inf_{q \in \cP'_{0,[0,1]}} \|p - q\|_1 \geq \varepsilon\}} \EE_p[ 1- \psi_\tau(\cD_N)] & \leq \frac{1}{10} + \exp(-n/8).
\end{align*}
\end{theorem}

\noindent {\bf Remarks: } 
\begin{itemize}
\item In the above theorem the constants $\frac{1}{10}$ are arbitrary and can be made smaller (or larger) by appropriately adjusting the constants $\zeta$ and $c$. In the case when $\ell_1$ and $\ell_2$ are of constant order the above test is optimal, in the sense that the critical radius rate $n^{-2/5}$ matches the lower bound given in Theorem \ref{first:lower:bound}. 
\item When $\ell_1$ and $\ell_2$ are allowed to scale with $n$ the test no longer results in the correct order for the critical radius (in particular, we can no longer treat the quantity $c$ as a constant and its dependence on $\ell_1$ and $\ell_2$ is not optimal). 
In the next section we provide more sophisticated test which is capable of matching the bound proved in Theorem \ref{first:lower:bound} for some regimes of $\ell_1$ and $\ell_2$. 
\item In order to show that our test has high power for sufficiently large $\varepsilon_n$, we follow a classical strategy of upper bounding the variance of our test statistic under the null and alternative, upper bounding its expectation under null, and lower bounding its expectation under the alternative. These bounds together with a careful choice of the threshold $\tau$, and an application of Chebyshev's inequality, are used to characterize the power of our proposed test.
We detail these calculations in Appendix~\ref{app:upper}.

\item A recurring complication, one that we need to address in the analysis of our tests in both the discrete  and continuous $X,Y$ setting is that our test statistic does not have expectation zero under the null. 
This is in sharp contrast to typical tests for goodness-of-fit and two-sample tests (for instance those analyzed in \cite{balakrishnan2017hypothesis,balakrishnan2018hypothesis,diakonikolas2016new,valiant2017automatic,ery2018remember}).
In more detail, under the null,
the binning operation used to discretize the $Z$ variable, moves us off the manifold of conditionally independent distributions (i.e. the discretized distribution need not satisfy conditional independence even if the original distribution does).

Exploiting the Lipschitzness assumptions in Definition~\ref{def:nullsmoothness}, we can argue that under the null, for sufficiently small bins, we do not move too far from the collection of conditionally independent distributions (say in the total variation sense). A naive reduction would yield an imprecise null hypothesis testing problem of attempting to distinguish distributions which are near-conditionally independent from those which are relatively far from conditionally independent. This imprecise null testing problem is however statistically challenging \cite{valiant2017estimating}, and this naive reduction fails to yield the optimal rates described in our upper bounds.

Instead, avoiding this indirect reduction, we take a more direct approach of uniformly upper bounding the expectation of our test statistic under the null. By directly using the Lipschitzness assumptions, and the factorization
structure of distributions under the null, we are able to obtain tighter bounds on the expected value of our test statistic under the null. This in turn yields near-optimal upper bounds on the critical radius.
\end{itemize}

\subsection{Upper Bound with Scaling Discrete $X$ and $Y$}\label{scaling:l1:l2:section}

In this section we present a more sophisticated test procedure which is capable of matching the bound of Theorem \ref{first:lower:bound} for some regimes of the sizes of the supports of $X$ and $Y$ --- $\ell_1$ and $\ell_2$. In contrast to the previous section, we now no longer assume that $\ell_1, \ell_2 = O(1)$. 
We note that throughout this section, without any loss of generality, we focus on the case when $\sqrt{\ell_1 \ell_2}/n \lesssim 1.$ When this condition is not satisfied, the lower bound in Theorem \ref{first:lower:bound} shows that the critical radius must be at least a constant, and in this regime upper bounds are trivial. Since we only characterize the critical radius up to constants, when we choose the separation between the null and alternate 
$\varepsilon$ to be a sufficiently large constant (say 2), there are no longer any distributions in the alternate, and the CI testing problem is trivial.

The key idea of this section is to use a weighted U-statistic in place of the (unweighted) U-statistic from Section \ref{fixed:l1:l2:section}. This weighting is sometimes referred to as ``flattening'' see, e.g., \cite{diakonikolas2016new, canonne2018testing}. A careful choice of the weighting yields a U-statistic with smaller variance (see Lemma~\ref{second:variance:upper:bound}), and the resulting test has higher power. 

To describe the weighting consider again the same scenario as in Section \ref{fixed:l1:l2:section}. Suppose we observe $\sigma \geq 4$ samples of two discrete covariates $(X', Y')$ supported on $[\ell_1]\times [\ell_2]$. Let $\cD = \{(X_1',Y_1'), \ldots, (X_\sigma', Y_\sigma')\}$ and $p_{X',Y'}$ be the distribution of $(X',Y')$. By losing at most three samples we may assume that $\sigma = 4 + 4t$ for some $t \in \NN$. Define $t_1 := \min(t, \ell_1)$ and $t_2 := \min(t, \ell_2)$. Next we split $\cD$ into three datasets of sizes $t_1$, $t_2$ and $2t + 4$ respectively: $\cD_{X'} = \{X'_i: i \in [t_1]\}$, $\cD_{Y'} = \{Y'_i : t_1 + 1 \leq i\leq t_1 + t_2 \}$, and $\cD_{X',Y'} = \{(X_i', Y'_i): 2t + 1 \leq i \leq \sigma\}$. The idea behind defining those three datasets is that the first two datasets --- $\cD_{X'}$ and $\cD_{Y'}$, will be used to calculate weights, while the last dataset $\cD_{X',Y'}$, which has at least $4$ observations, will be used to calculate the U-statistic. Construct the integers
\begin{align*}
1 + a_{xy} = (1 + a_x)(1 + a'_y),
\end{align*}
where $a_x$ are the number of occurrences of $x$ in $\cD_{X'}$ and $a'_y$ is the number of occurrences of $y$ in~$\cD_{Y'}$. 

Next, take $4$ distinct observations indexed by $i,j,k,l$ from the dataset $\cD_{X',Y'}$, and define the (weighted) kernel function
\begin{align*}
h^{\ba}_{ijkl} =  \frac{1}{4!} \sum_{\pi \in [4!]} \sum_{x \in [\ell_1], y \in [\ell_2]} \frac{\phi_{\pi_1\pi_2}(xy)\phi_{\pi_3 \pi_4}(xy)}{1 + a_{xy}},
\end{align*}
where $\pi$ is a permutation of $i,j,k,l$ and recall the definition of $\phi_{ij}(xy)$ \eqref{phi:def}. Here the super-indexing with $\ba$ of $h^{\ba}_{ijkl}$, indicates that the statistic is weighted by the numbers $1 + a_{xy}$ for $x \in [\ell_1], y\in [\ell_2]$. Notice that the idea of this weighting is similar to the weighting in a Pearson's $\chi^2$ test of independence. Indeed the quantity $a_{xy}$ is in expectation proportional to the product $p_{X'}(x)p_{Y'}(y)$. On the other hand, the expression $\phi_{\pi_1\pi_2}(xy)\phi_{\pi_3\pi_4}(xy)$ is unbiased for $(p_{X',Y'}(x,y) - p_{X'}(x)p_{Y'}(y))^2$. Next, to reduce the variance of $h_{ijkl}^{\ba}$, we construct the (weighted) U-statistic
\begin{align}\label{weighted:U:stat}
U_{W}(\cD) := \frac{1}{{2t + 4 \choose 4}} \sum_{i < j < k < l : (i,j,k,l) \in \cD_{X',Y'}} h^{\ba}_{ijkl},
\end{align}
where we abused notation slightly for $(i,j,k,l) \in \cD_{X',Y'}$ to mean taking four observations from the dataset $\cD_{X',Y'}$. 
For convenience of the notation we now give a definition from \cite{diakonikolas2016new}. 
\begin{definition}[Split Distribution]\label{split:distribution:main:text} Given a discrete distribution $p$ over $[d_1]\times[d_2]$ and a multi-set $S$ of elements of $[d_1]\times[d_2]$ we now define the split distribution $p_S$. Let $b_{xy} = \sum_{(x',y') \in S}\mathbbm{1}((x,y) = (x',y'))$. Thus $\sum_{(x,y) \in [d_1]\times[d_2]} 1 + b_{xy}  = d_1d_2 + |S|$. Define the set $B_S = \{(x,y,i) | (x,y) \in [d_1]\times[d_2], 1 \leq i \leq 1 + b_{xy}\}$. The split distribution $p_S$ is supported on $B_S$ and is obtained by sampling $(x,y)$ from $p$ and $i$ uniformly from the set $[1 + b_{xy}]$. 
\end{definition}
Given $S$ and $b_{xy}$ as in Definition \ref{split:distribution:main:text}, for any two discrete distributions $p$ and $q$ over $[d_1]\times[d_2]$ it follows that 
\begin{align*}
\|p_S - q_S\|_2^2  = \sum_{(x,y) \in [d_1]\times[d_2]} \frac{(p(x,y)- q(x,y))^2}{1 + b_{xy}}.
\end{align*}
Similarly for the split distribution $p_S$ we have that 
\begin{align*}
\|p_S\|_2^2 = \sum_{(x,y)\in [d_1]\times[d_2]} \frac{p^2(x,y)}{1 + b_{xy}}.
\end{align*}
Construct a multi-set $A$ by adding $a_{xy}$ occurrences of the pair $(x,y)$ to $A$. Using this notation it now follows that
\begin{align*}
\MoveEqLeft \EE[U_{W}(\cD) | \cD_{X'}, \cD_{Y'}] = \|p_{X',Y',A} - p^\Pi_{X',Y',A}\|_2^2 \\
& = \sum_{(x,y) \in [\ell_1]\times[\ell_2]} \frac{(p_{X',Y'}(x,y) - p_{X'}(x)p_{Y'}(y))^2}{ 1 + a_{xy}},
\end{align*}
where $p_{X',Y',A}$ is the $A$-split distribution $p_{X',Y'}$, and $p^\Pi_{X',Y',A}$ is the $A$-split distribution $p^\Pi_{X',Y'}$ where $p^\Pi_{X',Y'} = p_{X'}p_{Y'}$.
We will now show an analogous variance bound to the one in Lemma \ref{U:stat:variance:lemma}. We have the following
\begin{lemma}[Variance Upper Bound]\label{second:variance:upper:bound} For some absolute constant $C$, the following holds
\begin{align*}
\MoveEqLeft \Var[U_{W}(\cD) | \cD_{X'}, \cD_{Y'}] \leq  C\bigg(\frac{ \EE[U_{W}(\cD) | \cD_{X'}, \cD_{Y'}] \|p^\Pi_{X',Y',A}\|_2}{\sigma} \\
&+ \frac{\EE[U_{W}(\cD) | \cD_{X'}, \cD_{Y'}]^{3/2}}{\sigma} + \frac{ \|p^\Pi_{X',Y',A}\|^2_2}{\sigma^2} +  \frac{\EE[U_{W}(\cD) | \cD_{X'}, \cD_{Y'}] }{\sigma^2}\bigg).\nonumber
\end{align*}
\end{lemma}
\noindent In comparing to the result of Lemma~\ref{U:stat:variance:lemma} we see roughly that the variance bound now depends on the (typically much smaller) 
$L_2$-norm of the flattened or split distribution $p^\Pi_{X',Y',A}$, instead of the $L_2$ norm of the original distribution $p_{X',Y'}$. As emphasized in \cite{diakonikolas2016new,canonne2018testing}
this variance reduction achieved through flattening is critical for designing minimax optimal tests (particularly when $\ell_1$ and $\ell_2$ are allowed to grow with the sample-size $n$).

Now we are ready to define our test statistic. As before, the first step is to draw a random sample size $N \sim Poi(\frac{n}{2})$ and take $N$ subsamples of the $n$ observations, with the convention that if $N > n$ we accept the null hypothesis. Next, bin the support of the variable $Z$ into $d$ bins of equal size. Denote those bins with $\{C_1, \ldots, C_d\}$, so that $\cup_{i\in[d]} C_i = [0,1]$, and each $C_i$ is an interval of length $\frac{1}{d}$. Construct the datasets $\cD_{m} = \{(X_i, Y_i) : Z_i \in C_m, i \in [N]\}$. Let $\sigma_m = |\cD_m|$ be the sample size in each set $\cD_m$, so that $\sum_{m \in [d]} \sigma_m = N$. Recall that each set $\cD_m$ will be further separated into three sets $\cD_{m,X}$, $\cD_{m,Y}$ and $\cD_{m,X,Y}$, the first two of which are used for calculating weights, while the last one is used for the calculation of the weighted U-statistic.  For bins $\cD_m$ with at least $\sigma_m \geq 4$ observations, let for brevity $U_m = U_W(\cD_m)$. We now combine these different independence testing statistics into one CI testing statistic as follows. Let 
\begin{align}\label{T:stat:def}
T = \sum_{m \in [d]} \mathbbm{1}(\sigma_m \geq 4) \sigma_m \omega_m U_m,
\end{align}
where $\omega_m = \sqrt{\min(\sigma_m, \ell_1)\min(\sigma_m, \ell_2)}$ is a weighting factor, which further weights the statistics $U_m$. The presence of $\omega_m$ is necessitated by the weighting of the U-statistic \eqref{weighted:U:stat}. In order to show that the test based on the statistic $T$ has high power (and low type 1 error) 
we will prove that under the null hypothesis the value of $T$ is likely to be below a threshold $\tau$ (to be specified), while under the alternative hypothesis $T$ will likely exceed the value $\tau$. Define the test 
\begin{align}\label{general:discretization:test}
\psi_\tau(\cD_N) = \mathbbm{1}(T \geq \tau),
\end{align}
where $\cD_N = \{(X_1,Y_1,Z_1), \ldots, (X_N,Y_N,Z_N)\}$. We have the following result
\begin{theorem}[Scaling Discrete $X$, $Y$ Upper Bound]\label{main:theorem:scaling:discrete:XY} Set $d = \lceil \frac{n^{2/5}}{(\ell_1 \ell_2)^{1/5}}\rceil$ and set the threshold $\tau = \sqrt{\zeta d}$ for a sufficiently large absolute constant $\zeta$ (depending on $L$). Suppose that $\ell_1 \geq \ell_2$ satisfy the condition that 
$d \ell_1 \lesssim n$.
Then 
when $\varepsilon \geq c \frac{(\ell_1 \ell_2)^{1/5}}{n^{2/5}}$, for a sufficiently large absolute constant $c$ (depending on $\zeta$, $L$), 
we have that 
\begin{align*}
\sup_{p \in \cP_{0,[0,1],\chi^2}'(L)}\EE_p[\psi_\tau(\cD_k)] & \leq \frac{1}{10},\\ 
\sup_{p \in \{p \in \cQ_{0,[0,1],\TV}'(L): \inf_{q \in \cP'_{0,[0,1]}} \|p - q\|_1 \geq \varepsilon\}}\EE_p[1 - \psi_\tau(\cD_k)]& \leq \frac{1}{10} + \exp(-n/8).
\end{align*}
\end{theorem}
\noindent {\bf Remarks: } 
\begin{itemize} 
\item Some remarks regarding this result are in order. First, when $d \ell_1 \lesssim n$, the bound on the critical radius 
we obtain matches the information-theoretic limit derived in Theorem \ref{first:lower:bound}. 
An important special case (that we will use in our tests in the continuous $X$ and $Y$ setting) when this condition is automatically implied is when $\ell_1 \asymp \ell_2$.
To see this, observe that when $\ell_1 \asymp \ell_2$ we have that $d \ell_1 \lesssim n$ is equivalent to $\big(\frac{n}{\ell_1}\big)^{2/5} \lesssim \frac{n}{\ell_1}$ (for our choice of $d$) which is implied by the condition that $\frac{\ell_1}{n} \lesssim 1$. When this latter condition is not satisfied the lower bound on the critical radius in Theorem \ref{first:lower:bound} is a universal constant (and the upper bound is trivial).

We also note in passing that for our choice of $d$, the condition that $d \ell_1 \lesssim n$ is equivalent to the condition that $\frac{\ell_1^{4}}{\ell_2} \lesssim n^{3}$, which yields the claim in Section~\ref{sec:summary} that
our test is minimax optimal when $\frac{\ell_1^{4}}{\ell_2} \lesssim n^{3}$.

\item In contrast to Theorem \ref{main:theorem:finite:discrete:XY}, here we choose the 
null set of distributions as $\cP'_{0,[0,1], \chi^2}(L)$. As we discussed following Theorem~\ref{main:theorem:finite:discrete:XY}, one of the key difficulties is to 
characterize the effect of discretization of the $Z$ variable, in order to
upper bound the expectation of our test statistic under the null, over the appropriate Lipschitzness class. When $\ell_1$ and $\ell_2$ are allowed to scale, 
we show an upper bound on this expectation in terms of the $\chi^2$-divergence between the discretized null distribution and
the product of its marginals (see equation~\eqref{intermediate:bound:null:hypothesis:lemma} in Appendix~\ref{app:upper}). We in turn 
show that this discretization
error due to binning is appropriately small when the null distribution satisfies the $\chi^2$ Lipschitzness condition, i.e. belongs to $\cP_{0,[0,1],\chi^2}'(L)$.

\item As we detail further in Appendix~\ref{app:upper}, when the condition that $d \ell_1 \lesssim n$ is not satisfied we still provide upper bounds on the critical radius but 
these upper bounds do not match the lower bound in Theorem~\ref{first:lower:bound}. As we discuss further in Section~\ref{discussion:section} we believe that sharpening
either the lower or upper bound is challenging, requiring substantially different ideas, and we defer this to future work.

\item From a technical standpoint, analyzing the power of the test statistic in~\eqref{T:stat:def} is substantially more involved than the analysis of its fixed $\ell_1, \ell_2$ counterpart in~\eqref{eqn:fixedstat}.
Several complications are introduced in ensuring that the flattening weights (the terms $a_{xy}$ in the definition of our U-statistic in~\eqref{weighted:U:stat}) are well-behaved. In a classical fixed dimensional
setup (where $\ell_1,\ell_2$ and the number of bins $d$ are all held fixed) 
it would be relatively straightforward to argue that the flattening weights concentrate tightly around their expected values. In the high-dimensional setting that 
we consider these weights can have high variance and substantial work is needed to tightly bound the mean and variance of our test statistic.

This also highlights an important difference from the goodness-of-fit problem considered in~\cite{balakrishnan2017hypothesis,valiant2017automatic,blais2019distribution}. In the goodness-of-fit problem, where we test fit of the data to a \emph{known} 
distribution $p_0$ the corresponding weights in 
the Pearson $\chi^2$ statistic are fixed
and known to the statistician. In conditional independence testing these weights are estimated from data.


\end{itemize}

\subsection{Upper Bound in the Continuous Case}\label{cont:case:upper:bounds:section}

In this section consider testing for CI when $(X,Y,Z)$ are supported on $[0,1]^3$ and have a distribution which is absolutely continuously with respect to the Lebesgue measure. In view of the notation in Section \ref{continuous:case:lower:bound:section} this is equivalent to assuming that $p_{X,Y,Z} \in \cE_{0,[0,1]^3}$. We begin our discussion with formally describing the test.

The testing strategy is related to the test described in Section \ref{scaling:l1:l2:section}. First draw $N\sim Poi(\frac{n}{2})$, and take arbitrary $N$ out of the $n$ observations in the case when $N \leq n$, and accept the null hypothesis if $N > n$. Next, we bin the support $[0,1]$ with bins  $\{C_1, C_2, \ldots, C_d\}$, where the sizes of those bins are equal and $\cup_{i \in [d]} C_i = [0,1]$. These bins will be used to discretize $Z$. In addition we create a second rougher (in the case $s \geq 1$) partition of $[0,1]$ into $d' := \lceil d^{1/s} \rceil$ intervals $\cup_{i \in [d']} C'_i = [0,1]$. These second bins will be used to discretize $X$ and $Y$. Specifically, we use these two sets of bins to discretize the observations $\cD_N = \{(X_i, Y_i, Z_i)\}_{i \in [N]}$ as follows. First define the discretization function $g:[0,1] \mapsto [d']$ by $g(x) = j$ iff $x \in C'_j$. Next consider the set of observations $\cD_N' = \{(g(X_i), g(Y_j), Z_i)\}_{i \in [N]}$. We can now use the test defined in \eqref{general:discretization:test}: $\psi_\tau(\cD'_N)$ with an appropriately selected threshold $\tau$ and the bins $\{C_1, C_2, \ldots, C_d\}$ to discretize $Z$ with in order to test for CI. We have the following result.

\begin{theorem}[Continuous $X,Y,Z$ Upper Bound]\label{continuous:case:upper:bound:theorem} Set $d = \lceil n^{2s/(5s + 2)} \rceil$ and set the threshold $\tau = \sqrt{\zeta d}$ for a sufficiently large $\zeta$ (depending on $L$). Let $\cH_0(s) = \cP_{0,[0,1]^3,\TV}(L) \cup \cP_{0,[0,1]^3,\chi^2}(L)$ when $s \geq 1$ and $\cH_0(s) = \cP_{0,[0,1]^3,\chi^2}(L)$ when $s < 1$. Then, for a sufficiently
 large absolute constant $c$ (depending on $\zeta, L$),
 when
 $\varepsilon \geq c n^{-2s/(5s + 2)}$,
we have that
\begin{align*}
\sup_{p \in \cH_0(s)} \EE_p[\psi_\tau(\cD'_k)] & \leq \frac{1}{10},\\ 
\sup_{p \in \{p \in \cQ_{0,[0,1]^3, \TV}(L,s): \inf_{q \in \cP_{0,[0,1]^3}} \|p - q\|_1 \geq \varepsilon\}} \EE_p[1 - \psi_\tau(\cD'_k)] & \leq \frac{1}{10} + \exp(-n/8).
\end{align*}
\end{theorem}


\noindent {\bf Remarks: }
\begin{itemize}

\item Theorem \ref{continuous:case:upper:bound:theorem} shows that the test $\psi_\tau(\cD'_N)$ matches the lower bound derived in Theorem~\ref{lower:bound:continuous:case}, showing that
under appropriate Lipschitzness conditions our test is a minimax 
optimal nonparametric test for conditional independence. 

\item We note that in this setting, a careful analysis of the expectation of our statistic under the null shows that the null set of distributions can be taken as $\cP_{0,[0,1]^3,\TV}(L) \cup \cP_{0,[0,1]^3,\chi^2}(L)$ which is a larger set of distributions in comparison to that of Theorem \ref{main:theorem:scaling:discrete:XY}.

\item Finally, the analysis in the continuous setting builds extensively on our analysis for the test in~\eqref{general:discretization:test}. However, as we detail in Appendix~\ref{app:upper} (see Lemmas~\ref{lem:aa},~\ref{lem:bb} and~\ref{lem:cc}),
careful analysis is needed to show 
that the additional discretization error of the $X$ and $Y$ variables does not change the mean and variance of our test statistic too much (under both the null and alternative).

\item In addition, in Appendix \ref{multivariate:Z:extension} we derive another upper bound for the case where $Z \in [0,1]^{d_Z}$ and $s \geq 1$ when $d_Z \leq 2$. It turns out that the minimax rate in this case is $n^{-2s/((4+ d_Z)s + 2)}$, which generalizes the above result (for the case $s \geq 1$). In addition we derive a matching lower bound. 
\end{itemize}

\section{Investigating Lipschitzness Conditions}
\label{examples:section}
In our upper and lower bounds, in order to tractably test conditional independence in the nonparametric setting, we impose various Lipschitzness conditions on the distributions under consideration. In order to build further intuition for these
conditions, in this section we derive several inclusions which relate the Lipschitzness classes defined in Sections \ref{discrete:case:lower:bound:section} and \ref{continuous:case:lower:bound:section}. 
We then give examples of natural classes of distributions which satisfy our various Lipschitzness conditions.

\subsection{Relationships between the Lipschitzness classes}

Recall the definitions of the null Lipschitzness classes in Definition~\ref{def:nullsmoothness}.
Our first result shows that the class of H\"{o}lder smooth distributions
contains the class of TV smooth distributions and $\chi^2$ smooth distributions.


\begin{lemma} 
We have the following inclusions 
\begin{align}
\label{eqn:first_inc} \cP_{0,[0,1],\chi^2}'(L) &\subseteq \cP'_{0,[0,1],\TV^2}(L), \\
\cP'_{0,[0,1],\TV}(\sqrt{L}) &\subseteq \cP'_{0,[0,1],\TV^2}(L).
\end{align}
\end{lemma}

\begin{proof} To prove this result 
we state a simple but useful direct corollary of the Cauchy-Schwarz inequality, which is also known in the literature as the T2 Lemma.
\begin{lemma}[T2 Lemma] For positive reals $\{u_i\}_{i \in [k]}$ and $\{v_i\}_{i \in [k]}$ we have
\begin{align*}
\frac{(\sum_{i \in [k]} u_i)^2}{\sum_{i \in [k]} v_i} \leq \sum_{i \in [k]} \frac{u_i^2}{v_i}.
\end{align*}
\end{lemma} 
\noindent By the T2 Lemma it is simple to see that
\begin{align*}
\MoveEqLeft d_{\chi^2}(p_{X | Z = z}, p_{X|Z = z'})  = \sum_{x} \frac{(p_{X | Z}(x |z) - p_{X | Z}(x | z'))^2}{p_{X | Z}(x | z')} \\
& \geq  \frac{(\sum_{x}|p_{X | Z}(x |z) - p_{X | Z}(x | z')|)^2}{\sum_x p_{X | Z}(x | z')} = \|p_{X|Z = z} - p_{X|Z= z'}\|_1^2.
\end{align*}
Hence we have that 
\begin{align*}
d_{\chi^2}(p_{X | Z = z}, p_{X|Z = z'}) \leq L |z - z'|~\implies~\|p_{X|Z = z} - p_{X|Z= z'}\|_1 \leq \sqrt{L |z - z'|},
\end{align*} 
and therefore we obtain the inclusion in~\eqref{eqn:first_inc}.
To derive the second inclusion note that when $z, z' \in [0,1]$ we have $|z - z'| \leq \sqrt{|z - z'|}$, and therefore 
\begin{align*}
\|p_{X|Z = z} - p_{X | Z = z'}\|_1 \leq \sqrt{L} |z - z'|~\implies~\|p_{X|Z = z} - p_{X | Z = z'}\|_1 \leq \sqrt{L |z - z'|}. 
\end{align*}
\end{proof}

\noindent In Definition~\ref{def:nullsmoothness} we assume that the marginal distributions of $X$ and $Y$ conditional on $Z$ are
each smooth. Our next result shows that up to a factor of $2$ this is equivalent to assuming TV Lipschitzness on the joint distribution of $(X,Y)$ conditional on 
$Z$.

\begin{lemma}\label{lemma:TV:inclusions} Define the class of distributions $\cP''_{0,[0,1], \TV}(L) \subset \cP'_{0,[0,1]}$ such that for each $p_{X,Y,Z} \in \cP''_{0,[0,1], \TV}(L)$ and all $z, z' \in [0,1]$:
\begin{align*}
\|p_{X,Y|Z = z} - p_{X,Y| Z = z'}\|_1 \leq L |z - z'|.
\end{align*}
Then \begin{align*}
\cP''_{0,[0,1], \TV}(L) \subseteq \cP'_{0,[0,1], \TV}(L),~~~\text{and}~~~ 
\cP'_{0,[0,1], \TV}(L) \subseteq \cP''_{0,[0,1], \TV}(2L). 
\end{align*}
\end{lemma}

\begin{proof} The first inclusion is a consequence of the triangle inequality:
\begin{align*}
\max(\|p_{X|Z = z} - p_{X|Z = z'}\|_1, \|p_{Y|Z = z} - p_{Y|Z = z'}\|_1) \leq \|p_{X,Y|Z = z} - p_{X,Y| Z = z'}\|_1.
\end{align*}
To obtain the second inclusion we note that $p_{X,Y|Z = z} = p_{X | Z = z} p_{Y | Z = z}$ and $p_{X,Y|Z = z'} = p_{X | Z = z'} p_{Y | Z = z'}$, and that $d_{\TV}$ is sub-additive on product distributions \cite{valiant2017automatic} so that
\begin{align*}
\|p_{X,Y|Z = z} - p_{X,Y| Z = z'}\|_1 \leq \|p_{X|Z = z} - p_{X|Z = z'}\|_1 + \|p_{Y|Z = z} - p_{Y|Z = z'}\|_1.
\end{align*}
\end{proof}

\noindent Similar statements to Lemma \ref{lemma:TV:inclusions} hold for the classes $\cP'_{0,[0,1], \TV^2}(L)$ and $\cP_{0,[0,1]^3,\TV}$. 
For brevity we do not state them here. We now state another similar result for the Lipschitzness class $\cP'_{0,[0,1], \chi^2}$. 

\begin{lemma} Define the class of distributions $\cP''_{0,[0,1], \chi^2}(L) \subset \cP'_{0,[0,1]}$, such that for each $p_{X,Y,Z} \in \cP''_{0,[0,1], \chi^2}(L)$ and all $z, z' \in [0,1]$ we have
\begin{align*}
d_{\chi^2}(p_{X,Y | Z = z}, p_{X,Y | Z = z'}) \leq L |z - z'|. 
\end{align*}
Then 
\begin{align*} 
\cP''_{0,[0,1], \chi^2}(L) \subseteq  \cP'_{0,[0,1], \chi^2}(L)~~~\text{and}~~~
\cP'_{0,[0,1], \chi^2}(L) \subseteq \cP''_{0,[0,1], \chi^2}(2L + L^2).
\end{align*}
\end{lemma}

\begin{proof} We start by showing the first inclusion. Note that by the T2 Lemma  
\begin{align*}
\MoveEqLeft d_{\chi^2}(p_{X,Y | Z = z}, p_{X,Y | Z = z'}) = \sum_{x,y} \frac{p^2_{X,Y|Z}(x,y|z)}{p_{X,Y|Z}(x,y|z')} - 1 \\
& \geq \sum_x \frac{(\sum_y p_{X,Y|Z}(x,y|z))^2}{\sum_y p_{X,Y|Z}(x,y|z')} - 1 = d_{\chi^2}(p_{X | Z = z}, p_{X| Z = z'}).
\end{align*}
By symmetry it also follows that $d_{\chi^2}(p_{X,Y | Z = z}, p_{X,Y | Z = z'}) \geq d_{\chi^2}(p_{Y | Z = z}, p_{Y| Z = z'})$ which shows the first inclusion. For the second inclusion using the fact that $p_{X,Y|Z = z} = p_{X | Z = z} p_{Y | Z = z}$ and $p_{X,Y|Z = z'} = p_{X | Z = z'} p_{Y | Z = z'}$, it is simple to verify that
\begin{align*}
d_{\chi^2}(p_{X,Y|Z= z}, p_{X,Y|Z= z'}) & = d_{\chi^2}(p_{X|Z= z}, p_{X|Z= z'}) + d_{\chi^2}(p_{Y|Z= z}, p_{Y|Z= z'}) \\
& + d_{\chi^2}(p_{X|Z= z}, p_{X|Z= z'})d_{\chi^2}(p_{Y|Z= z}, p_{Y|Z= z'}),
\end{align*}
which yields the desired conclusion by noting that this expression in turn is smaller than $2L|z - z'| + L^2 |z - z'|^2 \leq 2L|z - z'| + L^2 |z - z'|$, when $p_{X,Y,Z} \in \cP'_{0,[0,1], \chi^2}(L)$. 
\end{proof}
\noindent A similar result also holds for the set $\cP_{0,[0,1]^3,\chi^2}(L)$ but once again we do not state the result here for brevity. 

\subsection{Distribution Families in our Lipschitzness Classes}
Next we give some concrete examples of distributions which belong to the different Lipschitzness classes. We begin by showing that 
smoothness of the log-conditional density is sufficient to ensure that the distribution belongs to both the TV and $\chi^2$ Lipschitzness classes. We then
show that a broad subset of exponential family distributions have a smooth log-conditional distributions.

\begin{lemma} \label{log:Lipschitz:functions:lemma} Take a distribution $p_{X,Y,Z} \in \cP'_{0,[0,1]}$. Suppose that the functions $\log p_{X|Z}(x|z)$, $\log p_{Y|Z}(y |z)$ are $L$-Lipschitz in $z$ for all values of $x$ and $y$. Then the distribution $p_{X,Y,Z}$ belongs to $\cP'_{0,[0,1], \TV}(e^L-1) \cap \cP'_{0,[0,1], \chi^2}(e^L-1)$.  
\end{lemma}
\begin{proof}
We begin by showing that $p_{X,Y,Z} \in \cP'_{0,[0,1], \chi^2}(e^L-1)$. Note that
\begin{align*}
\sum_{x} \frac{p^2_{X|Z}(x| z)}{p_{X|Z}(x|z')} - 1 = \sum_{x} \bigg(\frac{p_{X|Z}(x| z)}{p_{X|Z}(x|z')} - 1\bigg)p_{X|Z}(x| z). 
\end{align*}
As a consequence it suffices to show that,
\begin{align*}
\frac{p_{X|Z}(x| z)}{p_{X|Z}(x|z')} - 1 \leq (e^L-1) |z - z'|,
\end{align*}
for all $z,z'\in[0,1]$ and all $x$ (and the analogous claim for $p_{Y|Z}$) in order to conclude that $p_{X,Y,Z} \in \cP'_{0,[0,1], \chi^2}(e^L-1)$. 
Since $\log p_{X|Z}(x | z)$ is $L$-Lipschitz in $z$ it follows that for values of $|z - z'| \leq 1$: 
\begin{align*}
\MoveEqLeft \frac{p_{X|Z}(x| z)}{p_{X|Z}(x|z')} - 1 \leq \exp(L|z - z'|) - 1 = L|z-z'| + \sum_{k \geq 2}(L|z - z'|)^k/k! \\
& \leq L|z-z'| + L|z-z'| \sum_{k \geq 2} L^{k-1}/k! = L|z-z'| + L|z-z'| (e^L - 1 - L)/L \\
& = (e^L - 1) |z- z'|.
\end{align*}
This, together with an identical claim for $p_{Y|Z}$, proves the first claim, i.e., $p_{X,Y,Z} \in \cP'_{0,[0,1], \chi^2}(e^L-1)$. To establish the second claim note that 
\begin{align*}
\MoveEqLeft \sum_{x} |p_{X|Z}(x|z) - p_{X|Z}(x|z')|  \\
& = \sum_{x} \bigg(\frac{\max(p_{X|Z}(x|z), p_{X|Z}(x|z'))}{\min(p_{X|Z}(x|z), p_{X|Z}(x|z'))} - 1\bigg) \min(p_{X|Z}(x|z), p_{X|Z}(x|z')) \\
& \leq   \sum_{x} \bigg(\frac{\max(p_{X|Z}(x|z), p_{X|Z}(x|z'))}{\min(p_{X|Z}(x|z), p_{X|Z}(x|z'))} - 1\bigg) p_{X|Z}(x|z).
\end{align*}
Hence the same proof as above applies. This completes the proof.
\end{proof}

\noindent We now state several similar and related results without proof, noting that their proofs are nearly identical to the proof of Lemma \ref{log:Lipschitz:functions:lemma}. 
\begin{lemma}\label{continuous:lemma:lipschitz:TV:smooth:null} Take a distribution $p_{X,Y,Z} \in \cQ'_{0,[0,1]}$. Suppose that the function $\log p_{X,Y|Z}(x,y|z)$ is $L$-Lipschitz in $z$ for all values of $x$ and $y$. Then $p_{X,Y,Z} \in \cQ'_{0,[0,1],\TV}(e^L-1)$.
\end{lemma}

\begin{lemma}\label{discrete:lemma:lipschtiz:alternative} Let $p_{X,Y,Z} \in \cP_{0,[0,1]^3}$. Suppose that the functions $\log p_{X|Z}(x|z)$, $\log p_{Y|Z}(y |z)$ are $L$-Lipschitz in $z$ for all values of $x$ and $y$. Then the distribution $p_{X,Y,Z}$ also belongs to $p_{X,Y,Z} \in \cP_{0,[0,1]^3, \TV}(e^L-1) \cap \cP_{0,[0,1]^3, \chi^2}(e^L-1)$.  
\end{lemma}

\begin{lemma}\label{continuous:lemma:lipschitz:TV:smooth} Let $p_{X,Y,Z} \in \cQ_{0,[0,1]^3}$. Suppose that the function $\log p_{X,Y|Z}(x,y|z)$ is $L$-Lipschitz in $z$ for all $x$ and $y$, and further that the function $p_{X,Y|Z}(x,y|z)$ is jointly $C$-Lipschitz in $x$ and $y$, for all $z$, i.e.
\begin{align}\label{xy:Lipschitzness:condition}
|p_{X,Y|Z}(x,y|z) - p_{X,Y|Z}(x',y'|z)| \leq C(|x - x'| + |y - y'|).
\end{align}
Then $p_{X,Y,Z} \in \cQ_{0,[0,1]^3,\TV}((e^L-1) \vee \sqrt{2} C, 1)$.
\end{lemma}

\noindent Lemmas \ref{continuous:lemma:lipschitz:TV:smooth:null} and \ref{continuous:lemma:lipschitz:TV:smooth} are regarding the continuous case, and are therefore slightly different from Lemmas \ref{log:Lipschitz:functions:lemma} and \ref{discrete:lemma:lipschtiz:alternative}. Hence for completeness we give the proof of Lemma \ref{continuous:lemma:lipschitz:TV:smooth} in the appendix. Roughly, these results taken together show that Lipschitzness of the log conditional density imply the various Lipschitzness conditions we impose. Our next set of results shows that 
a broad class of natural exponential family type distributions, in fact, have smooth log conditional densities.

\begin{lemma}\label{log:Lipschitz:densities} Consider the density $p_{W|Z}(w|z) \varpropto \exp(g(w,z))$, where $g(w,z)$ is an $L$-Lipschitz function in $z \in [0,1]$ for all values of $w$. Then the function $\log p_{W|Z}(w|z)$ is $2L$-Lipschitz. 
\end{lemma}
\noindent 
We note that in the lemma above, $W$ can be taken as a vector of any dimension 
so the lemma applies to $p_{X|Z}(x|z)$ and $p_{Y|Z}(y|z)$ as well as to $p_{X,Y|Z}(x,y|z)$. The lemma also applies in both discrete $W$ as well as continuous $W$ cases. 

\begin{proof}
We consider the differences
\begin{align*}
\MoveEqLeft \log \frac{\exp (g(w, z))}{\sum_w \exp (g(w, z))} - \log \frac{\exp (g(w, z'))}{\sum_w \exp (g(w, z'))} \\
& \leq  (g(w,z) - g(w,z')) -\log \frac{\sum_w \exp (g(w,z))}{\sum_w \exp (g(w,z'))}.
\end{align*}
Next we use Jensen's inequality and the fact that $-\log$ is a convex function to show that
\begin{align*}
-\log \frac{\sum_w \exp (g(w,z))}{\sum_w \exp (g(w,z'))} & = -\log \frac{\sum_w \exp(g(w,z'))\exp (g(w,z) - g(w,z'))}{\sum_w \exp (g(w,z'))} \\
& \leq \sum_{w} \frac{\exp(g(w,z'))}{\sum_w \exp (g(w,z'))} (g(w,z') - g(w,z))\\
& \leq \sum_{w} \frac{\exp(g(w,z'))}{\sum_w \exp (g(w,z'))}  |g(w,z') - g(w,z)| \\
& \leq L |z - z'|.
\end{align*}
Putting things together we get 
\begin{align*}
\MoveEqLeft \log \frac{\exp (g(w, z))}{\sum_w \exp (g(w,z))} - \log \frac{\exp (g(w,z'))}{\sum_w \exp (g(w,z'))} \\
& \leq |g(w,z) - g(w,z')| +L |z - z'| \leq 2 L |z - z'|.
\end{align*}
Reversing the roles of $z$ and $z'$ we conclude. The same proof goes through in the continuous case, where summations have to be substituted with integrals.
\end{proof}

\noindent Finally, in the continuous case we provide a family of distributions for which 
$\log p_{X,Y|Z}(x,y|z)$ is $L$-Lipschitz in $z$ and $p_{X,Y|Z}(x,y|z)$ is $C$-Lipschitz in $x$ and $y$ as required in Lemma~\ref{continuous:lemma:lipschitz:TV:smooth}.

\begin{lemma}
Suppose that $g(x,y,z) : [0,1]^3 \mapsto [-M,M]$ is a bounded $L$-Lipschitz function, i.e., $|g(x,y,z) - g(x',y',z')| \leq L(|x - x'| + |y-y'| + |z - z'|)$. Take $p_{X,Y,Z}(x,y,z) \varpropto \exp(g(x,y,z))$. Then 
\begin{align*}
p_{X,Y|Z}(x,y|z) = \frac{\exp(g(x,y,z))}{\int_{[0,1]^2} \exp(g(x,y,z)) dx dy},
\end{align*} 
satisfies \eqref{xy:Lipschitzness:condition} with a constant $C = Le^{2M}$ and furthermore $\|p_{X,Y|Z= z} - p_{X,Y|Z = z'}\|_1 \leq (e^{2L} -1) |z - z'|$. 
\end{lemma}

\begin{proof} By Lemmas \ref{continuous:lemma:lipschitz:TV:smooth} and \ref{log:Lipschitz:densities}, since $g$ is $L$-Lipschitz in $z$ for all $x,y$ we have that $\|p_{X,Y|Z= z} - p_{X,Y|Z = z'}\|_1 \leq (e^{2L} -1) |z - z'|$. 
It remains to show that \eqref{xy:Lipschitzness:condition} holds with the appropriate constant $C$. By definition we have

\begin{align*}
\frac{|\exp(g(x,y,z)) - \exp(g(x',y',z))|}{\int_{[0,1]^2} \exp(g(x,y,z)) dx dy} \leq \exp(M) |\exp(g(x,y,z)) - \exp(g(x',y',z))|.
\end{align*}
Denote for brevity $g = g(x,y,z)$ and $g' = g(x',y',z)$ and note that $|g|, |g'| \leq M$. By a Taylor expansion
\begin{align*}
\MoveEqLeft |e^g - e^{g'}| \leq  |g - g'|\sum_{k = 1}^{\infty} \frac{\sum_{i = 0}^{k-1} |g|^i |g'|^{(k-1-i)}}{k!} \leq  |g - g'| \exp(M) \\
& \leq L\exp(M) (|x - x'| + |y-y'|).
\end{align*}
We conclude that
\begin{align*}
\frac{|\exp(g(x,y,z)) - \exp(g(x',y',z))|}{\int_{[0,1]^2} \exp(g(x,y,z)) dx dy} \leq L\exp(2M) (|x - x'| + |y-y'|),
\end{align*}
which is our desired result.
\end{proof}



\section{Simulations} \label{numerical:experiments:section}

In this section we report some numerical results on synthetic data to validate some of our theoretical predictions.

We note that all of our procedures require specifying a rejection 
threshold $\tau$ for the different tests. While we know the precise order of $\tau$ we do not know the appropriate constant. 
In order to handle this in practice we use a permutation approach which is often used in practice (see for instance \cite{zhang2012kernel}). 
In more details, we calculate the statistic $T$, and perform a permutation to obtain a reference distribution for the test statistic $T$ under the null hypothesis. 
Recall that we construct the datasets $\cD_m = \{(X_i, Y_i) : Z_i \in C_m, i \in [N]\}$ for each of the $d$ bins $C_m$. For each $\cD_m$ we permute the $X_i$ and $Y_i$ values to simulate independently drawn values. Suppose
that $\sigma_m$ samples fall in the bin $C_m$, then for a permutation $\pi: [\sigma_m] \mapsto [\sigma_m]$ we consider $\cD_m^\pi = \{(X_{\pi(i)}, Y_i) : Z_i \in C_m, i \in [N]\}$. 
We recalculate the statistic $T$ over different sets $\cD_m^\pi$ (using different permutations $\pi$ for each set), and we repeat this $M$ times, each time denoting the value permuted statistic with $T_i$ for $i \in [M]$. Finally we compare our statistic $T$ with the values of the statistics in the set $\{T_1, \ldots, T_M\}$ and return the value 
$M^{-1} \sum_{i \in [M]}\mathbbm{1}(T_i > T)$. We would then reject the null hypothesis if this value is smaller than some pre-specified cutoff (say $0.05$).

This procedure is motivated by the intuition that permuting indexes within bins $Z_i \in C_m$ generates approximately conditionally independent samples. While this intuition is apparent, in contrast to the settings of two-sample testing and independence testing, it is not straightforward to show that this procedure correctly controls the Type I error. We note that this permutation 
procedure works remarkably well in practice. However, rigorously proving the validity of this permutation procedure, and studying its power,
warrants further research and is delegated to future work. 

\subsection{Finite Discrete $X$ and $Y$}

In this subsection we consider finite discrete $X$ and $Y$ with fixed number of categories $\ell_1 = 2$ and $\ell_2 = 3$. In order for us to construct examples that satisfy the conditions of Theorems \ref{main:theorem:finite:discrete:XY} or \ref{main:theorem:scaling:discrete:XY}, we rely on the examples studied in Section \ref{examples:section}. Under the null hypothesis we consider the following probabilities
\begin{align*}
p_{X,Y|Z}(1,1|z) & \varpropto \exp (z + \tanh(z)), ~~~~~~~~~~~\,~~~~ p_{X,Y|Z}(1,2|z) \varpropto \exp (z + \cos(z)), \\
p_{X,Y|Z}(1,3|z) & \varpropto  \exp (z + \sin(z)), ~~~~~~~~~~~~~~~~~~ p_{X,Y|Z}(2,1|z) \varpropto \exp (\cos(z) - 1+ \tanh(z)),\\
p_{X,Y|Z}(2,2|z) & \varpropto \exp (\cos(z) - 1 + \cos(z)),  ~~~~~~ p_{X,Y|Z}(2,3|z) \varpropto  \exp (\cos(z) - 1 + \sin(z)).
\end{align*}
In this setting all of the exponents are Lipschitz, and can be decomposed so that the random variables are conditionally independent. Under the alternative we consider the following distribution
\begin{align*}
p_{X,Y|Z}(1,1|z) & \varpropto \exp (z), ~~~~~~~~~~~~~~~~ p_{X,Y|Z}(1,2|z) \varpropto \exp (\tanh(z)), \\
p_{X,Y|Z}(1,3|z) & \varpropto  \exp (\sin(z)), ~~~~~~~~~~ p_{X,Y|Z}(2,1|z) \varpropto \exp (\cos(z)),\\
p_{X,Y|Z}(2,2|z) & \varpropto \exp (z + 1),  ~~~~~~~~~~~ p_{X,Y|Z}(2,3|z) \varpropto  \exp (\tanh(z) - 1).
\end{align*}
In the example above the probabilities do not factor as products so the variables are not conditionally independent, however all functions in the exponents are still 
Lipschitz so that the distribution is TV smooth by Lemma \ref{log:Lipschitz:densities}. Figure \ref{discrete:results:figure} shows the results of running the weighted test of Section \ref{scaling:l1:l2:section} on the above examples. For each sample size of $N = 100, 200, \ldots, 1000$, we perform $100$ simulations. 
Within each simulation we permute $M = 100$ times and compute the value
$M^{-1} \sum_{i \in [M]}\mathbbm{1}(T_i > T)$. The final size and power are calculated based on how many (out of the $100$) values were smaller than or equal to $0.05$. 

\begin{figure}
\centering
\begin{subfigure}{.5\textwidth}
\includegraphics[width=.75\linewidth]{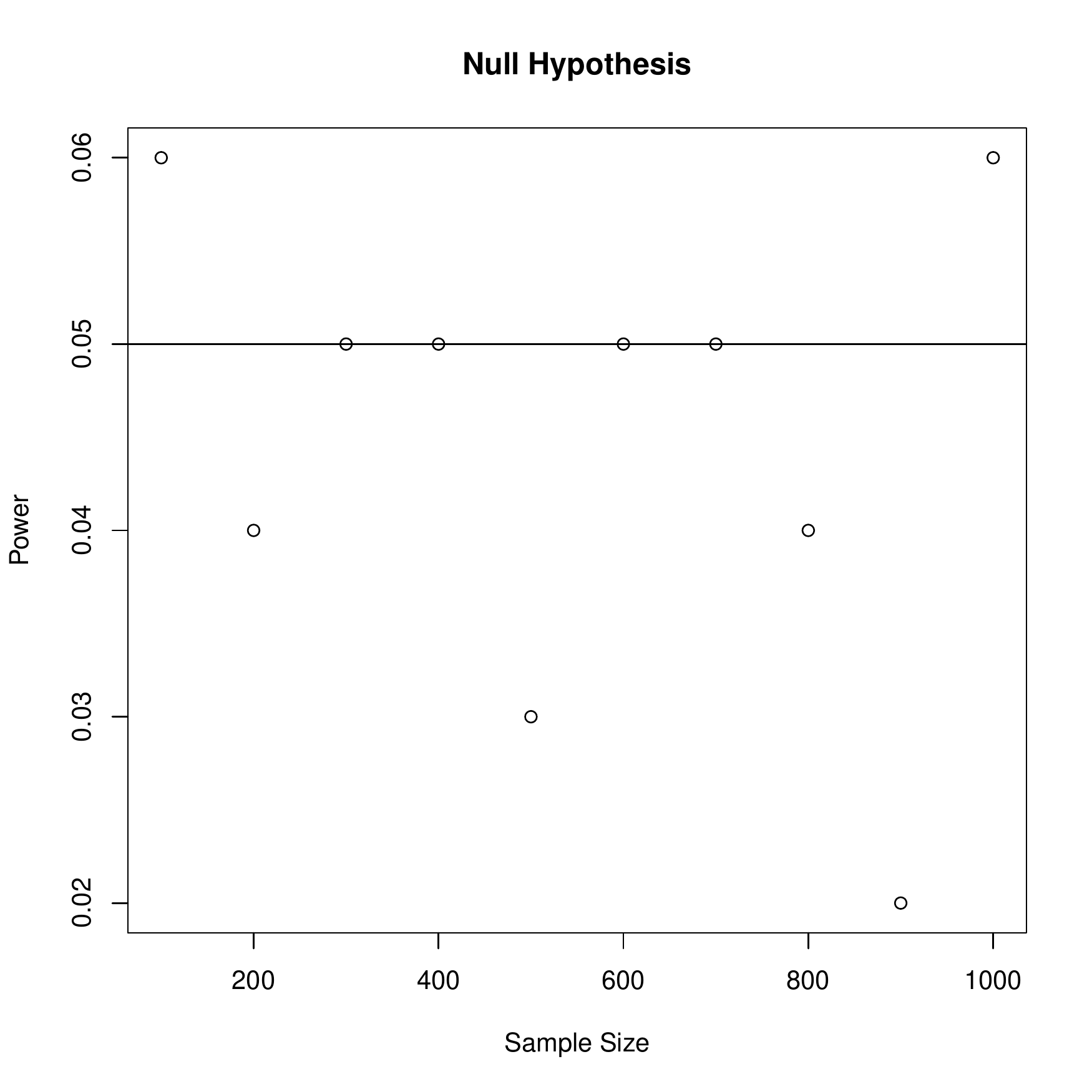}
\end{subfigure}\begin{subfigure}{.5\textwidth}
\includegraphics[width=.75\linewidth]{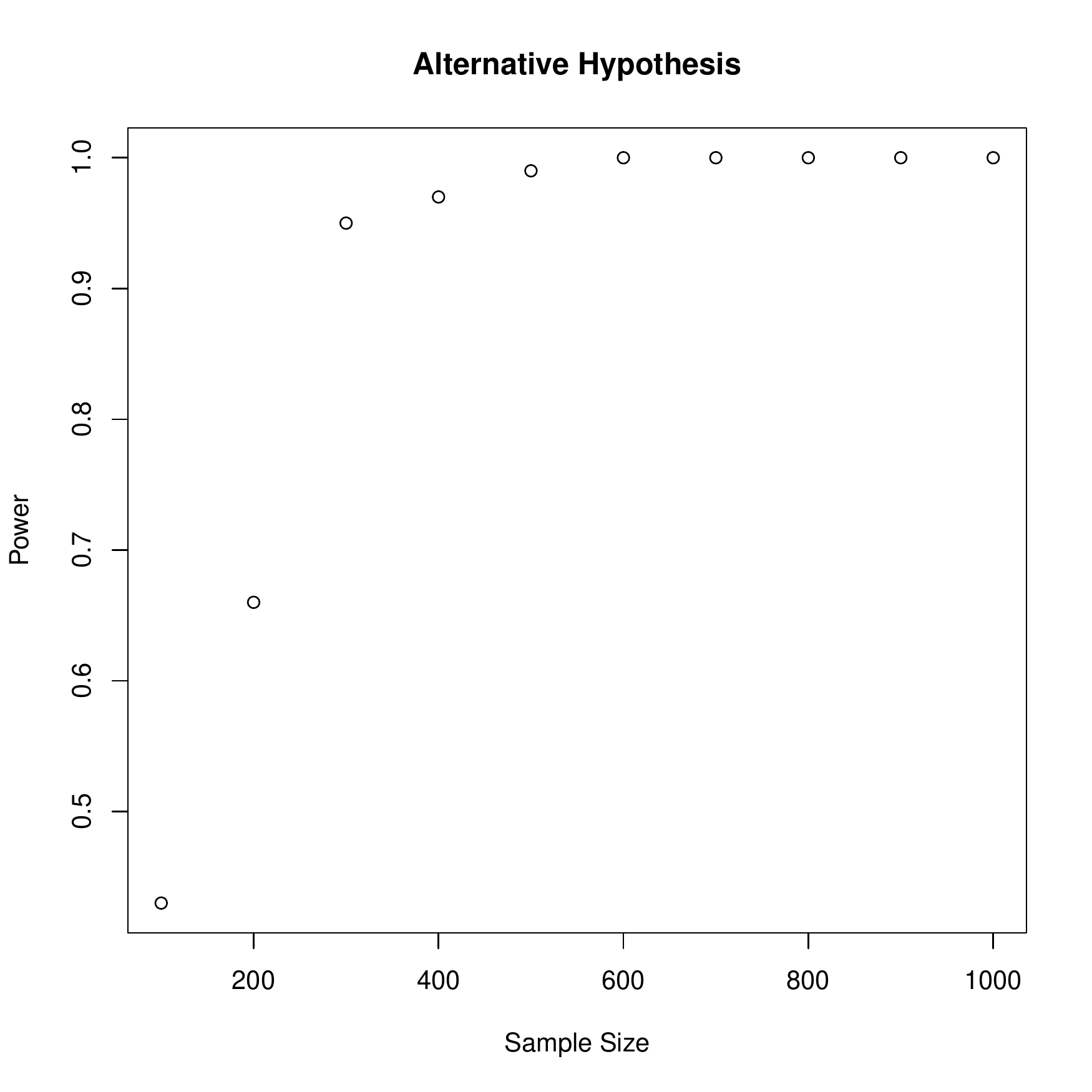}
\end{subfigure}\caption{This figure displays the size and power of the test in the discrete $X,Y$ and continuous $Z$ example. We see that under the null hypothesis the size is gravitating around $0.05$ which is also the most common size across all simulations. The power of the test increases steadily with the increase of the sample size, and reaches $1$ when the sample size is $1000$. }\label{discrete:results:figure}
\end{figure}

\subsection{Continuous $X,Y$ and $Z$}

In this subsection we consider the following examples. Under $H_0$ we generate 
\begin{align*}
X = \frac{U_1 + Z}{2} \mbox{ and } Y = \frac{U_2 + Z}{2},
\end{align*} 
where $U_1, U_2, Z \sim U([0,1])$ are independent. Under the alternative, $H_1$, we generate 
\begin{align*}
X = \frac{U_1 + U + Z}{3} \mbox{ and } Y = \frac{U_2 + U + Z}{3},
\end{align*} 
where $U, U_1, U_2, Z \sim U([0,1])$ are independent. A straightforward calculation (see Appendix~\ref{app:exp}) shows that these distributions belong to the classes 
$\cP_{0,[0,1]^3,\TV}(L)$ and $\cQ_{0,[0,1]^3, \TV}(L,1)$ (respectively) 
for appropriately chosen constants $L$, so that the conditions of Theorem \ref{continuous:case:upper:bound:theorem} hold. 

Figure \ref{continuous:results:figure} shows the results of running the weighted continuous test described in Section \ref{cont:case:upper:bounds:section} for these examples. For each sample size of $N = 100, 200, \ldots, 1000$, we perform $100$ simulations. 
Within each simulation we permute $M = 100$ times
and compute the value
$M^{-1} \sum_{i \in [M]}\mathbbm{1}(T_i > T)$. The final size and power are calculated based on how many (out of the $100$) values were smaller than or equal to $0.05$. 

\begin{figure}
\centering
\begin{subfigure}{.5\textwidth}
\includegraphics[width=.75\linewidth]{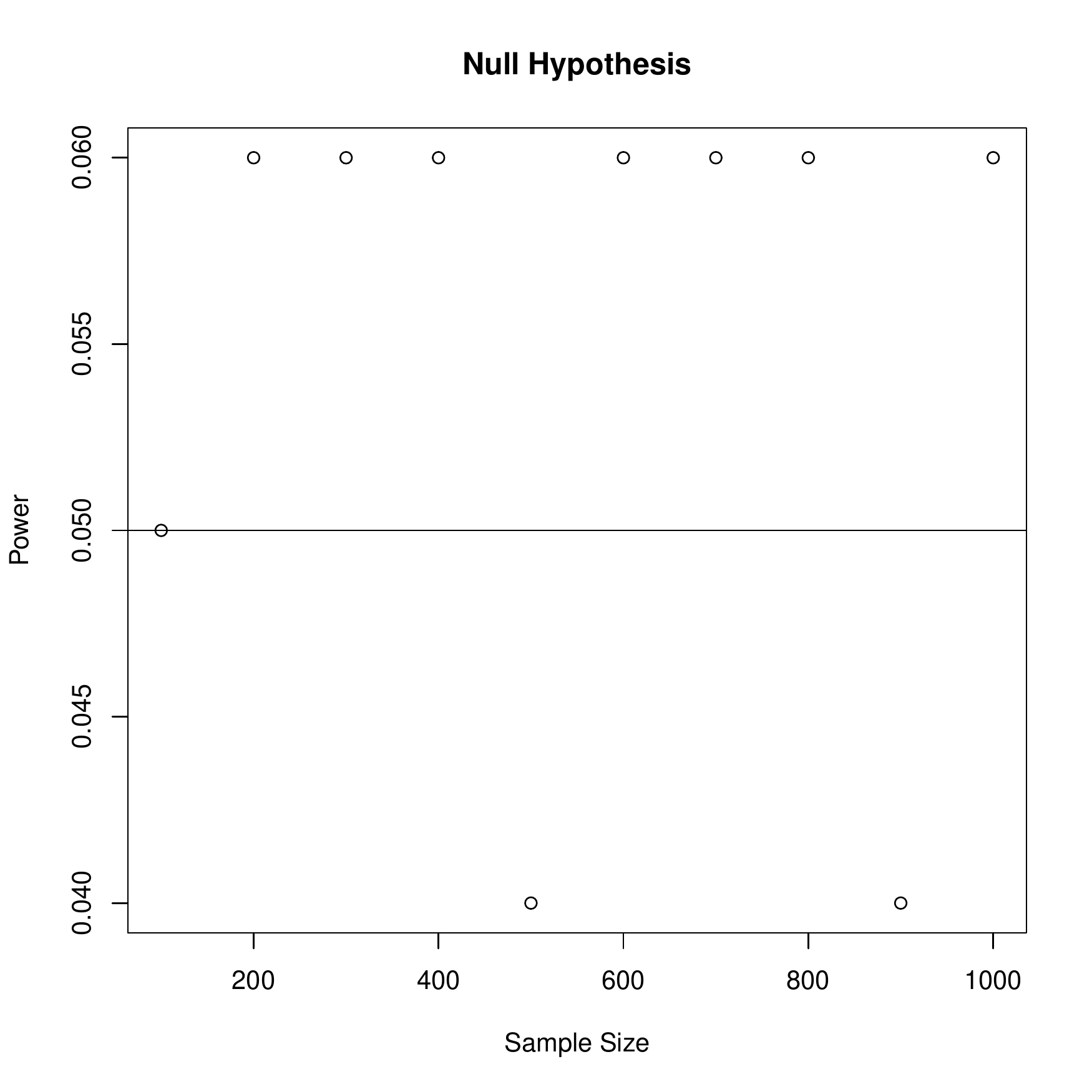}
\end{subfigure}\begin{subfigure}{.5\textwidth}
\includegraphics[width=.75\linewidth]{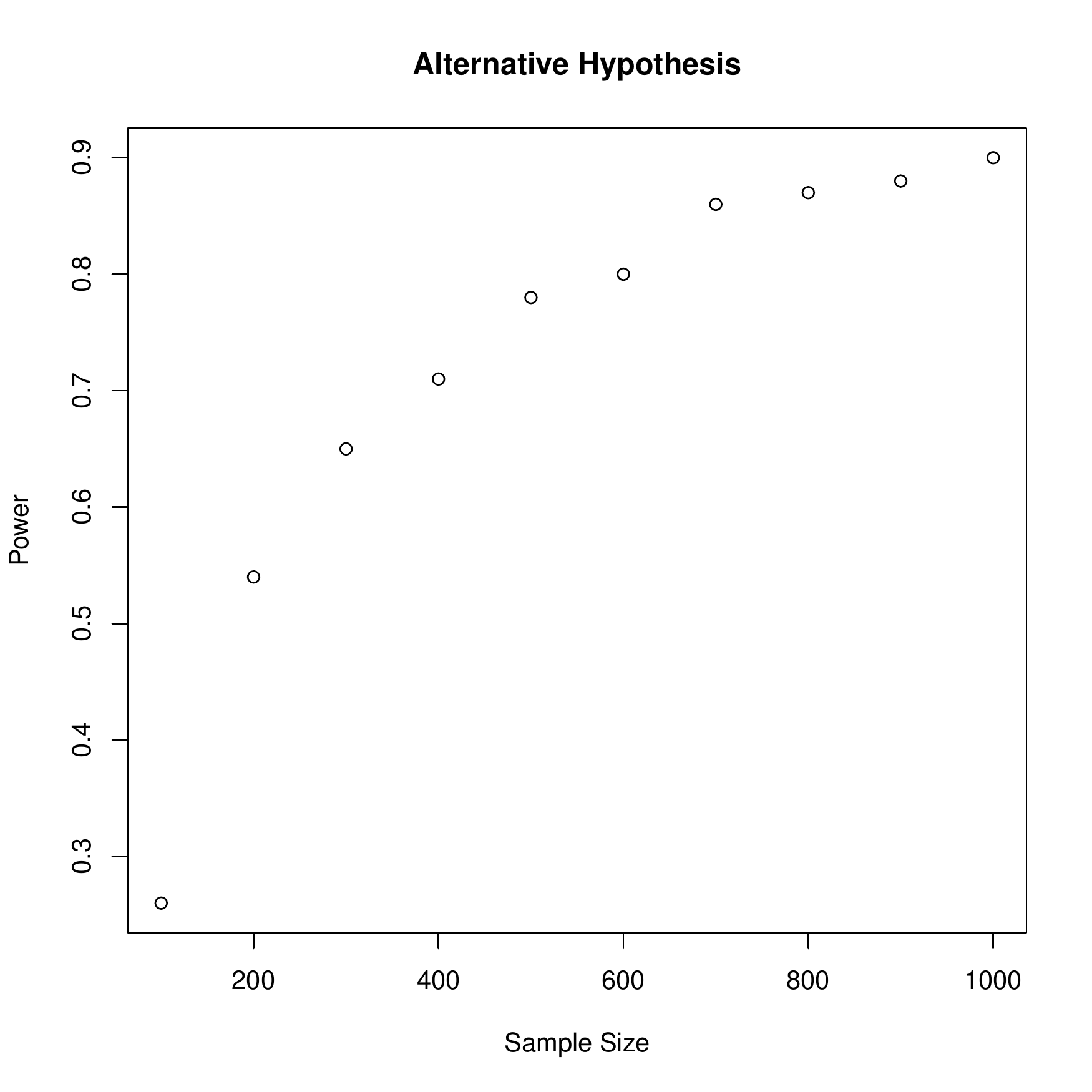}
\end{subfigure}\caption{This figure displays the size and power of the test in the continuous $X,Y,Z$ example. We see that under the null hypothesis the size is very slightly inflated at $0.06$ for most of the simulations, which may be due to the limited number of replications of each simulation and also due to the limited number of permutations within each simulation. The power of the test increases steadily with the increase of the sample size, and reaches $0.9$ when the sample size is $1000$. }\label{continuous:results:figure}
\end{figure}

\section{Discussion}\label{discussion:section}
In this paper, we have studied nonparametric CI testing from a minimax perspective. We derived upper and lower 
bounds on the minimax critical radius in three main settings --- (1) $X,Y$ discrete and supported on a fixed number of 
categories, $Z$ continuous on $[0,1]$, (2) $X,Y$ discrete on a growing number of categories $Z$ continuous on $[0,1]$ and (3) $X,Y,Z$ absolutely continuous and supported on $[0,1]^3$. 
In order to develop interesting minimax bounds, we introduced and studied several natural Lipschitzness conditions for conditional distributions.
In addition we provided a novel construction of a coupling between a conditionally independent distribution and an arbitrary distribution of bounded support, leading to a new proof of the hardness result of \citet{shah2018hardness}. Finally, the CI tests that we developed are implementable and perform well in practice as evidenced by our simulation study in Section \ref{numerical:experiments:section}.

There are several open questions which we intend to investigate in our future work. Moving beyond the total variation metric, a natural challenge is to derive 
minimax rates for the critical radius in other metrics. Another technical challenge is to move beyond the requirement that 
$\frac{\ell_1^{4}}{\ell_2} \lesssim n^{3}$ (where $\ell_1 \geq \ell_2$), 
which we impose 
in the scaling $\ell_1, \ell_2$ case. We believe that the analysis in this case is challenging and 
would require designing new tests, or deriving new lower bound techniques, and is left for future research.
Identifying conditions under which the natural permutation procedure of Section \ref{numerical:experiments:section} correctly controls the Type I error and has high power is also a challenging direction that we hope to pursue.

\section{Acknowledgments} The authors would like to thank Ilmun Kim for helpful discussions on the topic.

\bibliographystyle{plainnat}
\bibliography{condindeptesting}

\appendix

\newpage

\section{Proofs from Section \ref{hardness:section}}\label{hardness:appendix}

\begin{proof}[Proof of Theorem \ref{no-free-lunch:thm}] For the sake of simplicity we only consider the case $d_X = d_Y = d_Z = 1$. The more general case follows the same strategy of proof, with some minor modifications. Following the proof of Shah and Peters \cite{shah2018hardness}, it suffices to re-prove the following key lemma in their argument. 

\begin{lemma}\label{shah:peters:main:lemma} Suppose $(X,Y,Z) \in \RR^3$ have a distribution supported either on $[-M,M]^3$ for some $M \in (0,\infty)$, or on $(-\infty, \infty)^3$. Let $(X_i,Y_i,Z_i)_{i \in [n]}$ be $n$ i.i.d. copies of $(X,Y,Z)$. Given $\delta > 0$ there exists $C := C(\delta)$ such that for all $\varepsilon > 0$ and all Borel sets $D \subseteq \RR^{3n} \times [0,1]$, it's possible to construct an i.i.d. sequence $(\tilde X_i, \tilde Y_i, \tilde Z_i)_{i \in [n]}$ such that $\tilde X_i \independent \tilde Y_i | \tilde Z_i$ for all $i$ and 
\begin{enumerate}
\item[i.] $\PP(\max_{i \in [n]}\|(\tilde X_i, \tilde Y_i, \tilde Z_i) - ( X_i,  Y_i,  Z_i)\|_\infty < \varepsilon) > 1 - \delta$,
\item[ii.] If $U$ is uniform on $[0,1]$ independently of $(\tilde X_i, \tilde Y_i, \tilde Z_i)_{i \in [n]}$ then
$$
\PP(((\tilde X_i, \tilde Y_i, \tilde Z_i)_{i \in [n]}, U) \in D) \leq C \mu(D),
$$
\end{enumerate} 
where $\mu$ is the Lebesgue measure. 
\end{lemma}

\begin{remark}The lemma is stated and proved assuming $(X,Y,Z) \in \RR^3$, but the proof trivially extends to any $(X,Y,Z)\in \RR^{d_X + d_Y + d_Z}$ for $d_X, d_Y, d_Z \in \NN$. 
\end{remark}
We prove this result below. With this the proof is complete. 
\end{proof}

\begin{proof}[Proof of Lemma \ref{shah:peters:main:lemma}]

\noindent {\bf Step I (preparation).} First consider the case that the support is $(-\infty, \infty)^3$, i.e., $M = \infty$. We can always find an $M' := M'(\delta) < \infty$ such that $\PP(\|(X,Y,Z)\|_\infty > M') < \delta/2n$. Construct $\bar X, \bar Y, \bar Z$ which coincide with $(X,Y,Z)$ if $\|(X,Y,Z)\|_\infty \leq M'$ and are uniform on $[-M', M']^3$ otherwise. By the union bound $\PP(\forall i\in [n] : (\bar X_i, \bar Y_i, \bar Z_i) = (X_i,Y_i,Z_i)) > 1 -\delta/2$. We henceforth work with $(\bar X_i, \bar Y_i, \bar Z_i)_{i \in [n]}$ (denoted with $(X_i,Y_i,Z_i)_{i \in [n]}$ for convenience), and will show that there exist $(\tilde X_i, \tilde Y_i, \tilde Z_i)_{i \in [n]}$ satisfying
$\PP(\max_{i \in [n]}\|(\tilde X_i, \tilde Y_i, \tilde Z_i) - ( \bar X_i,  \bar Y_i,  \bar Z_i)\|_\infty < \varepsilon) = 1$
which implies i., and we will show that ii. is also satisfied. Hence we will assume $M < \infty$ from now on.

Second we note that (without loss of generality) we may assume that the density $p_{X,Y,Z}(x,y,z)$ is bounded by some constant $L := L(\delta)$. This is so since each distribution of (potentially) unbounded density can be well approximated by a distribution of bounded density with high probability. To see this note that the set $S_{\bar L} := \{(x,y,z) | p_{X,Y,Z}(x,y,z)> \bar L\} \downarrow \varnothing$ when $\bar L \rightarrow \infty$. Therefore for any $\delta$ we can take $\bar L(\delta)$ large enough so that $\PP((X,Y,Z) \in S^c_{\bar L(\delta)}) > 1 - \delta/2n$. Thus we can construct $\bar X, \bar Y, \bar Z$ as $X,Y,Z$ if $(X,Y,Z)\in S^c_{\bar L(\delta)}$ and $\bar X, \bar Y, \bar Z$ being uniform on $[-M,M]^3$ otherwise. This distribution has density bounded by $ L(\delta) := \bar L(\delta) + \delta/(2n(2M)^3)$ and satisfies $\PP((\bar X, \bar Y, \bar Z) = (X,Y,Z)) > 1 -\delta/2n$, and therefore by the union bound $\PP(\forall i\in [n] : (\bar X_i, \bar Y_i, \bar Z_i) = (X_i,Y_i,Z_i)) > 1 -\delta/2$. As before, we will henceforth work with $(\bar X_i, \bar Y_i, \bar Z_i)_{i \in [n]}$ (denoted with $(X_i,Y_i,Z_i)_{i \in [n]}$ for convenience), and will show that there exist $(\tilde X_i, \tilde Y_i, \tilde Z_i)_{i \in [n]}$ satisfying
$\PP(\max_{i \in [n]}\|(\tilde X_i, \tilde Y_i, \tilde Z_i) - ( \bar X_i,  \bar Y_i,  \bar Z_i)\|_\infty < \varepsilon) = 1$
which implies i., and we will show that ii. is also satisfied. 

\noindent {\bf Step II (construction).}  Let $\{A_1, \ldots, A_{m}\}$ denote an equi-partition of $[-M,M]$ in intervals. Similarly let $\{B_1,\ldots, B_{m}\}$ and $\{C_1, \ldots, C_m\}$ be equi-partitions of $[-M,M]$. Divide each $C_k$ further in $m^2$ sub-intervals of equal length denoted by $C_{ijk}$, so that each of these small intervals corresponds to a pair $(A_i, B_j)$. The lengths of each interval $A_i$, $B_i$ or $C_i$ is $\frac{2M}{m}$, while the length of an interval $C_{ijk}$ is $\frac{2M}{m^3}$. Given a draw $(X,Y,Z)$ we construct $(\tilde X, \tilde Y, \tilde Z)$ as follows. Suppose that $X \in A_i$, $Y \in B_j$ and $Z \in C_k$. Then we generate uniformly $\tilde Z \in C_{ijk}$ and $(\tilde X, \tilde Y)$ uniformly in $A_i \times B_j$. By definition then $\tilde X \independent \tilde Y | \tilde Z$. We refer to Figure \ref{viz:smart:construction} for a visualization of this construction. In addition it is clear that by construction $\PP(\max_{i \in [n]}\|(\tilde X_i, \tilde Y_i, \tilde Z_i) - ( X_i,   Y_i,   Z_i)\|_\infty < \frac{2M}{m}) = 1$. Hence if we take $m$ large enough so that $\frac{2M}{m} < \varepsilon$ we guarantee that i. is satisfied. What is more we may write out the density of $(\tilde X, \tilde Y, \tilde Z)$ as 
$$
p_{\tilde X, \tilde Y, \tilde Z}(\tilde x, \tilde y, \tilde z) = \sum_{i,j,k} \frac{m^5}{(2M)^3} \mathbbm{1}(\tilde x \in A_i, \tilde y \in B_j, \tilde z \in C_{ijk}) \PP(X \in A_i, Y \in B_j, Z \in C_k)
$$

\noindent {\bf Step III (showing part ii.).}  Recall that we are assuming that the distribution $p_{X, Y, Z}(x,y,z) \leq L$ for some constant $L > 0$. It is simple to see that the probability that $(X,Y) \in A_i \times B_j$ is bounded as
$$
\int_{A_i \times B_j \times [-M,M]} p_{X, Y, Z}(x,y,z) dx dy dz\leq \frac{L(2M)^3}{m^2}. 
$$
It follows that if we have $n$ observations $(X_i, Y_i, Z_i)_{i \in [n]}$ the probability to have at least two points $(X_k, Y_k)$ and $(X_l, Y_l)$ in one set $A_i \times B_j$ for some $i$ and $j$ is bounded by
$$
\bigg(\frac{L(2M)^3}{m^2}\bigg)^{n} \bigg((m^2)^n - m^2 (m^2 - 1) \ldots (m^2 - n + 1)\bigg) = O\bigg(\frac{(L(2M)^3)^n}{m^2}\bigg),
$$
since the number of all possible arrangements with points belonging to different sets $A_i \times B_j$ is $ m^2 (m^2 - 1) \ldots (m^2 - n + 1)$ while the total number of possible arrangements for the $n$ points is $(m^2)^n$. Denote by $S$ the complement of this event. Note that when $S$ happens all $(\tilde X_i, \tilde Y_i, \tilde Z_i)_{i \in [n]}$ have points $(\tilde X_i, \tilde Y_i)$ in different rectangles and vice versa. 

Next, suppose that $D$ is an arbitrary fixed Borel set. We have
\begin{align*}
\PP( ((\tilde X_i, \tilde Y_i, \tilde Z_i)_{i \in [n]}, U) \in D) & \leq \PP( ((\tilde X_i, \tilde Y_i, \tilde Z_i)_{i \in [n]}, U) \in D \cap (S \times [0,1])) \\
& + \PP( ((\tilde X_i, \tilde Y_i, \tilde Z_i)_{i \in [n]}, U) \not \in S \times [0,1]) 
\end{align*}
We already have a bound on the second term on the RHS above:
\begin{align*}
\PP( ((\tilde X_i, \tilde Y_i, \tilde Z_i)_{i \in [n]}, U) \not \in S \times [0,1]) = O\bigg(\frac{(L(2M)^3)^n}{m^2}\bigg).
\end{align*}

Suppose now that we randomize the assignment on the set $C_{ijk}$. In other words there is a permutation $\pi : [m^2] \mapsto [m^2]$\footnote{Here we use $\pi : [m^2] \mapsto [m^2]$ with a slight abuse of notation. We mean $\pi$ permuting from all ordered pairs of indices $(i,j)$ where $i,j \in [m]$ to all ordered pairs of indices $(k,l)$ where $k,l \in [m]$.} that assigns each pair $A_i, B_j$ to an interval $C_{\pi_{ij}k}$. Denote by $(\tilde X^\pi, \tilde Y^\pi, \tilde Z^\pi)$ the vectors generated in such manner. Clearly all properties described above hold for $(\tilde X_i^\pi, \tilde Y_i^\pi, \tilde Z_i^\pi)_{i \in [n]}$ for any permutation $\pi$. We have that
\begin{align*}
\MoveEqLeft \frac{1}{(m^2)!}\sum_{\pi \in [(m^2)!]} \PP( ((\tilde X^\pi_i, \tilde Y^\pi_i, \tilde Z^\pi_i)_{i \in [n]}, U) \in D \cap (S \times [0,1])) \\
& = \frac{1}{(m^2)!}\sum_{\pi \in [(m^2)!]} \int_{D \cap (S \times [0,1])} \prod_{l \in [n]} \sum_{i,j,k} \frac{m^5}{(2M)^3} \mathbbm{1}(\tilde x_l \in A_i, \tilde y_l \in B_j, \tilde z_l \in C_{\pi_{ij}k}) \PP(X \in A_i, Y \in B_j, Z \in C_k)\\
& \leq \frac{(L m^2)^n}{(m^2)!}\sum_{\pi \in [(m^2)!]} \int_{D \cap (S \times [0,1])} \prod_{l \in [n]} \sum_{i,j,k} \mathbbm{1}(\tilde x_l \in A_i, \tilde y_l \in B_j, \tilde z_l \in C_{\pi_{ij}k})\\
& = \frac{(L m^2)^n}{(m^2)!}\sum_{\{i_l\}_{l \in [n]},\{j_l\}_{l \in [n]},\{k_l\}_{l \in [n]}} \int_{D \cap (S \times [0,1])} \sum_{\pi \in [(m^2)!]} \mathbbm{1}\bigg(\tilde \bx \in \prod_{l \in [n]} A_{i_l}, \tilde \by \in \prod_{l \in [n]} B_{j_l}, \tilde \bz \in \prod_{l \in [n]} C_{\pi_{i_lj_l}k_l}\bigg),
\end{align*}
where in the above summation we have $\{i_l\}_{l \in [n]},\{j_l\}_{l \in [n]},\{k_l\}_{l \in [n]}$ are sequences of $n$ numbers from $[m]$. Since the integration is over the set $D \cap (S \times [0,1])$ all pairs of $(i_l, j_l)$ need to be unique otherwise the integral is $0$. Hence, the summation is over all  $\{i_l\}_{l \in [n]},\{j_l\}_{l \in [n]},\{k_l\}_{l \in [n]}$ sequences of $n$ numbers from $[m]$ for which no two pairs $(i_l, j_l)$ and $(i_k, j_k)$ are the same. Thus to fix $\pi_{i_lj_l}$ for all $(i_l, j_l)_{l \in [n]}$ and permute all others there are $(m^2 - n)!$ permutations. Next note that $\prod_{l \in [n]} C_{k_l} = \prod_{l \in [n]} \sum_{ij} C_{ijk_l}$ contains all unique permutations of $n$ elements (and more) and therefore the summation above is bounded as
\begin{align*}
\MoveEqLeft \frac{(L m^2)^n}{(m^2)!}\sum_{i_l,j_l,k_l} \int_{D \cap (S \times [0,1])} \sum_{\pi \in [(m^2)!]} \mathbbm{1}\bigg(\tilde \bx \in \prod_{l \in [n]} A_{i_l}, \tilde \by \in \prod_{l \in [n]} B_{j_l}, \tilde \bz \in \prod_{l \in [n]} C_{\pi_{i_lj_l}k_l}\bigg) \\
&\leq \frac{(L m^2)^n (m^2 - n)!}{(m^2)!}\sum_{i_l,j_l,k_l} \int_{D \cap (S \times [0,1])} \mathbbm{1}\bigg(\tilde \bx \in \prod_{l \in [n]} A_{i_l}, \tilde \by \in \prod_{l \in [n]} B_{j_l}, \tilde \bz \in \prod_{l \in [n]} C_{k_l}\bigg)\\
& \leq  \frac{(L m^2)^n (m^2 - n)!}{(m^2)!}\mu(D \cap (S \times [0,1])) \leq \frac{(L m^2)^n (m^2 - n)!}{(m^2)!}\mu(D).
\end{align*}
Therefore there exists a permutation $\pi^*$ such that 
$$
\PP( ((\tilde X^{\pi^*}_i, \tilde Y^{\pi^*}_i, \tilde Z^{\pi^*}_i)_{i \in [n]}, U) \in D \cap (S \times [0,1])) \leq \frac{(L m^2)^n (m^2 - n)!}{(m^2)!}\mu(D).
$$
This finishes the proof, by taking $m$ sufficiently large so that 
$$
\PP( ((\tilde X^{\pi^*}_i, \tilde Y^{\pi^*}_i, \tilde Z^{\pi^*}_i)_{i \in [n]}, U) \in D) \leq \frac{(L m^2)^n (m^2 - n)!}{(m^2)!}\mu(D) + O\bigg(\frac{(L(2M)^3)^n}{m^2}\bigg) \leq C \mu(D). 
$$
\end{proof}

\begin{proof}[Proof of Corollary \ref{counting:meas:XY:cont:Z}] The following proof was suggested to us by the AE. Let $d_Z = 1$ for simplicity (the proof obviously extends to the more general case). Suppose there exists a valid test $\psi_n$ for the discrete case and an alternative distribution $Q$ where it has nontrivial power. $\psi_n$ can be converted into a test for the continuous case by applying it to binned versions of $X$ and $Y$. Since independence is preserved under binning, this test continues to be valid (i.e. controls the Type I error for all absolutely continuous null distributions). Next one can create a continuous alternative related to $Q$ by first generating $W \sim Q$ and then reporting a random variable uniformly distributed on the $(X, Y)$ cell associated with $W$. Clearly, such a distribution is absolutely continuous with respect to the Lebesgue measure on $\RR^3$. The modified test will have power against this alternative --- a contradiction to Theorem 3.3.

\end{proof}

\section{Proofs from Section \ref{minimax:lower:bounds:section}}
\label{app:lb}

\subsection{Poissonization}
\label{sec:poisson}
In this section we will demonstrate that the results of the Section \ref{minimax:lower:bounds:section} remain valid under the assumption of Poissonization. Concretely, suppose that instead of a fixed sample size $n$, we are given a random sample of size $N \sim \operatorname{Poi}(n)$. In order for us to redefine the minimax risk, suppose that we are given now a sequence of tests $\{\psi_k\}_{k = 0}^\infty$ indexed by the sample size, where as before each $\psi_k$ is a Borel measurable function such that $\psi_k : \supp(\cD_k) \mapsto [0,1]$. We define the Poissonized minimax risk as 
\begin{align}\label{minimax:risk:poi}
\MoveEqLeft \overline R_n(\cH_0,\overline \cH_0, \cH_1, \varepsilon) = \inf_{\{\psi_k\}_{k = 0}^\infty}\bigg\{\sup_{p \in \cH_0} \sum_{k = 0}^\infty \PP(N = k) \EE_p[\psi_k(\cD_k)] \nonumber \\
& + \sup_{p \in \{p \in \cH_1: \inf_{q \in \overline \cH_0} \|p - q\|_1 \geq \varepsilon\}} \sum_{k = 0}^\infty \PP(N = k) \EE_p [1 - \psi_k(\cD_k)]\bigg\},
\end{align}
As before we also define the corresponding critical radius 
\begin{align}\label{critical:radius:poi}
\overline \varepsilon_n(\cH_0,\overline \cH_0, \cH_1) = \inf\bigg\{\varepsilon : \overline R_n(\cH_0, \overline \cH_0, \cH_1, \varepsilon) \leq \frac{1}{3}\bigg\}. 
\end{align}

We will now state a lemma which relates the minimax risk \eqref{minimax:risk:poi} to the minimax risk \eqref{minimax:risk}. Our arguments are based on the proof of equation (11) in \cite{wu2016minimax}, but for completeness we provide them in 
Appendix~\ref{app:lb}.
\begin{lemma}\label{minimax:risk:to:poi:risk} Suppose that $\cH_0, \cH_1$ are dominated sets of measures by some common $\sigma$-finite measure. We have that
\begin{align*}
\overline R_{2n}(\cH_0,\overline \cH_0, \cH_1, \varepsilon) -  \exp(-(1-\log 2)n)  \leq R_n(\cH_0, \overline \cH_0, \cH_1, \varepsilon) \leq 2 \overline R_{n/2}(\cH_0, \overline \cH_0, \cH_1, \varepsilon).
\end{align*}
\end{lemma}
Using Theorems \ref{first:lower:bound} and \ref{lower:bound:continuous:case} and the right inequality of Lemma \ref{minimax:risk:to:poi:risk} we arrive at the following corollaries which are stated without proof. 
These results show that the lower bounds in Theorems~\ref{first:lower:bound} and \ref{lower:bound:continuous:case}, which were developed for a fixed sample size $n$, continue to hold under Poissonization.

\begin{corollary}\label{first:lower:bound:poi}Let $\overline \cH_0 = \cP_{0,[0,1]}'$. Suppose that $\cH_0$ is either of $\cP_{0,[0,1], \TV}'(L) $, $\cP_{0,[0,1], \TV^2}'(L) $ or $\cP_{0,[0,1], \chi^2}'(L)$, while $\cH_1 = \cQ_{0,[0,1], \TV}'(L)$ for some fixed $L \in \RR^+$. Then we have that, for some absolute constant $c_0 > 0$, the critical radius defined in \eqref{critical:radius:poi} is bounded as
\begin{align*}
\overline \varepsilon_n(\cH_0,\overline \cH_0, \cH_1) \geq c_0 \bigg(\frac{(\ell_1 \ell_2)^{1/5}}{n^{2/5}} \wedge 1\bigg).
\end{align*}
\end{corollary}

\begin{corollary}\label{lower:bound:continuous:case:poi}Let $\overline \cH_0 = \cP_{0,[0,1]^3}$. Suppose that $\cH_0$ is either $\cP_{0,[0,1]^3, \TV}(L)$ or $\cP_{0,[0,1]^3, \chi^2}(L)$, and $\cH_1 = \cQ_{0,[0,1]^3, \TV}(L,s)$ for some fixed $L \in \RR^+$. Then we have that for some absolute constant $c_0 > 0$,
 \begin{align*}
\overline \varepsilon_n(\cH_0,\overline \cH_0, \cH_1) \geq \frac{c_0}{n^{2/7}}.
\end{align*}
\end{corollary}

\begin{proof}[Proof of Lemma \ref{minimax:risk:to:poi:risk}] We first observe that $R_n(\cH_0, \overline \cH_0, \cH_1, \varepsilon) \leq 1$ and furthermore it is clear that $R_n(\cH_0, \overline \cH_0, \cH_1, \varepsilon)$ is a decreasing function in $n$. Hence for $N \sim \operatorname{Poi}(2n)$
\begin{align*}
\overline R_{2n}(\cH_0,\overline \cH_0, \cH_1, \varepsilon) & \leq \sum_{0 \leq k \leq n} \PP(N = k) R_k(\cH_0, \overline \cH_0, \cH_1, \varepsilon) +  \sum_{ k > n} \PP(N = k) R_k(\cH_0, \overline \cH_0, \cH_1, \varepsilon) \\
& \leq \PP(\operatorname{Poi}(2n) \leq n) + R_n(\cH_0, \overline \cH_0, \cH_1, \varepsilon)\\
& \leq \exp(-(1-\log 2)n) + R_n(\cH_0, \overline \cH_0, \cH_1, \varepsilon),
\end{align*}
where in the last inequality we used a Chernoff bound \cite{wu2016minimax}. By Lemma 1 on page 476 of \cite{le2012asymptotic}, we know that 
\begin{align*}
R_n(\cH_0, \overline \cH_0, \cH_1, \varepsilon) = \sup_{\pi_0, \pi_1} \inf_{\psi}\bigg\{ \EE_{p\sim \pi_0} [\psi(\cD_n)] + \EE_{q\sim \pi_1} [1 - \psi(\cD_n)]\bigg\},
\end{align*}
where $\pi_0$ and $\pi_1$ range over prior distributions over the sets $\cH_0$ and $\{p \in \cH_1: \inf_{q \in \overline \cH_0} \|p - q\|_1 \geq \varepsilon\}$. For the second inequality first fix two prior distributions $\pi_0$ and $\pi_1$, and take an arbitrary sequence of tests $\psi_0, \psi_1, \psi_2, \ldots$ indexed by the sample size.

It is unclear whether the sequence 
\begin{align*}
\alpha_k = \EE_{p \sim \pi_0}[\psi_k(\cD_k)] + \EE_{q \sim \pi_1} [1 - \psi_k(\cD_k)],
\end{align*}
is decreasing with $k$, but we can make this sequence monotone in the following way. Define $\{\tilde \alpha_k\}$ recursively as $\tilde \alpha_k = \tilde \alpha_{k-1} \wedge \alpha_k$, and define the corresponding sequence of tests
\begin{align*}
\tilde \psi_k(\cD_k)  = \begin{cases}\tilde \psi_{k-1}(\cD_{k-1}), \mbox{ if }  \tilde \alpha_k = \tilde \alpha_{k-1} \\ \psi_k(\cD_k),  \mbox{ otherwise}\end{cases}
\end{align*}
Take $N \sim \operatorname{Poi}(n/2)$. This sequence of estimates satisfies 
\begin{align*}
\MoveEqLeft \sum_{k = 0}^\infty \PP(N = k) \{\EE_{p \sim \pi_0} [\psi_k(\cD_k)] +  \EE_{q\sim \pi_1} [1 - \psi_k(\cD_k)]\} \\
& \geq \sum_{k = 0}^\infty \PP(N = k) \{\EE_{p \sim \pi_0} [\tilde \psi_k(\cD_k)] +  \EE_{q\sim \pi_1} [1 - \tilde \psi_k(\cD_k)]\}\\
&\geq \frac{1}{2}\bigg\{\EE_{p \sim \pi_0} [\tilde \psi_n(\cD_n)] +  \EE_{q\sim \pi_1} [1 - \tilde \psi_n(\cD_n)]\bigg\}\\
& \geq\frac{1}{2} \inf_{\tilde \psi_n}\bigg\{ \EE_{p \sim \pi_0}   [\tilde \psi_n(\cD_n)] +  \EE_{q\sim \pi_1} [1 - \tilde \psi_n(\cD_n)]\bigg\},
\end{align*}
where we used that $\PP(\operatorname{Poi}(n/2) \geq n) \leq \frac{1}{2}$ by Markov's inequality. Taking a supremum over $\pi_0$ and $\pi_1$ we obtain $R_n(\cH_0, \overline \cH_0, \cH_1, \varepsilon) $ on the right hand side. On the left hand side using that the Bayes risk is upper bounded by the $\sup_{p}$ and $\sup_{q}$ and taking an infimum over all sequences $\psi_0, \psi_1, \psi_2, \ldots$ concludes that
\begin{align*}
\overline R_{n/2}(\cH_0, \overline \cH_0, \cH_1, \varepsilon)  \geq \frac{1}{2} R_n(\cH_0, \overline \cH_0, \cH_1, \varepsilon). 
\end{align*}
\end{proof}

\begin{proof}[Proof of Theorem \ref{first:lower:bound}]
To derive a lower bound we will first show how to obtain multiple distributions which are far from independent by perturbing the uniform discrete distribution. Suppose for simplicity that $\ell_1 = 2 \ell_1'$ and $\ell_2 = 2\ell_2'$ for some integers $\ell_1'$ and $\ell_2'$ (although this simplifies our calculation we will remark how to fix the calculation for the odd case as well).  
 
We first construct a single null hypothesis distribution. Suppose that $Z \sim U[0,1]$. Let the basic null distribution be given by the density $p_{X,Y | Z} (x,y|z) = \frac{1}{\ell_1 \ell_2}$ for each $x, y \in [\ell_1] \times [\ell_2]$ and $z \in [0,1]$. Clearly, this distribution belongs to all three sets of null distributions $\cP_{0,[0,1], \TV}'(L) $, $\cP_{0,[0,1], \TV^2}'(L)$ and $\cP_{0,[0,1], \chi^2}'(L)$.  

Next we will perturb this null distribution $p$ to obtain alternative distributions. Let $\Delta = (\delta_{xy})_{x \in [\ell_1'], y \in [\ell_2']}$ be a matrix of $\pm 1$ numbers $\delta_{xy}$. We create the $\ell_1 \times \ell_2$ matrix $\tilde \Delta$ so that 
\begin{align*}
\tilde \delta_{xy}  & = \delta_{xy} \mbox { for } x,y \in [\ell_1']\times[\ell_2'], \\
\tilde \delta_{xy} & = -\delta_{(x -\ell_1')y} \mbox { for } x > \ell_1', y \in [\ell_2'], \\
\tilde \delta_{xy} & = -\delta_{x(y -\ell_2')} \mbox { for } x \in [\ell_1'], y  > \ell_2', \\
\tilde \delta_{xy} & = \delta_{(x-\ell_1')(y -\ell_2')} \mbox { for } x  > \ell_1', y  > \ell_2'.
\end{align*}

\begin{figure}
\centering
\includegraphics[scale=.2]{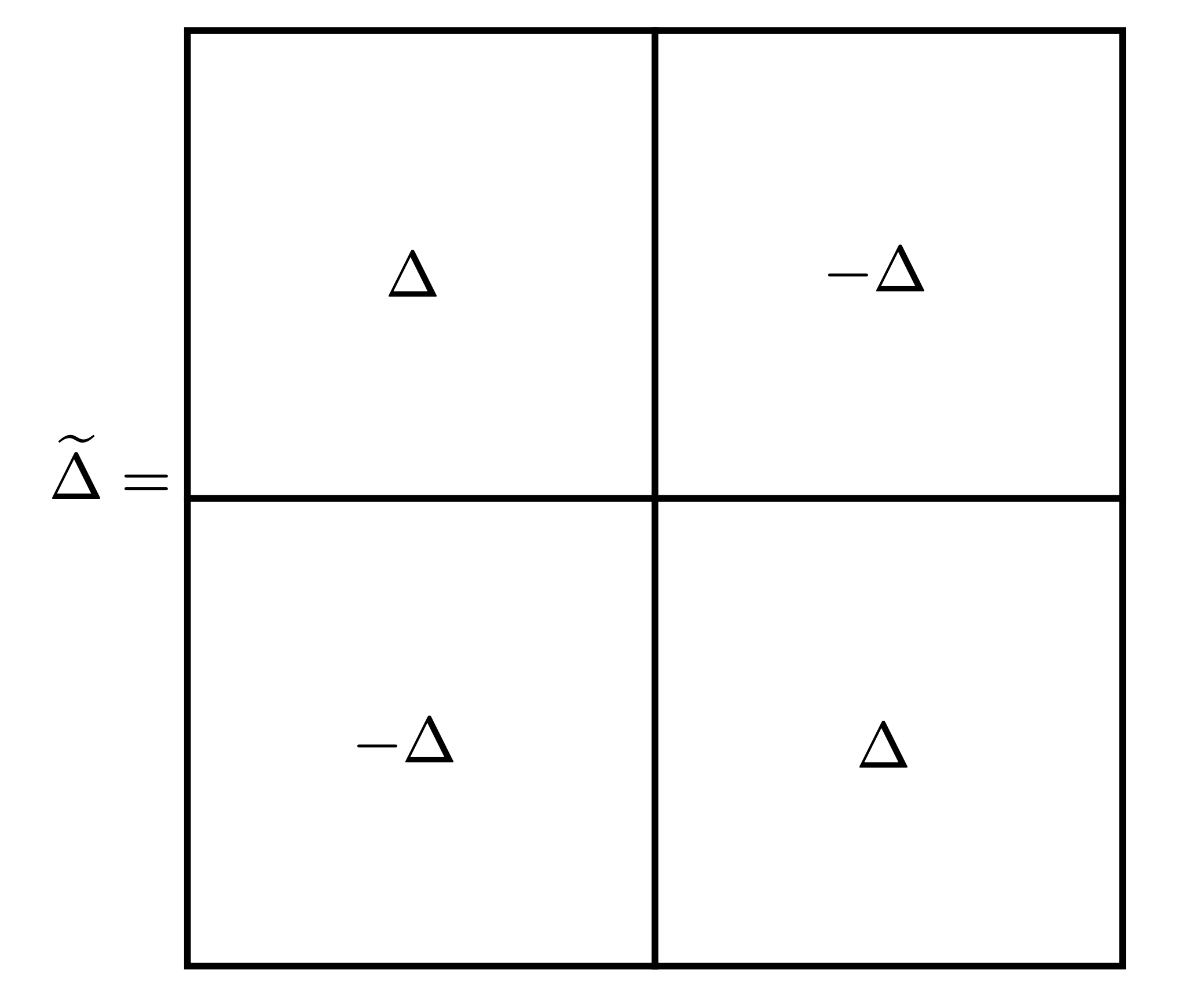}
\caption{The construction of the matrix $\tilde\Delta$ from $\Delta$.}\label{tilde:Delta:figure}
\end{figure}

It is simple to check that all row sums and column sums of the matrix $\tilde \Delta$ are 0 (see also Figure \ref{tilde:Delta:figure}). We perturb the null distribution $p$ using the following procedure: for $x,y \in [\ell_1] \times [\ell_2]$ we take $q_{X,Y|Z}(x,y|z) = \frac{1}{\ell_1 \ell_2} + \tilde \delta_{xy} \eta_\nu(z)$, where 
$$
\eta_\nu(z) = \rho \sum_{j \in [d]} \nu_j h_{j, d}(z),
$$
where $\rho > 0$ is a constant, $d \in \NN$, $\nu_i \in \{-1,+1\}$, and $h_{j, d}(z) = \sqrt{d}h(d z - j + 1)$ for $z \in [(j-1)/d, j/d]$, and $h$ is an infinitely differentiable function supported on $[0,1]$ 
such that $\int h(z) dz = 0$ and $\int h^2(z) dz = 1$. Since the row sums and column sums of $\tilde \Delta$ are $0$, it is simple to verify that the marginals of the distribution remain unchanged under this perturbation, i.e. $q_{X | Z}(x|z) = \sum_{y} q_{X,Y| Z}(x,y|z) = p_{X|Z}(x|z) = \frac{1}{\ell_1}$ and similarly $q_{Y | Z}(y|z) =\sum_{x} q_{X,Y| Z}(x,y|z) = p_{Y|Z}(y|z) = \frac{1}{\ell_2}$. 
We note that in the case that one of $\ell_1$ or $\ell_2$ or both is odd, the fix is to add one row and/or column to the matrix $\Delta$ to be fixed, and reason as in the even case. 

When perturbing, in order to ensure that we create valid probability distributions, we need to satisfy the conditions that 
$$
\frac{1}{\ell_1 \ell_2} - \rho \sqrt{d} \|h\|_{\infty} \geq 0,
$$
and 
$$
\frac{1}{\ell_1 \ell_2} + \rho \sqrt{d} \|h\|_{\infty} \leq 1.
$$
We will ensure this by our choices of $\rho$ and $d$. Next, we need to verify that $q_{X,Y,Z} \in \cQ_{0,[0,1], \TV}'(L)$. We start by showing that $\|q_{X,Y |Z = z} - q_{X,Y|Z = z'}\|_1 \leq L |z - z'|$. We have
\begin{align*}
\|q_{X,Y | Z= z} - q_{X,Y| Z= z'}\|_1 & = \ell_1\ell_2 |\eta_{\nu}(z) - \eta_\nu(z')|. 
\end{align*}
Now the derivative of $\eta_{\nu}(z)$, $|\frac{d}{dz} \eta_{\nu}(z)|$ is bounded by $d^{3/2}\rho \|h'\|_{\infty}$. Thus the above holds when 
\begin{align*}
\|q_{X,Y |Z= z} - q_{X,Y | Z=z'}\|_1  \leq \ell_1\ell_2 \rho d^{3/2} \|h'\|_{\infty} \leq L.
\end{align*}

We let $Z \sim U([0,1])$. Next we will show that the constructed distributions $q_{X,Y,Z}$ are $\varepsilon$ far from being independent, that is we will show that 
\begin{align*}
\inf_{p \in \cP_{0,[0,1]}'} \|q_{X,Y,Z} - p\|_1 \geq \varepsilon,
\end{align*}
for some $\varepsilon > 0$. 
To this end we need the following result which is essentially proved in \cite{canonne2018testing} for the discrete case, here we prove it for continuous $Z$:

\begin{lemma}\label{TV:proj:vs:all} Suppose that a distribution $q$ satisfies
\begin{align*}
\varepsilon = \inf_{p \in \cP_{0,[0,1]}'}\|q_{X,Y,Z} - p\|_1.
\end{align*} 
Then 
\begin{align*}
\|q_{X,Y,Z} - q_{X |Z} q_{Y|Z}q_Z\|_1 \leq 6 \varepsilon.
\end{align*}
\end{lemma}

The proof of Lemma \ref{TV:proj:vs:all} can be found in Appendix~\ref{app:lb}. Suppose that $h$ satisfies $\int |h(z)|dz = c$ for some $0 < c < 1$. We then have that the $L_1$ distance between $q_{X,Y,Z}$ and $q_{X|Z}q_{Y |Z}q_Z$ satisfies
$$
\varepsilon := \EE_Z \ell_1\ell_2 |\eta_\nu(Z)| = \ell_1\ell_2 \sum_{j \in [d]} \int \rho |\nu_j h_{j,k}(z)| dz = \ell_1 \ell_2 \rho \sqrt{d}c,
$$
where in the above we used that $h_{j,k}$ have disjoint support. By Lemma \ref{TV:proj:vs:all} this shows that $q_{X,Y,Z}$ is at least $\varepsilon/6$ from any conditionally independent distribution in $L_1$ distance.

Next, we put uniform priors over $\nu$ and $\Delta$, i.e., the random variables $(\nu_i)_{i \in [d]}$ and $(\delta_{xy})_{x,y \in [\ell_1]\times[\ell_2]}$ are taken as i.i.d. Rademachers. The likelihood ratio is 
\begin{align*}
W = \EE_{\nu, \Delta} \prod_{i = 1}^n (1 + \tilde \delta_{X_i, Y_i} \bar \eta_\nu(Z_i))
\end{align*}
where $\bar \eta_\nu(z) = \ell_1 \ell_2 \eta_\nu(z)$ and the expectation is taken over all Rademacher sequences $\nu$ and $\Delta$.  By a standard argument \cite{ery2018remember, balakrishnan2017hypothesis,ingster2003nonparametric}  the risk of the likelihood ratio (which is the optimal test by Neyman-Pearson's Lemma) is bounded from below by $1 - \frac{1}{2}\sqrt{\Var_0 W}$. Hence it suffices to study $\EE_0(W^2) - 1$ (here $\EE_0$ is the expectation under the null hypothesis). We have
\begin{align*}
\EE_0 W^2 & = \EE_{\nu,\nu',\Delta,\Delta'} \prod_{i = 1}^n \EE_0 (1 + \tilde \delta_{X_i, Y_i}\bar \eta_\nu(Z_i))(1 + \tilde \delta'_{X_i, Y_i}\bar \eta_{\nu'}(Z_i))\\
& =\EE_{\nu,\nu',\Delta,\Delta'}\prod_{i = 1}^n \EE_{Z_i} \sum_{(x,y) \in [\ell_1]\times [\ell_2]}  \frac{(1 + \tilde \delta_{xy}\bar \eta_\nu(Z_i))(1 + \tilde \delta'_{xy}\bar \eta_{\nu'}(Z_i))}{\ell_1\ell_2}\\
& =  \EE_{\nu,\nu',\Delta,\Delta'}\prod_{i = 1}^n  \bigg(  1 + \frac{\EE_{Z_i} \bar \eta_\nu(Z_i)\bar \eta_{\nu'}(Z_i)}{\ell_1 \ell_2}\sum_{(x,y) \in [\ell_1]\times [\ell_2]} \tilde \delta_{xy}\tilde \delta'_{xy} \bigg)\\
&= \EE_{\nu,\nu',\Delta,\Delta'}\prod_{i = 1}^n (1 + 4\ell_1 \ell_2 \rho^2\langle\nu, \nu'\rangle \langle\Delta, \Delta' \rangle)\\
& \leq\EE_{\nu,\nu',\Delta,\Delta'} \exp(4n \ell_1 \ell_2 \rho^2\langle\nu, \nu'\rangle \langle\Delta, \Delta' \rangle)
\end{align*}
In the above $\langle \Delta, \Delta'\rangle = \Tr(\Delta\T \Delta')$ is the standard matrix dot product, while $\langle \nu, \nu' \rangle$ is the standard vector dot product and $\EE_{\nu,\nu',\Delta,\Delta'}$ is the expectation with respect to independent Rademacher draws of $\nu, \nu', \Delta, \Delta'$. Thus,
\begin{align*}
\EE W^2 & \leq \EE_{\Delta, \Delta'} \EE_{\nu, \nu'} [\exp(4n \ell_1 \ell_2 \rho^2\langle\nu, \nu'\rangle \langle\Delta, \Delta' \rangle)] = \EE_{\Delta, \Delta'} \cosh(4n \ell_1 \ell_2 \rho^2 \langle\Delta, \Delta' \rangle)^d \\
& \leq \EE_{\Delta, \Delta'} \exp((4n \ell_1 \ell_2 \rho^2 \langle\Delta, \Delta' \rangle)^2d/2), 
\end{align*}
where we used the inequality $\cosh(x) \leq \exp(x^2/2)$, which can be verified by a Taylor expansion. Next, since when we condition on one value of $\Delta'$ all values of $\langle\Delta, \Delta'\rangle$ happen with the same probability as if we conditioned on any other value of $\Delta'$ we have the identity
\begin{align}
 \EE_{\Delta, \Delta'} \exp((4n \ell_1 \ell_2 \rho^2 \langle\Delta, \Delta' \rangle)^2d/2) =  \EE_{\Delta} \exp((4n \ell_1 \ell_2 \rho^2 \sum_{xy \in [\ell_1']\times[\ell_2']}\delta_{xy})^2d/2)\label{exponential:identity:lower:bound}
\end{align}
Note that if one has i.i.d. Rademacher random variables $\delta_{xy}$ and i.i.d. standard normal variables $W_{xy}$ for any nonnegative integers $a_{xy}$ for  $ x,y \in [\ell_1']\times [\ell_2']$ one has
$$
\EE \prod_{xy \in [\ell_1']\times [\ell_2']} \delta^{a_{xy}}_{xy}\leq \EE \prod_{xy \in [\ell_1']\times [\ell_2']} W^{a_{xy}}_{xy}
$$
Expanding the exponential function in \eqref{exponential:identity:lower:bound} one can control all moments of $\delta_{xy}$ using the inequality above with corresponding moments of $W_{xy}$. Thus we conclude
\begin{align}
\EE_{\Delta} \exp((4n \ell_1 \ell_2 \rho^2 \sum_{xy \in [\ell_1']\times[\ell_2']}\delta_{xy})^2d/2)\leq \EE_{\bW} \exp((4n \ell_1 \ell_2 \rho^2 \sum_{xy \in [\ell_1']\times[\ell_2']}W_{xy})^2d/2)
\end{align}
The random variable $\sum_{xy \in [\ell_1']\times[\ell_2']}W_{xy} \sim N(0, \ell_1'\ell_2')$ and therefore \\$(\sum_{xy \in [\ell_1']\times[\ell_2']}W_{xy})^2/\ell_1'\ell_2' := \chi^2$ has a $\chi^2(1)$ distribution. We have 
\begin{align}
 \EE_{\bW} \exp((4n \ell_1 \ell_2 \rho^2 \sum_{xy \in [\ell_1']\times[\ell_2']}W_{xy})^2d/2) \leq  \EE_{\chi^2} \exp((4n \ell_1 \ell_2 \rho^2)^2\ell_1'\ell_2'\chi^2d/2).
\end{align}
Suppose now that $(4n \ell_1 \ell_2 \rho^2)^2d\ell_1'\ell_2' < 1$. The above is the mgf of a chi-squared random variable and hence equals to
\begin{align*}
\sqrt{\frac{1}{1 - (4n \ell_1 \ell_2 \rho^2)^2d\ell_1'\ell_2'}}.
\end{align*}
This quantity can be made arbitrarily close to $1$ provided that $(4n \ell_1 \ell_2 \rho^2)^2d\ell_1'\ell_2' $ is small. Based on this select $\frac{1}{d} \asymp \frac{(\ell_1\ell_2)^{1/5}}{n^{2/5}} \wedge 1$, $\rho \asymp \frac{1}{\ell_1 \ell_2 d^{3/2}}$ for some sufficiently small constants. This ensures $\frac{1}{\ell_1 \ell_2} - \rho \sqrt{d} \|h\|_{\infty} \geq 0$ and $\ell_1 \ell_2 \rho d^{3/2} \|h'\|_{\infty} \leq L$. With these choices we obtain that the critical radius is bounded from below by a constant times $\frac{1}{d} \asymp \frac{(\ell_1\ell_2)^{1/5}}{n^{2/5}} \wedge 1$. 
\end{proof}

\begin{proof}[Proof of Lemma \ref{TV:proj:vs:all}] We first start by showing several bounds on the $L_1$ norm between two arbitrary distributions $p_{}$ and $p'_{}$ from $\cE_{0,[0,1]}'$. We have that the $L_1$ norm between the distributions is 
\begin{align*}
\|p - p'\|_1 = \int_{z \in [0,1]} \sum_{x,y \in [\ell_1]\times[\ell_2]} |p_{Z}(z) p_{X,Y| Z}(x,y| z) - p'_{Z}(z) p'_{X,Y| Z}(x,y|z)|dz.
\end{align*}
Using the triangle inequality we conclude that 
\begin{align*}
\|p - p'\|_1 & \geq \int_{z \in [0,1]}  \bigg|\sum_{x,y \in [\ell_1]\times[\ell_2]}p_{Z}(z)p_{X,Y| Z}(x,y| z) - p'_{Z}(z) p'_{X,Y| Z}(x,y| z)\bigg| dz \\
& = \int_{z \in [0,1]}  |p_{Z}(z) - p'_{Z}(z)| dz = \|p_Z- p'_Z\|_1.
\end{align*}
Next we observe the following identities
\begin{align*}
\MoveEqLeft \|p - p'\|_1 \\
& = \int_{z \in [0,1]} \sum_{x,y \in [\ell_1]\times[\ell_2]} |p_{Z}(z) (p_{X,Y| Z}(x,y| z) - p'_{X,Y| Z}(x,y| z)) + (p_Z(z) - p'_{Z}(z)) p'_{X,Y| Z}(x,y| z)|dz\\
& \geq \int_{z \in [0,1]} \sum_{x,y \in [\ell_1]\times[\ell_2]} p_{Z}(z) |p_{X,Y| Z}(x,y| z) - p'_{X,Y| Z}(x,y| z)| - |p_Z(z) - p'_{Z}(z)| p'_{X,Y| Z}(x,y| z)dz\\
& = \EE_Z \|p_{X,Y|Z} - p'_{X,Y|Z}\|_1  - \|p_Z - p'_Z\|_1. 
\end{align*}
Combining the last two identities we obtain that 
\begin{align}
\max(\EE_Z \|p_{X  |Z}- p'_{X |Z}\|_1,  \EE_Z \|p_{ Y|Z}- p'_{ Y|Z}\|_1) & \leq \EE_Z \|p_{X,Y|Z} - p'_{X,Y|Z}\|_1   \nonumber \\
&\leq \|p_Z - p'_Z\|_1 + \|p - p'\|_1 \leq 2\|p - p'\|_1,\label{tv:bound:marginalized:dists}
\end{align}
where the first inequality follows by the triangle inequality. 
Next, reversing the triangle inequality from before we obtain, 
\begin{align*}
\MoveEqLeft \|p - p'\|_1 \\
& = \int_{z \in [0,1]} \sum_{x,y \in [\ell_1]\times[\ell_2]} |p_{Z}(z) (p_{X,Y| Z}(x,y| z) - p'_{X,Y| Z}(x,y| z)) + (p_Z(z) - p'_{Z}(z)) p'_{X,Y| Z}(x,y| z)|dz\\
& \leq \int_{z \in [0,1]} \sum_{x,y \in [\ell_1]\times[\ell_2]} p_{Z}(z) |p_{X,Y| Z}(x,y| z) - p'_{X,Y| Z}(x,y| z)| + |p_Z(z) - p'_{Z}(z)| p_{X,Y| Z}(x,y| z)dz\\
& = \EE_Z \|p_{X,Y|Z} - p'_{X,Y|Z}\|_1  + \|p_Z-  p'_Z\|_1.
\end{align*}
Next, suppose that $q_{}$ is a distribution which is $\varepsilon$ far from being conditionally independent. We will denote the distribution $q_{X | Z} q_{Y | Z} q_Z$ with $\tilde q$. Let $p'$ be a distribution which is $\varepsilon$ away from $q$ in $\|\cdot \|_1$ and $p'$ is conditionally independent (if such a distribution does not exist we can take a sequence that approximates the infimum). We have
\begin{align*}
\|q- \tilde q\|_1 \leq \|q- p'\|_1 + \|\tilde q - p'\|_1 \leq \varepsilon + \|\tilde q - p'\|_1
\end{align*}
We now handle the second term
\begin{align*}
\|\tilde q - p'\|_1 & \leq  \EE_Z \|\tilde q_{X,Y|Z} - p'_{X,Y|Z}\|_1  + \|q_Z- p'_Z\|_1\\
&\leq \EE_Z \|\tilde q_{X,Y|Z} - p'_{X,Y|Z}\|_1 + \varepsilon 
\end{align*}
Using that TV is sub-additive on product distributions we now have
\begin{align*}
\EE_Z \|\tilde q_{X,Y|Z}- p'_{X,Y|Z}\|_1 & \leq \EE_Z \|\tilde q_{X|Z}- p'_{X|Z}\|_1 + \EE_Z \|\tilde q_{Y|Z} - p'_{Y|Z}\|_1\\
& = \EE_Z \|q_{X|Z} - p'_{X|Z}\|_1 + \EE_Z \|q_{Y |Z}- p'_{Y|Z}\|_1\\
& \leq 4 \varepsilon,
\end{align*}
where we used \eqref{tv:bound:marginalized:dists} in the last bound. We conclude that 
$$
\|q- \tilde q\|_1 \leq 6 \varepsilon,
$$
which completes the proof. 
\end{proof}

\begin{proof}[Proof of Theorem \ref{lower:bound:continuous:case}]  Suppose $(X,Y,Z) \in [0,1]^3$ are three variables with a joint density with respect to the Lebesgue measure in $[0,1]^3$. Under the null hypothesis we specify the distribution as $p_{X,Y,Z}(x,y,z) = 1$ for all $(x,y,z) \in [0,1]^3$, or in other words the three variables have independent uniform distributions on $[0,1]$. Clearly this distribution belongs to the sets $\cP_{0,[0,1]^3,\TV}(L)$ and $\cP_{0,[0,1]^3,\chi^2}(L)$. Under the alternative hypothesis we specify the distribution as
\begin{align*}
q_{X,Y|Z}(x,y|z) = 1 + \gamma_{\Delta}(x,y)\eta_{\nu}(z),
\end{align*}
where as in the proof of Theorem \ref{first:lower:bound} 
$$
\eta_{\nu}(z) = \rho \sum_{j \in [d]} \nu_j h_{j, d}(z),
$$
where $\rho > 0$ is a constant, $d \in \NN$, $\nu_i \in \{-1,+1\}$ ,and $h_{j, d}(z) = \sqrt{d}h(d z - j + 1)$ for $z \in [(j-1)/d, j/d]$, and $h$ is an infinitely differentiable function supported on $[0,1]$ 
such that $\int h(z) dz = 0$ and $\int h^2(z) dz = 1$. Furthermore we take
$$
\gamma_{\Delta}(x,y) =  \rho^2 \sum_{j \in [d']}\sum_{i \in [d']} \delta_{ij} h_{i, d'}(x)h_{j, d'}(y),\footnotemark
$$
\footnotetext{Here $\Delta = \{\delta_{ij}\}_{i,j\in[d]}$.}and we let the marginal distribution of $Z$ be uniform on $[0,1]$ (i.e. we let $q_Z = p_Z \equiv 1$). In order for this perturbation to be meaningful we need that $1 \geq \sqrt{(d')^2 d} \|h\|_{\infty}^3 \rho^3$. It is simple to check that $\int_{[0,1]^2} q_{X,Y|Z}(x,y|z) dx dy = 1$. Let us now check what are the marginals of such a distribution conditioned on $z$. We have
$$
q_{X|Z}(x|z) = \int_{[0,1]} q_{X,Y | Z}(x,y|z) dy = 1 +  \eta_{\nu}(z)\int_{[0,1]} \gamma_{\Delta}(x,y) dy = 1.
$$
Similarly $\int_{[0,1]} q_{X,Y|Z}(x,y|z) dx = 1$. It is therefore clear that $q_{X|Z}(x|z) q_{Y|Z}(y|z)  \in \cH^{2,s}(L)$ for any $L$. We now check how far away is the distribution $q_{X,Y,Z}(x,y,z)= q_{X,Y|Z}(x,y|z)q_Z(z) = q_{X,Y|Z}(x,y|z)p_Z(z) = q_{X,Y|Z}(x,y|z)$ with respect to $p_{X,Y,Z}(x,y,z)$ (note that $p_{X,Y,Z} = q_{X | Z} q_{Y|Z} q_Z$) in total variation. We have
\begin{align*}
\|q_{X,Y|Z}q_Z - p_{X,Y,Z}\|_1 &= \int_{[0,1]^3} |\gamma_{\Delta}(x,y)\eta_{\nu}(z)| dx dy dz \\
&= \int_{[0,1]}  \rho \sum_{j \in [d']} |h_{j,d'}(x)| dx \int_{[0,1]}\rho \sum_{j \in [d']} |h_{j,d'}(y)|  dy \int_{[0,1]} \rho \sum_{j \in [d]}|\eta_{j,d}(z)| dz.
\end{align*} 
Calculating each of the above integrals and multiplying them yields 
\begin{align*}
\|q_{X,Y|Z}p_Z - p_{X,Y,Z}\|_1 = \|q_{X,Y,Z} - q_{X | Z} q_{Y|Z} q_Z\|_1 = \sqrt{d (d')^2} \rho^3 c^3,
\end{align*}
where $c = \int_{[0,1]} |h(x)| dx$. Using Lemma \ref{TV:proj:vs:all} (here we use this lemma with a slight abuse of notation since the lemma is only valid for discrete $X,Y$ and continuous $Z$, but the same proof extends to the continuous case) we have that 
$$\inf_{p \in \cP_{0,[0,1]^3}}\|q - p\|_1 \geq \frac{\sqrt{d (d')^2} \rho^3 c^3}{6}.$$ Next we check that the TV between the distributions $(X,Y | Z = z)$ and $(X,Y| Z = z')$ is Lipschitz.
\begin{align*}
\int_{[0,1]^2}|\gamma_{\Delta}(x,y)| |\eta_{\nu}(z) - \eta_{\nu}(z')| dx dy = \sqrt{(d')^2} \rho^2 c^2  |\eta_{\nu}(z) - \eta_{\nu}(z')|. 
\end{align*}
We now observe that the derivative of $\eta_{\nu}(z)$ is bounded by $\sqrt{d}d\rho \|h'\|_{\infty}$, therefore the above is bounded by
$$
\sqrt{(d')^2} \rho^2 c^2 \sqrt{d}d\rho \|h'\|_{\infty} |z - z'|. 
$$
Next we check that $q(x, y | z)$ belongs to the H\"{o}lder class in $x$ and $y$. We have that 
\begin{align*}
\MoveEqLeft \bigg|\frac{\partial^k}{\partial x^k} \frac{\partial^{\lfloor s \rfloor - k}}{\partial y^{\lfloor s \rfloor - k}}\gamma_{\Delta}(x,y)\eta_{\nu}(z) - \frac{\partial^k}{\partial x^k} \frac{\partial^{\lfloor s \rfloor - k}}{\partial y^{\lfloor s \rfloor - k}}\gamma_{\Delta}(x',y')\eta_{\nu}(z)\bigg| \\
& \leq \sqrt{d}\rho \|h\|_{\infty} \bigg|\frac{\partial^k}{\partial x^k} \frac{\partial^{\lfloor s \rfloor - k}}{\partial y^{\lfloor s \rfloor - k}}\gamma_{\Delta}(x,y) - \frac{\partial^k}{\partial x^k} \frac{\partial^{\lfloor s \rfloor - k}}{\partial y^{\lfloor s \rfloor - k}}\gamma_{\Delta}(x',y')\bigg|\\
& = \sqrt{d}\rho \|h\|_{\infty}\rho^2 (\sqrt{d'})^2 (d')^{\lfloor s \rfloor} \bigg| \sum_{i,j \in [d']} \delta_{ij}\bigg[ h^{(k)}(d'x - i + 1) h^{(\lfloor s \rfloor -k)}(d'y - j + 1) \\
& -h^{(k)}(d'x' - i + 1) h^{(\lfloor s \rfloor -k)}(d'y' - j + 1)  \bigg]\bigg|.
\end{align*}
Suppose now that $x \in \bigg[\frac{i_x - 1}{d'}, \frac{i_x}{d'}\bigg], y \in \bigg[\frac{j_y - 1}{d'}, \frac{j_y}{d'}\bigg]$, and $x' \in \bigg[\frac{i_{x'} - 1}{d'}, \frac{i_{x'}}{d'}\bigg], y' \in \bigg[\frac{j_{y'} - 1}{d'}, \frac{j_{y'}}{d'}\bigg]$. Therefore the above summation can be bounded as
\begin{align*}
& \sqrt{d}\rho \|h\|_{\infty}\rho^2 (\sqrt{d'})^2 (d')^{\lfloor s \rfloor} \bigg[\bigg| h^{(k)}(d'x - i_x + 1) h^{(\lfloor s \rfloor -k)}(d'y - j_y + 1) -h^{(k)}(d'x' - i_x + 1) h^{(\lfloor s \rfloor -k)}(d'y' - j_y + 1)\bigg|\\
&+  \bigg| h^{(k)}(d'x - i_{x'} + 1) h^{(\lfloor s \rfloor -k)}(d'y - j_{y'} + 1) -h^{(k)}(d'x' - i_{x'} + 1) h^{(\lfloor s \rfloor -k)}(d'y' - j_{y'} + 1)\bigg|  \bigg].
\end{align*}
Next we will handle the first expression in the bracket above:
\begin{align*}
& \leq |h^{(k)}(d'x - i_x + 1) -h^{(k)}(d'x' - i_x + 1)| |h^{(\lfloor s \rfloor -k)}(d'y - j_y + 1)| \\
& +  |h^{(k)}(d'x' - i_x + 1)| |h^{(\lfloor s \rfloor -k)}(d'y - j_y + 1) - h^{(\lfloor s \rfloor -k)}(d'y' - j_y + 1)|\\
& \leq d' \|h^{(k + 1)}\|_\infty |x - x'| \|h^{(\lfloor s \rfloor -k)}\|_{\infty} \wedge (2 \|h^{(k)}\|_\infty\|h^{(\lfloor s \rfloor -k)}\|_{\infty})  \\
& + d' \|h^{(\lfloor s \rfloor -k) + 1}\|_\infty |y - y'| \|h^{k}\|_{\infty} \wedge (2\|h^{(\lfloor s \rfloor -k)}\|_\infty \|h^{k}\|_{\infty})\\
& \leq C(1 \wedge d' \sqrt{(x - x')^2 + (y - y')^2})\\
& \leq C (d' \sqrt{(x - x')^2 + (y - y')^2})^{s - \lfloor s \rfloor},
\end{align*}
where in the last inequality we used that $(1 \wedge u)^a \leq u^a$ for $u > 0$ and $0 \leq a \leq 1$. We can handle the second expression in the bracket above in a similar way. 

In addition it is clear that any lower order $k \leq \lfloor s \rfloor$ partial derivatives with respect to $x$ and $y$ of $q(x,y|z)$ are bounded by $\sqrt{d}\rho^3 \|h\|_{\infty} \sqrt{d'}^2 (d')^{k} C$ for some constant $C$ which will depend on the function $h$.

It therefore suffices that $\sqrt{d}\rho^3 (\sqrt{d'})^2 (d')^s$ to be smaller than a constant and we will have both conditions satisfied. Now we write down the likelihood ratio between the null and the alternative mixing over all choices of Rademacher vector and matrix $\nu, \Delta$:
\begin{align*}
W = \EE_{\nu,\Delta} \prod_{i = 1}^n (1 + \gamma_{\Delta}(X_i, Y_i) \eta_{\nu}(Z_i)).
\end{align*}

The second moment of $W$ is 
\begin{align*}
\EE W^2 & = \EE_{\nu, \nu',\Delta,\Delta'}  \prod_{i = 1}^n \EE_0(1 + \gamma_{\Delta}(X_i, Y_i) \eta_{\nu}(Z_i))(1 +\gamma_{\delta'}(X_i, Y_i) \eta_{\nu'}(Z_i))\\
& = \EE_{\nu, \nu',\Delta,\Delta'}  \prod_{i = 1}^n (1 + \EE_0 \gamma_{\Delta}(X_i, Y_i) \eta_{\nu}(Z_i)\gamma_{\delta'}(X_i, Y_i) \eta_{\nu'}(Z_i)),
\end{align*}
where the above follows from the fact that $\EE_0 \eta_{\nu}(Z_i) = 0$ (and that $X_i$ and $Y_i$ are independent of $Z_i$ under the null hypothesis). Continuing the identities yields
\begin{align*}
\EE W^2 & =  \EE_{\nu, \nu',\Delta,\Delta'}  \prod_{i = 1}^n (1 + \EE_0 \gamma_{\Delta}(X_i, Y_i) \gamma_{\delta'}(X_i, Y_i) \EE_0\eta_{\nu}(Z_i) \eta_{\nu'}(Z_i))\\
& = \EE_{\nu, \nu',\Delta,\Delta'}  \prod_{i = 1}^n (1 + \rho^6 \langle\Delta, \Delta'\rangle\langle \nu,\nu'\rangle),
\end{align*}
where $\langle\Delta,\Delta' \rangle= \Tr(\Delta\T \Delta')$. From here the proof can continue as in Theorem \ref{first:lower:bound}. The final expression that needs to be smaller than a constant is $(n \rho^6)^2 d (d')^2$. Set $d \asymp n^{2s/(5s + 2)}$, $d' = d^{1/s}$, $\rho^3 \asymp d^{-(3/2 + 1/s)}$. This results in a rate $\asymp 1/d \asymp n^{-2s/(5s + 2)}$. 
\end{proof}

\section{Proofs from Section \ref{upper:bound:section}}
\label{app:upper}

This section contains the proofs of Sections \ref{fixed:l1:l2:section}, \ref{scaling:l1:l2:section} and \ref{cont:case:upper:bounds:section}.

\subsection{Proofs from Section \ref{fixed:l1:l2:section}}

\begin{proof}[Proof of Lemma \ref{U:stat:variance:lemma}] According to Section 12 of \cite{van2000asymptotic} the variance of the U-statistic \eqref{U:stat} equals to 
\begin{align*}
\MoveEqLeft O\bigg(\frac{1}{\sigma}\bigg) \Cov(h_{ijkl}, h_{ij'k'l'}) + O\bigg(\frac{1}{\sigma^2}\bigg)\Cov(h_{ijkl}, h_{ijk'l'}) \\
&+ O\bigg(\frac{1}{\sigma^3}\bigg)\Cov(h_{ijkl}, h_{ijkl'}) + O\bigg(\frac{1}{\sigma^4}\bigg)\Var(h_{ijkl}),
\end{align*}
where $i, j, k, l, j', k', l' \in [\sigma]$ are distinct indices of observations (it is ok if the sample size $\sigma$ is smaller than $7$, then the first terms simply do not contribute). We will now argue that 
\begin{align*}
\Cov(h_{ijkl}, h_{ij'k'l'}) \leq C \EE[U(\cD)] \max(\|p_{X',Y'}\|_2, \|p_{X'}p_{Y'}\|_2),
\end{align*}
and that all other covariances are bounded by $C \max(\|p_{X',Y'}\|^2_2, \|p_{X'}p_{Y'}\|^2_2)$ which will complete the proof.
In order to bound the first term from above it suffices to control the following expression
\begin{align*}
\MoveEqLeft \EE \sum_{x, y} \phi_{\pi_1\pi_2}(xy)\phi_{\pi_3 \pi_4}(xy)  \sum_{x',y'} \phi_{\pi_1'\pi_2'}(x'y')\phi_{\pi_3' \pi_4'}(x'y') \\
&- \EE \sum_{x, y} \phi_{\pi_1\pi_2}(xy)\phi_{\pi_3 \pi_4}(xy)  \EE \sum_{x',y'} \phi_{\pi_1'\pi_2'}(x'y')\phi_{\pi_3' \pi_4'}(x'y'),
\end{align*}
where $\pi$ is a permutation of $i,j,k,l$ and $\pi'$ is a permutation of $i,j',k',l'$. We can rewrite the above as
\begin{align}
\sum_{x,y,x',y'} \bigg\{\EE  \phi_{\pi_1\pi_2}(xy)\phi_{\pi_3 \pi_4}(xy) \phi_{\pi_1'\pi_2'}(x'y')\phi_{\pi_3' \pi_4'}(x'y') - (\EE \phi_{\pi_1\pi_2}(xy))^2(\EE \phi_{\pi_1'\pi_2'}(x'y'))^2 \bigg\}\label{equation:going:back:to}
\end{align}
where we used that by independence $\EE \phi_{\pi_1\pi_2}(xy)\phi_{\pi_3 \pi_4}(xy) = \EE \phi_{\pi_1\pi_2}(xy)\EE\phi_{\pi_3 \pi_4}(xy) = (\EE \phi_{\pi_1\pi_2}(xy))^2$ and similarly for $\pi'$. Without loss of generality suppose that $i$ is some of $\pi_1,\pi_2$ and $\pi_1', \pi_2'$.
Going back to equation \eqref{equation:going:back:to} we have
\begin{align*}
\MoveEqLeft \sum_{x,y,x',y'} \bigg\{\EE  \phi_{\pi_1\pi_2}(xy)\phi_{\pi_3 \pi_4}(xy) \phi_{\pi_1'\pi_2'}(x'y')\phi_{\pi_3' \pi_4'}(x'y') - (\EE \phi_{\pi_1\pi_2}(xy))^2(\EE \phi_{\pi_1'\pi_2'}(x'y'))^2 \bigg\}\\
& \leq \sum_{x,y,x',y'} \bigg\{\EE  \phi_{\pi_1\pi_2}(xy)\phi_{\pi_3 \pi_4}(xy) \phi_{\pi_1'\pi_2'}(x'y')\phi_{\pi_3' \pi_4'}(x'y')\bigg\}\\
& = \sum_{x,y,x',y'} \bigg\{\EE (\phi_{\pi_1\pi_2}(xy)\phi_{\pi_1'\pi_2'}(x'y'))\EE \phi_{\pi_3 \pi_4}(xy) \EE\phi_{\pi_3' \pi_4'}(x'y') \bigg\}\\
&\leq \sqrt{\sum_{x,y,x',y'} \{\EE (\phi_{\pi_1\pi_2}(xy) \phi_{\pi_1'\pi_2'}(x'y'))\}^2} \sqrt{\sum_{x,y,x',y'} \{\EE \phi_{\pi_3 \pi_4}(xy)\}^2 \{\EE\phi_{\pi_3' \pi_4'}(x'y')\}^2}\\
& = \sqrt{\sum_{x,y,x',y'} \{\EE (\phi_{\pi_1\pi_2}(xy) \phi_{\pi_1'\pi_2'}(x'y'))\}^2} \EE[U(\cD)]
\end{align*}
where the next to last step follows from Cauchy-Schwarz. To this end we formalize the following lemma.
\begin{lemma}\label{simple:lemma:discrete} For two random variables $A,B \in \{\pm1,0\}$ we have 
$$
(\EE A B)^2 \leq \EE[|B| | |A| = 1] (\EE |A|)^2.
$$
\end{lemma}
Now we apply Lemma \ref{simple:lemma:discrete} to the first term on the RHS with $A = \phi_{\pi_1\pi_2}(xy)$ and $B = \phi_{\pi_1'\pi_2'}(x'y')$ noting that $\phi$ can only take values $\{\pm1, 0\}$. We obtain
\begin{align*}
\sum_{x,y,x',y'} \{\EE (\phi_{\pi_1\pi_2}(xy) \phi_{\pi_1'\pi_2'}(x'y'))\}^2 \leq \sum_{x,y,x',y'} \EE [|\phi_{\pi_1'\pi_2'}(x'y'))| | |\phi_{\pi_1\pi_2}(xy)| = 1] (\EE |\phi_{\pi_1\pi_2}(xy)|)^2
\end{align*}
Note that 
\begin{align*}|\phi_{\pi_1\pi_2}(xy)| 
& \leq \mathbbm{1}(X'_{\pi_1} = x, Y'_{\pi_1} = y) + \mathbbm{1}(Y'_{\pi_2} = y) \mathbbm{1}(X'_{\pi_1} = x),
\end{align*}
and a similar inequality holds for $|\phi_{\pi_1'\pi_2'}(x'y')|$. Thus 
\begin{align*}
\sum_{x',y'} |\phi_{\pi_1'\pi_2'}(x'y')| \leq \sum_{x',y'}\mathbbm{1}(X'_{\pi_1'} = x', Y'_{\pi_1'} = y') + \mathbbm{1}(Y'_{\pi_2'} = y') \mathbbm{1}(X'_{\pi_1'} = x') \leq 2.
\end{align*}
Hence 
\begin{align*}
\MoveEqLeft \sum_{x, y} \EE [\sum_{x',y'}|\phi_{\pi_1'\pi_2'}(x'y'))| | |\phi_{\pi_1\pi_2}(xy)| = 1] (\EE |\phi_{\pi_1\pi_2}(xy)|)^2 \leq 2\sum_{x, y}(\EE |\phi_{\pi_1\pi_2}(xy)|)^2 \\
& \leq 2 \sum_{x, y} (p_{X',Y'}(x,y) + p_{X'}(x)p_{Y'}(y))^2 \leq 4 (\|p_{X',Y'}\|^2_2 + \|p_{X'}p_{Y'}\|_2^2),
\end{align*}
which is what we wanted to show.  Now it remains to prove Lemma \ref{simple:lemma:discrete}.

\begin{proof}[Proof of Lemma \ref{simple:lemma:discrete}]
We have
\begin{align*}
(\EE AB)^2 \leq (\EE |AB|)^2 = (\EE[|B| | |A| = 1] \EE |A|)^2 \leq \EE[|B| | |A| = 1] (\EE |A|)^2,
\end{align*}
where in the last step we used that $\EE [|B| | |A| = 1] \leq 1$.
\end{proof}

Now we will show how to bound any higher order terms: $\Cov( h_{ijkl}, h_{ijk'l'})$ where $i, j, k,\allowbreak l, k', l' \in [\sigma]$ (and it's possible for $k', l'$ to be equal to $k$ or $l$). To bound these terms, following the same strategy as before, it suffices to control the quantity
\begin{align*}
\sum_{x,y,x',y'} \bigg\{\EE  \phi_{\pi_1\pi_2}(xy)\phi_{\pi_3 \pi_4}(xy) \phi_{\pi_1'\pi_2'}(x'y')\phi_{\pi_3' \pi_4'}(x'y')\bigg\}
\end{align*}
We will now use the fact that for random variables $A,B \in \{\pm 1,0\}$
$$
\EE AB \leq \EE|AB| = \EE[|B| | |A| = 1] \EE |A|,
$$
where $A = \phi_{\pi_1\pi_2}(xy)\phi_{\pi_3 \pi_4}(xy)$ and $B = \phi_{\pi_1'\pi_2'}(x'y')\phi_{\pi_3' \pi_4'}(x'y')$. We have
\begin{align*}
\MoveEqLeft \sum_{x,y,x',y'} \bigg\{\EE  \phi_{\pi_1\pi_2}(xy)\phi_{\pi_3 \pi_4}(xy) \phi_{\pi_1'\pi_2'}(x'y')\phi_{\pi_3' \pi_4'}(x'y')\bigg\}  \\
&\leq \sum_{x,y,x',y'} \EE[|\phi_{\pi_1'\pi_2'}(x'y')\phi_{\pi_3' \pi_4'}(x'y')| | |\phi_{\pi_1\pi_2}(xy)\phi_{\pi_3 \pi_4}(xy)| = 1] \EE |\phi_{\pi_1\pi_2}(xy)\phi_{\pi_3 \pi_4}(xy)|
\end{align*}
We now use that as we saw before
$$
|\phi_{\pi_1\pi_2}(xy)| \leq \mathbbm{1}(X'_{\pi_1} = x, Y'_{\pi_1} = y) + \mathbbm{1}(Y'_{\pi_2} = y) \mathbbm{1}(X'_{\pi_1} = x),
$$
and analogously for the others. Furthermore $|\phi_{\pi_1\pi_2}(xy)| \leq 1$. We have 
\begin{align*}
\MoveEqLeft \sum_{x',y'}\EE[|\phi_{\pi_1'\pi_2'}(x'y')\phi_{\pi_3' \pi_4'}(x'y')| | |\phi_{\pi_1\pi_2}(xy)\phi_{\pi_3 \pi_4}(xy)| = 1] \leq \sum_{x',y'}\EE[|\phi_{\pi_1'\pi_2'}(x'y')| |\phi_{\pi_1\pi_2}(xy)\phi_{\pi_3 \pi_4}(xy)| = 1]\\
& \leq  2.
\end{align*}
Next by independence, 
\begin{align*}
\MoveEqLeft\sum_{x, y}\EE |\phi_{\pi_1\pi_2}(xy)\phi_{\pi_3 \pi_4}(xy)| = \sum_{x, y} \EE |\phi_{\pi_1\pi_2}(xy)| \EE|\phi_{\pi_3 \pi_4}(xy)| \\
& \leq \sum_{x, y} (p_{X',Y'}(x,y) + p_{X'}(x)p_{Y'}(y))^2 \\
& \leq \sum_{x, y}2(p_{X',Y'}^2(x,y) + p_{X'}^2(x)p_{Y'}^2(y)) = 2 (\|p_{X',Y'}\|_2^2 + \|p_{X'}p_{Y'}\|_2^2).
\end{align*}
This completes the proof.
\end{proof}

Below we will prove the following version of Theorem \ref{main:theorem:finite:discrete:XY}
\begin{theorem}[Finite Discrete $X$, $Y$ Upper Bound] \label{main:theorem:finite:discrete:XY:Poi} Set $d = \lceil n^{2/5} \rceil$ and let $\tau = \zeta n^{1/5}$ for a sufficiently large absolute constant $\zeta$ (depending on $L$). 
Finally, suppose that $\varepsilon \geq c n^{-2/5}$, for a sufficiently large constant $c$ (depending on $\zeta$, $L$, $\ell_1,\ell_2$).
Then we have that 
\begin{align*}
\sup_{p \in \cP_{0,[0,1],\TV^2}'(L) \cup \cP_{0,[0,1],\TV}'(L) \cup \cP_{0,[0,1],\chi^2}'(L)} \sum_{k = 0}^\infty \PP(N = k)\EE_p[\psi_\tau(\cD_k)] & \leq \frac{1}{10},\\ 
\sup_{p \in \{p \in \cQ_{0,[0,1],\TV}'(L): \inf_{q \in \cP'_{0,[0,1]}} \|p - q\|_1 \geq \varepsilon\}} \sum_{k = 0}^\infty \PP(N = k)\EE_p[ 1- \psi_\tau(\cD_k)] & \leq \frac{1}{10}.
\end{align*}
\end{theorem}
\begin{proof}[Proof of Theorem \ref{main:theorem:finite:discrete:XY}] We now derive Theorem \ref{main:theorem:finite:discrete:XY} from Theorem \ref{main:theorem:finite:discrete:XY:Poi}. Note that the test is equivalent to $\mathbbm{1}(T \geq \tau)\mathbbm{1}(N \leq n)$. Hence under the null hypothesis we have
\begin{align*}
\EE\mathbbm{1}(T \geq \tau)\mathbbm{1}(N \leq n) \leq \EE \mathbbm{1}(T \geq \tau) \leq \frac{1}{10},
\end{align*}
where the above expectation is with respect to the randomness of the samples, as well as the randomness of $N$. On the other hand under the alternative we have 
\begin{align*}
\EE (1- \mathbbm{1}(T \geq \tau)\mathbbm{1}(N \leq n)) \leq \EE (1 - \mathbbm{1}(T \geq \tau)) + \EE \mathbbm{1}(N \geq n) \leq  \frac{1}{10}+ \exp(-n/8),
\end{align*}
where we used the bound provided in the following note \cite{canonne2017short}. 

\end{proof}

\begin{proof}[Proof of Theorem \ref{main:theorem:finite:discrete:XY:Poi}] 
To prove this theorem we will assume that $N \sim Poi(n)$ instead of $Poi(n/2)$ for convenience. Since this changes only constant factors, we can do this WLOG. We define 
\begin{align}\label{qxy:def}
q_{xy}(m) = \frac{\int_{C_m} p_{X,Y|Z}(x,y|z) d P_Z(z)}{\PP(Z \in C_m)} = \int_{C_m} p_{X,Y|Z}(x,y|z) d \tilde P_Z(z),
\end{align}
where $p_{X,Y|Z}(x,y|z)$ is the conditional distribution of $X,Y | Z = z$, and $P_Z(z)$ is the distribution of $Z$ which is absolutely continuous with respect to the Lebesgue measure, and
\begin{align*}
d \tilde P_Z(z) = \frac{ d P_Z(z)}{\PP(Z \in C_m)},
\end{align*}
is the conditional distribution of $Z | Z \in C_m$. Further define
\begin{align}\label{qxcdot:qcdoty:def}
q_{x\cdot}(m) = \sum_{y \in [\ell_2]} q_{xy}(m) = \int_{C_m} p_{X|Z}(x | z)d \tilde P(z),  ~~~ q_{\cdot y}(m) = \sum_{x \in [\ell_1]} q_{xy}(m) =  \int_{C_m} p_{Y|Z}(y | z)d \tilde P(z), 
\end{align}

\noindent \textbf{Analysis of the Expectation of $T$.} Conditioning on $\sigma = (\sigma_m)_{m \in [d]}$, and using the fact that $\EE[U_m | \sigma_m] =  \sum_{x,y} (q_{xy}(m) - q_{x\cdot}(m)q_{\cdot y}(m))^2$ is independent of $\sigma_m$, we have

$$
\EE[T] = \EE [\EE[T | \sigma]] = \sum_{m \in [d]} \EE[U_m | \sigma_m] \EE[\sigma_m \mathbbm{1}(\sigma_m \geq 4)] 
$$
Let $p_m = \PP(Z \in C_m)$. Since $\sigma_m \sim \operatorname{Poi}(n p_m)$, Lemma 3.1. of \cite{canonne2018testing} shows that 
$$\EE[\sigma_m \mathbbm{1}(\sigma_m \geq 4)]  \geq \gamma \min(n p_m, (n p_m)^4),$$
where $\gamma = 1 - \frac{5}{2e}$. 
Observe that even under the null $\EE[U_m|\sigma_m] = \sum_{x,y} (q_{xy}(m) - q_{x\cdot}(m)q_{\cdot y}(m))^2 \neq 0$ in general. 
We will now prove that under the null hypothesis we have:
\begin{align}\label{expectation:under:H0}
\EE[U_m|\sigma_m]  \leq \frac{L^2}{d^2}. 
\end{align}
\begin{lemma}\label{expectation:upper:bound:lemma} Since the distribution of $(X,Y,Z)$ belongs to the class $\cP_{0,[0,1], \TV^2}'$ of Definition \ref{def:nullsmoothness} we have \eqref{expectation:under:H0}.
\end{lemma}
\begin{proof} We have
\begin{align*}
\sum_{x,y} (q_{xy}(m) - q_{x\cdot}(m)q_{\cdot y}(m))^2 \leq \bigg(\sum_{x,y} |q_{xy}(m) - q_{x\cdot}(m)q_{\cdot y}(m)|\bigg)^2 
\end{align*}
Furthermore, the following chain of identities holds
\begin{align*}
\MoveEqLeft \sum_{x,y} |q_{xy}(m) - q_{x\cdot}(m)q_{\cdot y}(m)| \\
& = \sum_{x,y} \bigg|\int (p_{X|Z}(x|z) - \int p_{X|Z}(x|z) d\tilde P_Z(z)) (p_{Y|Z}(y|z) - \int p_{Y|Z}(y|z) d\tilde P_Z(z))d\tilde P_Z(z)\bigg|\\
& \leq \int\sum_{x} \bigg|p_{X|Z}(x|z) - \int p_{X|Z} (x | z) d\tilde P_Z(z)\bigg|\sum_{y}\bigg|p_{Y|Z}(y|z) - \int p_{Y|Z}(y|z) d\tilde P_Z(z)\bigg|d\tilde P_Z(z)\\
&\stackrel{(i)}{\leq}  \int \int \sum_{x}|p_{X|Z}(x|z) -  p_{X|Z}(x|z') |d\tilde P_Z(z') \int \sum_{y}|p_{Y|Z}(y|z) -  p_{Y|Z}(y|z') |d\tilde P_Z(z')d\tilde P_Z(z)\\
& = \int \int \|p_{X|Z = z} - p_{X | Z = z'}\|_1d\tilde P_Z(z') \int \|p_{Y|Z = z} - p_{Y | Z = z'}\|_1 d\tilde P_Z(z') d\tilde P_Z(z)\\
& \stackrel{(ii)}{\leq} \int \int \sqrt{L |z - z'|} d\tilde P_Z(z')\int \sqrt{L |z - z'|} d\tilde P_Z(z')d\tilde P_Z(z)\\
&\leq \frac{L}{d},
\end{align*}
where $(i)$ follows by Jensen's inequality, and $(ii)$ follows by the fact that $p \in \cP_{0,[0,1], \TV^2}'$. This completes the proof. 

\end{proof}

Hence a bound on $\EE[T]$ is
\begin{align}\label{expectation:under:H0:times:n}
 \EE[T] \leq \frac{n L^2}{d^2},
\end{align}
since $\sum_{m \in [d]} \EE[\sigma_m \mathbbm{1}(\sigma_m \geq 4)]  \leq \sum_{m \in [d]} \EE[\sigma_m] = \sum_{m \in [d]}np_m = n$. Next we need to lower bound the expectation under the alternative. To this end consider the following

\begin{lemma}\label{lemma:lower:bound:expectation} We have that
\begin{align*}
\sum_{m \in [d]} \sqrt{\EE[U_m|\sigma_m]} p_m & =  \sum_{m \in [d]} \sqrt{\sum_{x, y} (q_{xy}(m) - q_{x\cdot}(m)q_{\cdot y}(m))^2} p_m \\
& \geq \frac{\EE_Z \|p_{X,Y|Z} - p_{X|Z} p_{Y|Z}\|_1 -  3 \frac{L}{d} }{\sqrt{\ell_1 \ell_2}} =:  \frac{\eta}{\sqrt{\ell_1 \ell_2}},
\end{align*}
\end{lemma}

\begin{proof}[Proof of Lemma \ref{lemma:lower:bound:expectation}] First we will show that the function $z \mapsto  \|p_{X,Y| Z = z} - p_{X| Z = z}p_{Y| Z = z}\|_1$ is continuous. Take two values $z, z' \in C_m$ and observe that 
\begin{align*}
\MoveEqLeft | \|p_{X,Y| Z = z} - p_{X| Z = z}p_{Y| Z = z}\|_1 -  \|p_{X,Y| Z = z'} - p_{X| Z = z}p_{Y| Z = z'}\|_1| \\
& \leq \|p_{X,Y| Z = z} - p_{X,Y| Z = z'} + p_{X| Z = z'}p_{Y| Z = z'}- p_{X| Z = z}p_{Y| Z = z}\|_1\\
& \leq \|p_{X,Y| Z = z} - p_{X,Y| Z = z'}\|_1 + \|p_{X| Z = z'}p_{Y| Z = z'}- p_{X| Z = z}p_{Y| Z = z}\|_1\\
& \leq \|p_{X,Y| Z = z} - p_{X,Y| Z = z'}\|_1 + \|p_{X| Z = z'}- p_{X| Z = z}\|_1 + \|p_{Y| Z = z'}- p_{Y| Z = z}\|_1\\
& \leq 3 L |z - z'|,
\end{align*}
where we first used the triangle inequality, next the fact that $\|\cdot\|_1$ is sub-additive on product distributions and finally we used our assumption on the distribution $p_{X,Y |Z}$ and noted that by the triangle inequality 
\begin{align*}
\max(\|p_{X| Z = z'}- p_{X| Z = z}\|_1, \|p_{Y| Z = z'}- p_{Y| Z = z}\|_1) \leq \|p_{X,Y| Z = z} - p_{X,Y| Z = z'}\|_1.
\end{align*}
Since $C_m$ is compact it follows that the function $z \mapsto  \|p_{X,Y| Z = z} - p_{X| Z = z}p_{Y| Z = z}\|_1$ achieves its maximum. Suppose that 
$$
z^*_m \in \argmax_{z \in C_m} \|p_{X,Y| Z = z} - p_{X| Z = z}p_{Y| Z = z}\|_1. 
$$
 By Cauchy-Schwarz we have
\begin{align*}
\sqrt{\sum_{x, y} (q_{xy}(m) - q_{x\cdot}(m)q_{\cdot y}(m))^2} & \geq \frac{\sum_{x,y} |q_{xy}(m) - q_{x\cdot}(m)q_{\cdot y}(m))|}{\sqrt{\ell_1\ell_2}}. 
\end{align*}
Now we apply the triangle inequality to obtain
\begin{align*}
\MoveEqLeft \sum_{x,y} |q_{xy}(m) - q_{x\cdot}(m)q_{\cdot y}(m))| \geq \|p_{x,y|z = z^*_m} - p_{X|Z = z^*_m}p_{Y| Z = z^*_m}\|_1 - \sum_{x, y} |q_{xy}(m) - p_{X,Y| Z} (x,y|z_m^*)|\\
& - \sum_{x,y} |q_{x\cdot}(m)(q_{\cdot y}(m) - p_{Y|Z}(y|z_m^*))| - \sum_{x,y} |(p_{Y|Z}(y|z_m^*)(q_{x\cdot}(m) - p_{X|Z}(x|z_m^*))|. 
\end{align*}
For the first term we use
\begin{align*}
 \sum_{x, y} |q_{xy}(m) - p_{X,Y| Z} (x,y|z_m^*)| & = \sum_{x,y} \bigg|\int_{C_m} p_{X,Y|Z}(x,y|z) - p_{X,Y|Z}(x,y|z_m^*) d \tilde P(z)\bigg| \\
& \leq  \int_{C_m} \sum_{x,y} |p_{X,Y|Z}(x,y|z) - p_{X,Y|Z}(x,y|z_m^*)| d\tilde P(z)\\
& =  \int_{C_m} \|p_{X,Y|Z = z}- p_{X,Y|Z = z_m^*}\|_1 d\tilde P(z)\\
& \leq  \int_{C_m}L| z- z_m^*|d\tilde P(z) \leq L \diam(C_m) = \frac{L}{d},
\end{align*}
For the second term we have
\begin{align*}
\sum_{x,y} |q_{x\cdot}(m)(q_{\cdot y}(m) - p_{Y|Z}(y|z_m^*))| & = \sum_{x,y} \bigg|\int_{C_m} q_{x\cdot}(m)(p_{Y|Z}(y|z) - p_{Y|Z}(y | z_m^*)) d\tilde P(z)\bigg| \\
& \leq  \int_{C_m} \sum_{x,y} |q_{x\cdot}(m)(p_{Y|Z}(y | z) - p_{Y|Z}(y | z_m^*)) |d\tilde P(z)\\
& =  \int_{C_m} \|p_{Y|Z = z} - p_{Y|Z = z_m^*}\|_1 d\tilde P(z)\\
& \leq L \diam(C_m) = \frac{L}{d}.
\end{align*}
The last term is similar to the previous term so we conclude that 
$$
\sqrt{\sum_{x, y} (q_{xy}(m) - q_{x\cdot}(m)q_{\cdot y}(m))^2} \geq \frac{ \|p_{x,y|z = z^*_m} - p_{X |Z = z^*_m}p_{Y|Z = z^*_m}\|_1  - 3 \frac{L}{d}}{\sqrt{\ell_1\ell_2}}
$$
Summing up over $m$ and noting that by the definition of $z_m^*$ we obtain
$$
\sum_{m \in [d]}\|p_{XY| Z = z^*_m} - p_{X |Z = z^*_m}p_{Y| Z = z^*_m}\|_1 p_m \geq \EE_Z \|p_{X,Y|Z} - p_{X |Z} p_{Y|Z}\|_1 \geq \inf_{q \in \cP_{0,[0,1]}}\|p_{X,Y,Z} - q\|_1
$$
where we used the fact that the distribution $p_{X|Z}p_{Y|Z}p_Z$ is a conditionally independent distribution.

\end{proof}
Next we have
$$
\sum_{m \in [d]} \EE[U_m | \sigma_m] \EE[\sigma_m \mathbbm{1}(\sigma_m \geq 4)] \geq \gamma \sum_{m: (n p_m) > 1} \EE[U_m | \sigma_m] n p_m + \gamma \sum_{m: (n p_m) \leq 1} \EE[U_m | \sigma_m] (n p_m)^4.
$$
We now consider two cases:
\begin{itemize}
\item[i.] In the first case we assume	
\begin{align*}
\sum_{m: np_m > 1} \sqrt{\EE[U_m|\sigma_m]} n  p_m \geq \frac{n \eta}{2\sqrt{\ell_1 \ell_2}}.
\end{align*}
Then by Cauchy-Schwarz we have 
\begin{align*}
\sum_{m: (n p_m) > 1} \EE[U_m | \sigma_m] n p_m \geq \frac{(\sum_{m: np_m > 1} \sqrt{\EE[U_m|\sigma_m]} n  p_m)^2}{\sum_{m: (n p_m) > 1} n p_m} \geq \frac{n\eta^2}{4 \ell_1 \ell_2}. 
\end{align*}

\item[ii.] In the second case we suppose:
\begin{align*}
\sum_{m: np_m \leq 1} \sqrt{\EE[U_m|\sigma_m]} n  p_m \geq \frac{n \eta}{2\sqrt{\ell_1 \ell_2}}.
\end{align*}

By Jensen's inequality we have
\begin{align*}
\MoveEqLeft \sum_{m : np_m \leq 1} \frac{\EE[U_m | \sigma_m]^{1/3} }{\sum_{m : np_m \leq 1}  \EE[U_m | \sigma_m]^{1/3} }\EE[U_m | \sigma_m]^{2/3} (n p_m)^4 \\
&\geq \bigg(\sum_{m : np_m \leq 1} \frac{\EE[U_m | \sigma_m]^{1/3} }{\sum_{m : np_m \leq 1}  \EE[U_m | \sigma_m]^{1/3} }\EE[U_m | \sigma_m]^{1/6} n p_m\bigg)^4,
\end{align*}
which is equivalent to 
\begin{align*}
\bigg(\sum_{m : np_m \leq 1}  \EE[U_m | \sigma_m]^{1/3}\bigg)^3 \sum_{m : np_m \leq 1} \EE[U_m | \sigma_m] (n p_m)^4 \geq \bigg(\sum_{m : np_m \leq 1} \sqrt{\EE[U_m | \sigma_m]} n p_m\bigg)^4 \geq \frac{(n \eta)^4}{16 \ell_1^2 \ell_2^2},
\end{align*}

Since $\EE[U_m | \sigma_m] \leq \bigg(\sum_{x,y} |q_{xy}(m) -q_{x\cdot}(m)q_{\cdot y}(m)|\bigg)^2 \leq 4$ we have 
$$
\sum_{m : np_m \leq 1} \EE[U_m | \sigma_m] (n p_m)^4 \geq \frac{(n \eta)^4}{64 d^3 \ell_1^2 \ell_2^2}. 
$$
\end{itemize}

We will now select a threshold at the level of $\zeta \sqrt{d}$, and will give conditions on the minimum critical radius for each of the cases. We will use $\gtrsim$ in the sense bigger up to an absolute constant. We will assume that $\varepsilon - 3 \frac{L}{d} \geq \varepsilon/2$ so that $\eta \geq \varepsilon/2$. 
\begin{itemize}
\item In the first case we obtain the following bound 
$$
\frac{n\eta^2}{4\ell_1\ell_2} \gtrsim \zeta \sqrt{d}
$$
which is ensured when $\varepsilon \gtrsim  \sqrt{\frac{\zeta \sqrt{ d} \ell_1\ell_2}{n}} \vee \frac{1}{d}$
\item In the second case we have
$$
\frac{(n \eta)^4}{64 d^3\ell_1^2\ell_2^2} \gtrsim \zeta \sqrt{d},
$$
which happens when $\varepsilon \gtrsim \frac{\zeta^{1/4} d^{7/8}\sqrt{\ell_1\ell_2}}{n} \vee \frac{1}{d}$.
\end{itemize}
It is simple to check that when $d \asymp n^{2/5}$ the bigger of the two rates is $\varepsilon \gtrsim n^{-2/5}$.

\noindent \textbf{Analysis of the Variance of $T$.} The rule of total variance ensures that
$$
\Var T = \EE[\Var [T | \sigma]] + \Var[\EE [T | \sigma]],
$$
where $\sigma = (\sigma_m)_{m \in [d]}$. Put for brevity $T_m = \mathbbm{1}(\sigma_m \geq 4) U_m \sigma_m$ so that $T = \sum_{m \in [d]}T_m$. We first handle the first term. We have
$$
\Var [T | \sigma] = \sum_{m,k \in [d]} \Cov(T_m, T_k | \sigma_m, \sigma_k) = \sum_{m \in [d]} \Var(T_m | \sigma_m),
$$
where we used that $T_m$ and $T_k$ are independent given $\sigma_m, \sigma_k$. Using Lemma \ref{U:stat:variance:lemma} and the fact that $\sum_{x,y} q^2_{x,y}(m) \leq 1$ and $\sum_{x,y} q^2_{x\cdot}(m)q^2_{\cdot y}(m) \leq 1$, we have
\begin{align*}
\MoveEqLeft \sum_{m \in [d]} \Var(T_m | \sigma_m) \leq \sum_{m \in [d]} \sigma^2_m \mathbbm{1}(\sigma_m \geq 4) C\bigg(\frac{\EE[U_m | \sigma_m]}{\sigma_m} + \frac{1}{\sigma_m^2}\bigg) \\
&= \sum_{m \in [d]}  C(\EE[\mathbbm{1}(\sigma_m \geq 4)U_m\sigma_m | \sigma_m] + \mathbbm{1}(\sigma_m \geq 4)) . 
\end{align*}
Taking expectation of the expression above we end up with
$$
\EE[\Var[T | \sigma]] \leq C\bigg(\EE[T] + \EE \sum_{m \in [d]} \mathbbm{1}(\sigma_m \geq 4)\bigg) \leq C(\EE[T] + d).
$$

For the second term we have
$$
\EE[T | \sigma] = \sum_{m \in [d]} \sigma_m \mathbbm{1}(\sigma_m \geq 4) \EE[U_m | \sigma_m] = \sum_{m \in [d]} \sigma_m \mathbbm{1}(\sigma_m \geq 4) \sum_{x,y} (q_{xy}(m) - q_{x\cdot}(m)q_{\cdot y}(m))^2
$$
Since the $\sigma_m$ are independent we have
$$
\Var[\EE[T | \sigma]] = \sum_{m \in [d]} \Var[\sigma_m \mathbbm{1}(\sigma_m \geq 4)] \bigg(\sum_{x,y} (q_{xy}(m) - q_{x\cdot}(m)q_{\cdot y}(m))^2\bigg)^2
$$
By Claim 2.1. of \cite{canonne2018testing} we have that $ \Var[\sigma_m \mathbbm{1}(\sigma_m \geq 4)] \leq C'\EE[ \sigma_m \mathbbm{1}(\sigma_m \geq 4)]$, and $\sum_{x,y} (q_{xy}(m) - q_{x\cdot}(m)q_{\cdot y}(m))^2 \leq (\sum_{x,y} |q_{xy}(m) - q_{x\cdot}(m)q_{\cdot y}(m)|)^2 \leq 4$ thus
$$
\Var[\EE[T | \sigma]] \leq 4C'\sum_{m \in [d]} \EE[\sigma_m \mathbbm{1}(\sigma_m \geq 4)] \sum_{x,y} (q_{xy}(m) - q_{x\cdot}(m)q_{\cdot y}(m))^2 = 4C' \EE[T].
$$
Hence we conclude $\Var T \leq C(\EE[T] + d)$. 

\noindent \textbf{Putting Things Together. } Recall that the threshold is set as $\tau = \zeta n^{1/5}$. First we handle the null hypothesis. By Chebyshev's inequality we have
$$
\PP(|T - \EE T| \geq \tau) \leq \frac{\Var T}{\tau^2} = \frac{C(\EE[T] + d)}{\tau^2} \leq \frac{C(C'n^{1/5} + n^{2/5})}{\zeta n^{2/5}} \leq \frac{1}{10},
$$
when $\zeta$ is large enough, where we used the bound \eqref{expectation:under:H0:times:n}. In this scenario we have that $T \leq \tau + \EE T \leq 2 \tau$ for large enough $\zeta$. Under the alternative 

$$
\PP(|T - \EE T| \geq \EE T/2) \leq \frac{4\Var T}{(\EE T)^2} \leq 4 C\bigg(\frac{d}{(\EE T)^2} + \frac{1}{\EE T}\bigg) \leq \frac{1}{10},
$$
since $\EE T \gtrsim \zeta \sqrt{d} \gtrsim \zeta n^{1/5}$. 
\end{proof}

\subsection{Proofs of Section \ref{scaling:l1:l2:section}}

\begin{proof}[Proof of Lemma \ref{second:variance:upper:bound}] 
We will show that the variance $ \Var[U_{W}(\cD) | \cD_{X'}, \cD_{Y'}]$ is bounded as 
$$
 C\bigg(\frac{ \|p_{X',Y',A} - p^\Pi_{X',Y',A}\|_2^2\max(\|p_{X',Y',A}\|_2, \|p^\Pi_{X',Y',A}\|_2)}{\sigma} + \frac{\max(\|p_{X',Y',A}\|^2_2, \|p^\Pi_{X',Y',A}\|^2_2)}{\sigma^2}\bigg),
$$
We will now complete the proof assuming this is correct. We use the triangle inequality to obtain
$$
\|p_{X',Y',A}\|^2_2 \leq (\| p^\Pi_{X',Y',A} - p_{X',Y',A}\|_2 + \| p^\Pi_{X',Y',A}\|_2)^2 \leq 2 (\| p^\Pi_{X',Y',A}-p_{X',Y',A}\|_2^2 + \| p^\Pi_{X',Y',A}\|_2^2).
$$
This gives the following bound on the variance $\Var[U_{W}(\cD) | \cD_{X'}, \cD_{Y'}]$ 
\begin{align}\label{bound:on:the:variance}
 \MoveEqLeft C\bigg(\frac{ \|p_{X',Y',A}  - p^\Pi_{X',Y',A} \|_2^2\|p^\Pi_{X',Y',A} \|_2}{\sigma} + \frac{\|p_{X',Y',A}  - p^\Pi_{X',Y',A}\|_2^3}{\sigma} \nonumber\\
&+ \frac{ \|p^\Pi_{X',Y',A}\|^2_2}{\sigma^2} +  \frac{\|p_{X',Y',A}  - p^\Pi_{X',Y',A} \|_2^2}{\sigma^2}\bigg),
\end{align}
which is what we wanted to show. Now it remains to show the first bound. The calculation is almost identical to the one of Lemma \ref{U:stat:variance:lemma} (but we will repeat it for the sake of completeness). Note that the sample size $2t + 4 \geq \sigma/2$ so by adjusting the constant we can get a bound with $\sigma$ in place of $2t + 4$. Going back to equation \eqref{equation:going:back:to}
\begin{align*}
\MoveEqLeft 
\sum_{x,y,x',y'} \frac{\bigg\{\EE  \phi_{\pi_1\pi_2}(xy)\phi_{\pi_3 \pi_4}(xy) \phi_{\pi_1'\pi_2'}(x'y')\phi_{\pi_3' \pi_4'}(x'y') - (\EE \phi_{\pi_1\pi_2}(xy))^2(\EE \phi_{\pi_1'\pi_2'}(x'y'))^2 \bigg\}}{(1 + a_{xy})(1 + a_{x'y'})}\\
& \leq \sum_{x,y,x',y'} \frac{\bigg\{\EE  \phi_{\pi_1\pi_2}(xy)\phi_{\pi_3 \pi_4}(xy) \phi_{\pi_1'\pi_2'}(x'y')\phi_{\pi_3' \pi_4'}(x'y')\bigg\}}{(1 + a_{xy})(1 + a_{x'y'})}\\
& = \sum_{x,y,x',y'} \frac{\bigg\{\EE (\phi_{\pi_1\pi_2}(xy)\phi_{\pi_1'\pi_2'}(x'y'))\EE \phi_{\pi_3 \pi_4}(xy) \EE\phi_{\pi_3' \pi_4'}(x'y') \bigg\}}{(1 + a_{xy})(1 + a_{x'y'})}\\
&\leq \sqrt{\sum_{x,y,x',y'} \frac{\{\EE (\phi_{\pi_1\pi_2}(xy) \phi_{\pi_1'\pi_2'}(x'y'))\}^2}{(1 + a_{xy})(1 + a_{x'y'})}} \sqrt{\sum_{x,y,x',y'}  \frac{\{\EE\phi_{\pi_3 \pi_4}(xy)\}^2}{(1 + a_{xy})} \frac{\{\EE\phi_{\pi_3' \pi_4'}(x'y')\}^2}{(1 + a_{x'y'})}}\\
& = \sqrt{\sum_{x,y,x',y'} \frac{\{\EE (\phi_{\pi_1\pi_2}(xy) \phi_{\pi_1'\pi_2'}(x'y'))\}^2}{(1 + a_{xy})(1 + a_{x'y'})}}\|p_{X',Y',A}  - p^\Pi_{X',Y',A} \|_2^2
\end{align*}
where, as before we supposed that $i$ is some of $\pi_1,\pi_2$ and $\pi_1', \pi_2'$, and the next to last step follows from Cauchy-Schwarz. Now we apply Lemma \ref{simple:lemma:discrete} to the first term on the RHS with $A = \phi_{\pi_1\pi_2}(xy)$ and $B = \phi_{\pi_1'\pi_2'}(x'y')$. We obtain
\begin{align*}
\sum_{x,y,x',y'} \frac{\{\EE (\phi_{\pi_1\pi_2}(xy) \phi_{\pi_1'\pi_2'}(x'y'))\}^2}{(1 + a_{xy})(1 + a_{x'y'})} \leq \sum_{x,y,x',y'} \frac{\EE [|\phi_{\pi_1'\pi_2'}(x'y'))| | |\phi_{\pi_1\pi_2}(xy)| = 1]}{(1 + a_{x'y'})} \frac{(\EE |\phi_{\pi_1\pi_2}(xy)|)^2}{(1 + a_{xy})}
\end{align*}
Note that as before
\begin{align*}|\phi_{\pi_1\pi_2}(xy)| &\leq\mathbbm{1}(X_{\pi_1} = x, Y_{\pi_1} = y) + \mathbbm{1}(Y_{\pi_2} = y) \mathbbm{1}(X_{\pi_1} = x)
\end{align*}
and a similar inequality holds for $|\phi_{\pi_1'\pi_2'}(x'y')|$. Thus 
\begin{align*}
\sum_{x',y'} \frac{|\phi_{\pi_1'\pi_2'}(x'y')|}{1 + a_{x'y'}} \leq \sum_{x',y'}\frac{\mathbbm{1}(X_{\pi_1} = x', Y_{\pi_1} = y') + \mathbbm{1}(Y_{\pi_2} = y') \mathbbm{1}(X_{\pi_1} = x')}{1 + a_{x'y'}} \leq 2,
\end{align*}
since $a_{x'y'} \geq 0$. Thus 
\begin{align*}
\MoveEqLeft \sum_{x, y} \EE [\sum_{x',y'}\frac{|\phi_{\pi_1'\pi_2'}(x'y'))|}{1 + a_{x'y'}} | |\phi_{\pi_1\pi_2}(xy)| = 1] \frac{(\EE |\phi_{\pi_1\pi_2}(xy)|)^2}{1 + a_{xy}} \leq 2\sum_{x, y}\frac{(\EE |\phi_{\pi_1\pi_2}(xy)|)^2}{1 + a_{xy}} \\
& \leq 2 \sum_{x, y} \frac{(p_{X',Y'}(x,y) + p_{X'}(x)p_{Y'}(y))^2}{1 + a_{xy}} \leq 4 (\|p_{X',Y',A}\|^2_2 + \|p^\Pi_{X',Y',A}\|_2^2),
\end{align*}
which is what we wanted to show. 

Now we will show how to bound any higher order terms (i.e., according to \cite{van2000asymptotic} the variance of the next term is governed by $\Cov( h^{\ba}_{ijkl}, h^{\ba}_{ijk'l'})$ where it's possible for $k', l'$ to be equal to $k$ or $l$). To bound this we directly go back to the inequality
\begin{align*}
\leq \sum_{x,y,x',y'} \frac{\bigg\{\EE (\phi_{\pi_1\pi_2}(xy)\phi_{\pi_1'\pi_2'}(x'y'))\EE \phi_{\pi_3 \pi_4}(xy) \EE\phi_{\pi_3' \pi_4'}(x'y') \bigg\}}{(1 + a_{xy})(1 + a_{x'y'})}
\end{align*}
We will now use the fact that for $A, B \in \{0,\pm1\}$
$$
\EE AB \leq \EE|AB| = \EE[|B| | |A| = 1] \EE |A|,
$$
where $A = \phi_{\pi_1\pi_2}(xy)\phi_{\pi_3 \pi_4}(xy)$ and $B = \phi_{\pi_1'\pi_2'}(x'y')\phi_{\pi_3' \pi_4'}(x'y')$. We have
\begin{align*}
\MoveEqLeft \sum_{x,y,x',y'} \frac{\bigg\{\EE (\phi_{\pi_1\pi_2}(xy)\phi_{\pi_1'\pi_2'}(x'y'))\EE \phi_{\pi_3 \pi_4}(xy) \EE\phi_{\pi_3' \pi_4'}(x'y') \bigg\}}{(1 + a_{xy})(1 + a_{x'y'})}  \\
&\leq \sum_{x,y,x',y'} \EE\bigg[\frac{|\phi_{\pi_1'\pi_2'}(x'y')\phi_{\pi_3' \pi_4'}(x'y')|}{1 + a_{x'y'}} | |\phi_{\pi_1\pi_2}(xy)\phi_{\pi_3 \pi_4}(xy)| = 1\bigg] \frac{\EE |\phi_{\pi_1\pi_2}(xy)\phi_{\pi_3 \pi_4}(xy)|}{1 + a_{xy}}
\end{align*}
We now use that as we saw before
$$
|\phi_{\pi_1\pi_2}(xy)| \leq \mathbbm{1}(X_{\pi_1} = x, Y_{\pi_1} = y) + \mathbbm{1}(Y_{\pi_1} = y) \mathbbm{1}(X_{\pi_2} = x),
$$
and analogously for the others. Furthermore $|\phi_{\pi_1\pi_2}(xy)| \leq 1$. 
\begin{align*}
\MoveEqLeft \sum_{x',y'}\EE\bigg[\frac{|\phi_{\pi_1'\pi_2'}(x'y')\phi_{\pi_3' \pi_4'}(x'y')|}{1 + a_{x'y'}} \bigg| |\phi_{\pi_1\pi_2}(xy)\phi_{\pi_3 \pi_4}(xy)| = 1\bigg] \leq \sum_{x',y'}\EE\bigg[\frac{\phi_{\pi_1'\pi_2'}(x'y')}{1 + a_{x'y'}} \bigg| |\phi_{\pi_1\pi_2}(xy)\phi_{\pi_3 \pi_4}(xy)| = 1\bigg]\\
& \leq  2.
\end{align*}
Next by independence, 
\begin{align*}
\MoveEqLeft\sum_{x, y}\frac{\EE |\phi_{\pi_1\pi_2}(xy)\phi_{\pi_3 \pi_4}(xy)|}{1 + a_{xy}} = \sum_{x, y}\frac{ \EE |\phi_{\pi_1\pi_2}(xy)| \EE|\phi_{\pi_3 \pi_4}(xy)|}{1 + a_{xy}}  \\
& \leq \sum_{x, y} \frac{(p_{X',Y'}(x,y) + p_{X'}(x)p_{Y'}(y))^2}{1 + a_{xy}}  \leq 2\sum_{x, y} \frac{p^2_{X',Y'}(x,y) + p^2_{X'}(x)p^2_{Y'}(y)}{1 + a_{xy}} \\
& = 2 (\|p_{X',Y',A}\|^2_2 + \|p^\Pi_{X',Y',A}\|_2^2).
\end{align*}
This completes the proof.
\end{proof}

We now state the following Poissonized version of Theorem \ref{main:theorem:scaling:discrete:XY}
\begin{theorem}[Scaling Discrete $X$, $Y$ Upper Bound]\label{main:theorem:scaling:discrete:XY:Poi} Set $d = \lceil \frac{n^{2/5}}{(\ell_1 \ell_2)^{1/5}}\rceil$ and set the threshold $\tau = \sqrt{\zeta d}$ for a sufficiently large absolute constant $\zeta$ (depending on $L$). Suppose that $\ell_1 \geq \ell_2$ satisfy the condition that 
$d \ell_1 \lesssim n$.
Then 
when $\varepsilon \geq c \frac{(\ell_1 \ell_2)^{1/5}}{n^{2/5}}$, for a sufficiently large absolute constant $c$ (depending on $\zeta$, $L$), 
we have that 
\begin{align*}
\sup_{p \in \cP_{0,[0,1],\chi^2}'(L)}\sum_{k = 0}^\infty \PP(N = k)\EE_p[\psi_\tau(\cD_k)] & \leq \frac{1}{10},\\ 
\sup_{p \in \{p \in \cQ_{0,[0,1],\TV}'(L): \inf_{q \in \cP'_{0,[0,1]}} \|p - q\|_1 \geq \varepsilon\}}\sum_{k = 0}^\infty \PP(N = k)\EE_p[1 - \psi_\tau(\cD_k)]& \leq \frac{1}{10}.
\end{align*}
\end{theorem}

\begin{proof}[Proof of Theorem \ref{main:theorem:scaling:discrete:XY}] Using Theorem \ref{main:theorem:scaling:discrete:XY:Poi} the proof is the same as that of Theorem \ref{main:theorem:finite:discrete:XY}
\end{proof}

\begin{proof}[Proof of Theorem \ref{main:theorem:scaling:discrete:XY:Poi}] 
To prove this theorem we will assume that $N \sim Poi(n)$ instead of $Poi(n/2)$ for convenience. Since this changes only constant factors, we can do this WLOG. As discussed in the introduction of Section~\ref{scaling:l1:l2:section}, we remind the reader that we focus throughout this proof, without loss of generality, on the setting where $\sqrt{\ell_1 \ell_2}/n \lesssim 1$, noting that when this condition is not satisfied there is a trivial test which is minimax optimal.

For each dataset $\cD_m$ we will index with $m$ all the quantities defined in the main text. For example $t_{1,m}, t_{2,m}$ will refer to the sample sizes of $\cD_{m,X}$ and $\cD_{m,Y}$, while $t_m$ will be such that $\sigma_m = 4 + 4 t_m$. In addition $a_{xy}^m$, $a_x^m$ and $a_y^{'m}$ will denote the weighting amounts. Furthermore, $A_m$ will denote the multi-set of samples where $(x,y)$ appears $a_{xy}^m$ times. For brevity we will refer to the weighting randomness as $R_m$, i.e., $R_m = \{\cD_{m,X}, \cD_{m,Y}\}$. Furthermore we will denote $\sigma = \{\sigma_m\}_{m \in [d]}$ and $R = \{R_m\}_{m \in [d]}$. 

Let us also define 
\begin{align}\label{Tm:def}
T_m = \mathbbm{1}(\sigma_m \geq 4) \omega_m\sigma_m U_m.
\end{align}
Recall the definitions of $q_{xy}(m)$, $q_{x\cdot}(m)$ and $q_{\cdot y}(m)$ \eqref{qxy:def} and \eqref{qxcdot:qcdoty:def}. These distributions ``play the role'' of $p_{X',Y'}$ and $p_{X'}$ an $p_{Y'}$ from the main text. For brevity, denote the distributions with density $q_{xy}(m)$ and $q_{x \cdot}(m)q_{y\cdot}(m)$ with $q(m)$ and $q_\Pi(m)$. Denote the $A_m$-split distribution $q_{}(m)$ by $q_{A_m}(m)$ and the $A_m$-split distribution $q_{\Pi}(m)$ by $q_{\Pi,A_m}(m)$. 

Before we delve into the proof we give several useful definitions and results about the weighting which we take from \cite{canonne2018testing}.

\begin{definition}[1-Dimensional Split Distribution]\label{split:distribution} Given a discrete distribution $p$ over $[d]$ and a multi-set $S$ of elements of $[d]$ define the distribution $p_S$ over $[d + |S|]$. Let $a_i = \sum_{j \in S}\mathbbm{1}( j = i)$. Thus $\sum_{i \in [d]}1 + a_i  = d + |S|$. We can therefore associate elements of the set $[d + |S|]$ with elements in the set $B_S = \{(i,j) | i \in [d], 1 \leq j \leq 1 + a_i\}$. The split distribution $p_S$ is supported on $B_S$ and is obtained by sampling $i$ from $p$ and $j$ uniformly from the set $[1 + a_i]$. 
\end{definition}

\begin{lemma}[Fact 2.2 \cite{canonne2018testing}]\label{lemma:TV:identity:flattened} Let $p,q$ are distributions over $[d]$, and $S$ is a given multi-set of $[d]$. Then we can simulate a sample from $p_S$ or $q_S$ by taking a single sample from $p$ or $q$. It also holds that $d_{\TV}(p_S,q_S) = d_{\TV}(p,q)$. 
\end{lemma}

\begin{lemma}[Lemma 2.3 \cite{canonne2018testing}] Let $p$ be a discrete distribution over $[d]$. Then, for any multi-sets $S \subseteq S'$ of $[d]$, $\|p_S\|_2 \leq \|p_{S'}\|_2$, and if $S$ is obtained by $m$ independent samples from $p$, then $\EE[\|p_S\|_2^2] \leq \frac{1}{m+1}$. 
\end{lemma}

An important implication of the proof of this lemma is that if $a_i$ denotes the number of samples in $S$ which equal to $i$, when $S$ is drawn as $m$ independent samples from $p$:
\begin{align}\label{flattening:inequality}
\EE \frac{1}{1 + a_i} \leq \frac{1}{(m+1)p_i}.
\end{align}
Using the independence of $a_x^m$ and $a_y^m$ we can therefore conclude that by \eqref{flattening:inequality}
\begin{align}\label{flattening:expectation:bound}
\EE\bigg[\frac{1}{1 + a^m_{xy}}\bigg] =\EE\bigg[\frac{1}{1 + a^m_x}\bigg]\EE\bigg[\frac{1}{1 +a^{'m}_y}\bigg] \leq \frac{1}{(1 + t_{1,m})(1 + t_{2,m})q_{x\cdot}(m)q_{\cdot y}(m)},
\end{align}
where the expectation above is with respect to the randomness in $R_m$. Hence
\begin{align}\label{bound:on:l2:norm:qpi}
\EE \|q_{\Pi,A_m}(m)\|_2^2 & = \sum_{x,y} \EE \frac{(q_{x\cdot}(m)q_{\cdot y}(m))^2}{1 + a^m_{xy}}\leq \sum_{x,y} \frac{q_{x\cdot}(m)q_{\cdot y}(m)}{(1 + t_{1,m})(1 + t_{2,m})} = \frac{1}{(1 + t_{1,m})(1 + t_{2,m})}.
\end{align}
\noindent \textbf{Analysis of the Expectation.} Recall that the test statistic is 
\begin{align*}
T = \sum_{m \in [d]} T_m = \sum_{m \in [d]}\mathbbm{1}(\sigma_m \geq 4) \sigma_m \omega_m U_m.
\end{align*}

Recall that we denote the randomness of the flattening with $R_m$ for bin $m$, the sample size within each bin as $\sigma_m$, and let us denote the randomness associated with the estimator $U_m$ with $K_m$. We have
$$
\EE[T_m | \sigma_m, R_m] = \sigma_m \omega_m \|q_{A_m}(m) - q_{\Pi,A_m}(m)\|_2^2\mathbbm{1}(\sigma_m \geq 4).
$$
Here we use the following bound
\begin{align*}
\|q_{A_m}(m) - q_{\Pi,A_m}(m)\|_2 & \geq \frac{\|q_{A_m}(m) - q_{\Pi,A_m}(m)\|_1}{\sqrt{(\ell_1 + t_{1,m})(\ell_2 + t_{2,m})}} = \frac{\|q_{}(m) - q_{\Pi}(m)\|_1}{\sqrt{(\ell_1 + t_{1,m})(\ell_2 + t_{2,m})}} \\
& \geq \frac{\|q_{}(m) - q_{\Pi}(m)\|_1}{2\sqrt{\ell_1\ell_2}},
\end{align*}
where the first inequality follows from Cauchy-Schwarz (or simply by the T2 Lemma) and the fact that $\sum_{x,y} 1 + a_{xy}^m = \sum_{x,y} (1 + a_{x}^m)(1 + a_y^{'m}) = (\ell_1 + t_{1,m})(\ell_2 + t_{2,m})$, and the second identity follows from Lemma \ref{lemma:TV:identity:flattened}. Denote by $\varepsilon_m = d_{\TV}(q_{}(m), q_{\Pi}(m)) = \frac{\|q_{}(m) - q_{\Pi}(m)\|_1}{2}$ for convenience. We have that 
\begin{align}\label{Tm:lower:bound}
\EE[T_m | \sigma_m, R_m] \geq \sigma_m \omega_m\mathbbm{1}(\sigma_m \geq 4) \frac{\varepsilon^2_m}{\ell_1\ell_2}.
\end{align}
Denote by $\alpha_m = np_m$ where $p_m= \PP(Z \in C_m)$. We have the following lemma
\begin{lemma} The following inequality holds
\begin{align}\label{lower:bound:on:A}
\EE \sum_{m \in [d]}  \sigma_m \omega_m\mathbbm{1}(\sigma_m \geq 4) \frac{\varepsilon^2_m}{\ell_1\ell_2} \geq \gamma \sum_{m \in [d]} \frac{\varepsilon_m^2 }{\ell_1\ell_2} \min(\alpha_m \beta_m, \alpha_m^4),
\end{align}
where $\beta_m = \sqrt{\min(\alpha_m, \ell_1)\min(\alpha_m, \ell_2)}$ and $\gamma$ is some absolute constant. 
\end{lemma}

\begin{proof}
We have that 
\begin{align*}
\EE \sum_{m \in [d]}  \sigma_m \omega_m\mathbbm{1}(\sigma_m \geq 4) \frac{\varepsilon^2_m}{\ell_1\ell_2} =  \sum_{m \in [d]} \EE[ \sigma_m \omega_m\mathbbm{1}(\sigma_m \geq 4)] \frac{\varepsilon^2_m}{\ell_1\ell_2}
\end{align*}
By Claim 2.3 of \cite{canonne2018testing} we have that for $X \sim \operatorname{Poi}(\lambda)$: $$\EE[X \sqrt{\min(X,a) \min(X,b)} \mathbbm{1}(X \geq 4)] \geq \gamma \min(\lambda\sqrt{\min(\lambda, a)\min(\lambda, b)}, \lambda^4).$$ Thus
$$
\EE[ \sigma_m \omega_m\mathbbm{1}(\sigma_m \geq 4)] \geq \gamma \min(\alpha_m \sqrt{\min(\alpha_m, \ell_1)\min(\alpha_m, \ell_2)}, \alpha_m^4)
$$
which completes the proof. 
\end{proof}

Next suppose that $\inf_{q \in \cP_{0,[0,1]}'} \|p_{X,Y,Z} - q\|_1 > \varepsilon$. We want to show some lower bounds on the RHS of \eqref{lower:bound:on:A}.  We start by looking into the expression
\begin{align}\label{expression:to:lower:bound}
\sum_{m \in [d]} \frac{\varepsilon_m }{\sqrt{\ell_1\ell_2}} \alpha_m.
\end{align}
Recall that $\varepsilon_m = d_{\TV}(q_{}(m), q_{\Pi}(m))$. To obtain a lower bound we take 
$$
z^*_m \in \argmax_{z \in C_m} \|p_{X,Y| Z= z} - p_{X | Z = z}p_{Y|Z=z}\|_1, 
$$
where just as before in the proof of Lemma \ref{lemma:lower:bound:expectation} we can show that the map $z \mapsto \argmax_{z \in C_m} \|p_{X,Y| Z= z} - p_{X | Z = z}p_{Y|Z=z}\|_1$ is continuous. The rest of the proof is very similar to Lemma \ref{lemma:lower:bound:expectation} but we provide full details for completeness. By the triangle inequality we have 
\begin{align*}
2 \varepsilon_m & \geq \|p_{X,Y|Z = z_m^*} - p_{X |Z=z_m^*}p_{Y | Z = z_m^*}\|_1 - \|q(m) - p_{X,Y|Z = z_m^*}\|_1 \\
& -  \|q_{X \cdot}(m)(q_{\cdot Y}(m) - p_{Y|Z= z_m^*})\|_1 -  \|p_{Y|Z = z_m^*}(q_{X\cdot}(m) - p_{X| Z=z_m^*})\|_1,
\end{align*}
where we denoted the distributions with density $q_{x\cdot}(m)$: $q_{X\cdot}(m)$ and similarly for $q_{\cdot Y}(m)$. Now we will bound the three terms. We start with
\begin{align*}
\|q(m) - p_{X,Y| Z = z_m^*}\|_1 & = \sum_{x,y} \bigg|\int_{C_m} p_{X,Y|Z}(x,y|z) - p_{X,Y|Z}(x,y|z_m^*) d\tilde P(z)\bigg| \\
& \leq  \int_{C_m}\sum_{x,y} |p_{X,Y|Z}(x,y|z) - p_{X,Y|Z}(x,y|z_m^*)| d\tilde P(z)\\
& \leq  \int_{C_m}L| z- z_m^*|d\tilde P(z) \leq L \diam(C_m) = \frac{L}{d},
\end{align*}
where $d\tilde P(z) = \frac{d P(z)}{\PP(Z \in A_m)}$, is the conditional distribution of $Z \in C_m$. Similarly 
\begin{align*}
\|q_{X\cdot}(m)(q_{\cdot Y}(m) - p_{Y|z_m^*})\|_1  & = \sum_{x,y} \bigg|\int_{C_m} q_{x\cdot}(m)(p_{Y|Z}(y|z)- p_{Y|Z}(y|z_m^*)) d\tilde P(z)\bigg| \\
& \leq  \int_{C_m} \sum_{x,y} q_{x\cdot}(m)|p_{Y|Z}(y|z)- p_{Y|Z}(y|z_m^*)|d\tilde P(z)\\
& \leq L \diam(C_m) = \frac{L}{d}.
\end{align*}
The last term is similar to the previous term so we conclude that 
$$
\varepsilon_m \geq d_{\TV}(p_{X,Y | Z = z_m^*}, p_{X |Z = z_m^*}p_{Y |Z = z_m^*}) - \frac{3}{2} \frac{L}{d}. 
$$
Therefore 
\begin{align*}
\sum_{m \in [d]} \frac{\varepsilon_m }{\sqrt{\ell_1\ell_2}} \alpha_m & \geq \frac{n}{\sqrt{\ell_1\ell_2}} \bigg(\sum_{m} p_m d_{\TV}(p_{X,Y | Z = z_m^*}, p_{X |Z = z_m^*}p_{Y |Z = z_m^*}) - \frac{3}{2} \frac{L}{d}\bigg) \\
& \geq \frac{n}{\sqrt{\ell_1\ell_2}} \bigg(\frac{1}{2}\EE_{Z} \|p_{XY| Z} - p_{X| Z} p_{Y| Z}\|_1 - \frac{3}{2} \frac{L}{d} \bigg)\\
& =:  \frac{n}{\sqrt{\ell_1\ell_2}}  \eta~\footnotemark
\end{align*}
\footnotetext{Note here that $\EE_Z \|p_{X,Y| Z} - p_{X| Z} p_{Y| Z}\|_1 = \|p_{X,Y,Z} - p_{X | Z}p_{Y|Z}p_Z\|_1 \geq \varepsilon$}

Now, as in the proof of Theorem \ref{main:theorem:finite:discrete:XY}, the analysis is partitioned into two parts. The first part takes the set $M_H = \{m | \alpha_m^3 \geq \beta_m\}$ and $M_L = \{m | \alpha_m^3 < \beta_m\}$. By the above analysis we know that either $\sum_{m \in M_H} \frac{\varepsilon_m }{\sqrt{\ell_1\ell_2}} \alpha_m \geq \frac{n}{2\sqrt{\ell_1\ell_2}}  \eta$ or $\sum_{m \in M_L} \frac{\varepsilon_m }{\sqrt{\ell_1\ell_2}} \alpha_m \geq \frac{n}{2\sqrt{\ell_1\ell_2}} \eta$.
In the first case we want to lower bound 
\begin{align*}
\sum_{m \in M_H} \frac{\varepsilon_m^2 }{\ell_1\ell_2} \alpha_m \beta_m.
\end{align*}
In order to determine the value of $\beta_m$ we consider three more cases. Suppose without loss of generality that $\ell_2 \leq \ell_1$. Define the three sets $M_{H,1} = \{m | \ell_2 \leq \ell_1 \leq \alpha_m\}$, $M_{H,2} = \{m | \ell_2 \leq \alpha_m \leq \ell_1 \}$, $M_{H,3} = \{m | \alpha_m\leq \ell_2 \leq \ell_1 \}$. We have that 
$$
\max_{i \in [3]} \sum_{m \in M_{H,i}} \frac{\varepsilon_m }{\sqrt{\ell_1\ell_2}} \alpha_m \geq \frac{n}{6\sqrt{\ell_1\ell_2}}  \eta
$$
We now analyze the three cases depending on where the maximum above is achieved. 
\begin{itemize}
\item Suppose that the maximum is achieved at $i = 1$. Thus
$$
\sum_{m \in M_{H,1}} \frac{\varepsilon_m }{\sqrt{\ell_1\ell_2}} \alpha_m \geq \frac{n}{6\sqrt{\ell_1\ell_2}} \eta.
$$
We have in this subcase that $\beta_m = \sqrt{\ell_1 \ell_2}$. Therefore
\begin{align*}
 \sum_{m \in M_{H,1}} \frac{\varepsilon_m^2 }{\ell_1\ell_2} \alpha_m \beta_m &= \sqrt{\ell_1\ell_2}  \sum_{m \in M_{H,1}} \frac{\varepsilon_m^2 }{\ell_1\ell_2} \alpha_m \geq \sqrt{\ell_1\ell_2} \frac{\bigg(\sum_{m \in M_{H,1}} \frac{\varepsilon_m \alpha_m}{\sqrt{\ell_1\ell_2}}\bigg)^2}{\sum_{m \in M_{H,1}} \alpha_m } \\
& \geq \sqrt{\ell_1\ell_2} \frac{n^2 \eta^2}{36\ell_1\ell_2 (\sum_{m \in M_{H,1}} \alpha_m )}.
\end{align*}
Now we have that $\sum_{m \in M_{H,1}} \alpha_m  = \sum_{m \in M_{H,1}} n p_m\leq n \sum_{m \in [d]} p_m = n$. Thus we conclude
\begin{align*}
 \sum_{m \in M_{H,1}} \frac{\varepsilon_m^2 }{\ell_1\ell_2} \alpha_m \beta_m \geq \frac{n \eta^2}{36\sqrt{\ell_1\ell_2}}.
\end{align*}
\item In the second case we have that the maximum is achieved in $i = 2$ which means
$$
\sum_{m \in M_{H,2}} \frac{\varepsilon_m }{\sqrt{\ell_1\ell_2}} \alpha_m \geq \frac{n}{6\sqrt{\ell_1\ell_2}} \eta.
$$
In this case $\beta_m = \sqrt{\ell_2}\sqrt{\alpha_m}$, so we have
\begin{align*}
 \sum_{m \in M_{H,2}} \frac{\varepsilon_m^2 }{\ell_1\ell_2} \alpha_m \beta_m &= \sqrt{\ell_2} \sum_{m \in M_{H,2}} \frac{\varepsilon_m^2 }{\ell_1\ell_2} \alpha_m^{3/2} \geq \sqrt{\ell_2} \frac{\bigg(\sum_{m \in M_{H,2}} \frac{\varepsilon_m }{\sqrt{\ell_1\ell_2}} \alpha_m\bigg)^2}{\sum_{m \in M_{H,2}} \sqrt{\alpha_m}}. 
\end{align*}
To bound $\sum_{m \in M_{H,2}} \sqrt{\alpha_m}$ we note that 
$$
\sum_{m \in M_{H,2}} \sqrt{\alpha_m} \leq \sum_{m \in [d]} \sqrt{n}\sqrt{p_m} \leq \sqrt{nd}. 
$$
On the other hand, we also have that $\ell_2 \leq \alpha_m \leq \ell_1$ so that $n \geq \sum_{m \in M_{H,2}} \alpha_m \geq \ell_2 |M_{H,2}|$, so that $|M_{H,2}| \leq \frac{n}{\ell_2}$. Therefore $\sum_{m \in M_{H,2}} \sqrt{\alpha_m} \leq |M_{H,2}| \sqrt{\ell_1} \leq \frac{n\sqrt{\ell_1}}{\ell_2}$. Thus 
\begin{align*}
 \sum_{m \in M_{H,2}} \frac{\varepsilon_m^2 }{\ell_1\ell_2} \alpha_m \beta_m \geq  \sqrt{\ell_2} \frac{\bigg(\sum_{m \in M_{H,2}} \frac{\varepsilon_m }{\sqrt{\ell_1\ell_2}} \alpha_m\bigg)^2}{\min(\sqrt{nd}, \frac{n \sqrt{\ell_1}}{\ell_2})} \geq \frac{n^{3/2}\eta^2}{36 \ell_1\sqrt{\ell_2}\min(\sqrt{d}, \frac{ \sqrt{n\ell_1}}{\ell_2})}. 
\end{align*}
Note that in this case we have to have
\begin{align*}
\frac{n \eta}{6\sqrt{\ell_1\ell_2}}\leq \sum_{m \in M_{H,2}} \frac{\varepsilon_m }{\sqrt{\ell_1\ell_2}} \alpha_m \leq \frac{|M_{H,2}|}{\sqrt{\ell_1\ell_2}} \ell_1 \leq \frac{d \ell_1}{\sqrt{\ell_1\ell_2}},
\end{align*}
where we used the fact that $\varepsilon_m \leq 1$. Thus this case cannot happen when $n \eta \geq 6 d\ell_1$.

\item In the last sub-case we have that $\beta_m = \alpha_m$, and the maximum is achieved at $i = 3$ so that we have
$$
\sum_{m \in M_{H,3}} \frac{\varepsilon_m }{\sqrt{\ell_1\ell_2}} \alpha_m \geq \frac{n}{6\sqrt{\ell_1\ell_2}} \eta.
$$
By the AM-GM inequality we have that 
\begin{align*}
\sum_{m \in M_{H,3}} \frac{\varepsilon_m^2 }{\ell_1\ell_2} \alpha_m \beta_m &= \sum_{m \in M_{H,3}} \frac{\varepsilon_m^2 }{\ell_1\ell_2} \alpha_m^2 \geq \frac{\bigg(\sum_{ m \in M_{H,3}} \frac{\varepsilon_m \alpha_m}{\sqrt{\ell_1\ell_2} }\bigg)^2}{|M_{H,3}|} \geq \frac{n^2 \eta^2}{36 \ell_1\ell_2|M_{H,3}|}.
\end{align*}
We have a simple bound on the cardinality $|M_{H,3}| \leq d$ (since it cannot be more than the total number of categories). Thus we conclude the bound
\begin{align*}
\sum_{m \in M_{H,3}} \frac{\varepsilon_m^2 }{\ell_1\ell_2} \alpha_m \beta_m &= \sum_{m \in M_{H,3}} \frac{\varepsilon_m^2 }{\ell_1\ell_2} \alpha_m^2 \geq \frac{n^2 \eta^2}{36 \ell_1\ell_2d}.
\end{align*}
Similarly to before in this case we must have 
\begin{align*}
\frac{n \eta}{6\sqrt{\ell_1\ell_2}}\leq \sum_{m \in M_{H,3}} \frac{\varepsilon_m }{\sqrt{\ell_1\ell_2}} \alpha_m \leq \frac{|M_{H,3}|}{\sqrt{\ell_1\ell_2}} \ell_2 \leq \frac{d \ell_2}{\sqrt{\ell_1\ell_2}},
\end{align*}
i.e., when $n \eta \geq 6 d\ell_2$.
\end{itemize}
Finally we handle the last case in which we have
$$
\sum_{m \in M_L} \frac{\varepsilon_m }{\sqrt{\ell_1\ell_2}} \alpha_m \geq \frac{n}{2\sqrt{\ell_1\ell_2}} \eta.
$$
Here we have that, by Jensen's inequality:
\begin{align*}
\frac{1}{\ell_1\ell_2}\sum_{m \in M_L} \varepsilon_m^2  \alpha_m^4 \geq \frac{1}{\ell_1\ell_2}\frac{(\sum_{m \in M_L}  \varepsilon_m \alpha_m)^4}{(\sum_{m \in M_L}\varepsilon_m^{2/3})^3} \geq \frac{n^4 \eta^4 }{\ell_1 \ell_2 16(\sum_{m \in M_L}\varepsilon_m^{2/3})^3}.
\end{align*}
Now using that $\varepsilon_m \leq 1$ we have $\sum_{m \in M_L}\varepsilon_m^{2/3} \leq |M_L| \leq d$ thus we conclude that
\begin{align*}
\frac{1}{\ell_1\ell_2}\sum_{m \in M_L} \varepsilon_m^2  \alpha_m^4 \geq \frac{n^4 \eta^4 }{16\ell_1 \ell_2  d^3}.
\end{align*}
Finally in this case we have to have  
\begin{align*}
\frac{n \eta}{6\sqrt{\ell_1\ell_2}}\leq \sum_{m \in M_{L}} \frac{\varepsilon_m }{\sqrt{\ell_1\ell_2}} \alpha_m \leq \frac{|M_{L}|}{\sqrt{\ell_1\ell_2}} 1 \leq \frac{d}{\sqrt{\ell_1\ell_2}},
\end{align*}
or equivalently when $n \eta \geq 6d$ where we used that $\alpha_m^3 \leq \beta_m$ implies that $\alpha_m \leq 1$. 

We will now select a threshold at the level of $\zeta \sqrt{d}$, and will give conditions on the minimum sample size for each of the cases. We will use $\gtrsim$ in the sense bigger up to an absolute constant. We will assume that $\frac{\varepsilon}{2} - \frac{3}{2} \frac{L}{d} \geq \frac{\varepsilon}{4}$ so that $\eta \geq \frac{\varepsilon}{4}$. 
\begin{itemize}
\item In the first sub-case we have to satisfy $\frac{n \eta^2}{\sqrt{\ell_1\ell_2}} \gtrsim \sqrt{\zeta d}$. This is ensured when 
\begin{align}\label{epsilon:condition:one}
\varepsilon \gtrsim \sqrt{\frac{\sqrt{\zeta d\ell_1 \ell_2}}{n}} \vee \frac{1}{d}.
\end{align}
\item In the second sub-case we have $\frac{n^{3/2}\eta^2}{ \ell_1\sqrt{\ell_2}\frac{ \sqrt{n\ell_1}}{\ell_2}} \gtrsim \sqrt{\zeta d}$ or $\frac{n^{3/2}\eta^2}{ \ell_1\sqrt{\ell_2}\sqrt{d}} \gtrsim \sqrt{\zeta d},$
This is implied when 
\begin{align}\label{epsilon:condition:two}
\varepsilon \gtrsim \min\bigg(\sqrt{ \frac{\sqrt{\zeta d}\ell_1^{3/2}}{\sqrt{\ell_2}n}}, \sqrt{\frac{\sqrt{\zeta} d \ell_1 \sqrt{\ell_2}}{n^{3/2}}}, \frac{d \ell_1}{n} \bigg) \vee \frac{1}{d}.
\end{align}
\item In the third sub-case case we need $\frac{n^2 \eta^2}{36 \ell_1\ell_2d} \gtrsim \sqrt{\zeta d}$ 
which happens when 
\begin{align}\label{epsilon:condition:three}
\varepsilon \gtrsim \min\bigg(\frac{ \zeta^{1/4} d^{3/4}\sqrt{\ell_1\ell_2}}{n}, \frac{d\ell_2}{n}\bigg) \vee \frac{1}{d},
\end{align}
where the last condition enforces when this case is not feasible. 
\item Finally in the second case we need $\frac{n^4 \eta^4 }{16\ell_1 \ell_2  d^3} \gtrsim \sqrt{\zeta d}$, which is implied when 
\begin{align}\label{epsilon:condition:four}
\varepsilon \gtrsim \min\bigg(\frac{\zeta^{1/8} d^{7/8} (\ell_1\ell_2)^{1/4}}{n}, \frac{d}{n}\bigg)\vee \frac{1}{d}
\end{align}
\end{itemize}

\noindent \textbf{Analysis of the Variance.} Now we derive a bound on the variance of the statistic 
$$
\sum_{m \in [d]}\sigma_m \omega_m\mathbbm{1}(\sigma_m \geq 4) \frac{\varepsilon^2_m}{\ell_1\ell_2}
$$
Recall now that $\sigma_m \sim \operatorname{Poi}(\alpha_m)$ are independent and therefore
\begin{align}
\Var\bigg[\sum_{m \in [d]}\sigma_m \omega_m\mathbbm{1}(\sigma_m \geq 4) \frac{\varepsilon^2_m}{\ell_1\ell_2}\bigg] & = \sum \Var\bigg[\sigma_m \omega_m\mathbbm{1}(\sigma_m \geq 4) \frac{\varepsilon^2_m}{\ell_1\ell_2}\bigg] \nonumber \\
& = \sum_{m \in [d]}\frac{\varepsilon^4_m}{\ell_1^2\ell^2_2} \Var (\sigma_m \omega_m\mathbbm{1}(\sigma_m \geq 4)) \nonumber \\
& \leq \frac{C'}{\ell_1\ell_2} \sum_{m \in [d]}\frac{\varepsilon^2_m}{\ell_1\ell_2} \EE (\sigma_m \omega_m\mathbbm{1}(\sigma_m \geq 4)) \nonumber\\
& \leq \frac{C'}{\ell_1\ell_2} \EE\bigg[\sum_{m \in [d]}\sigma_m \omega_m\mathbbm{1}(\sigma_m \geq 4) \frac{\varepsilon^2_m}{\ell_1\ell_2}\bigg],\label{D:variance:bound}
\end{align}
where in the next to last inequality we used Claim 2.2. of \cite{canonne2018testing}, and $C'$ is an absolute constant described in that claim. 

We now bound the variance of the statistic $T$ (recall the definition \eqref{T:stat:def}). Since $T_m$ (recall definition \ref{Tm:def}) are independent given $\sigma_m, R_m$ we have that
$$
\Var[T | \sigma, R] = \sum_{m \in [d]} \Var[T_m | \sigma_m, R_m].
$$
Next, by definition of $T_m$ we have that $ \Var[T_m | \sigma_m, R_m] = \sigma^2_m \omega^2_m \mathbbm{1}(\sigma_m \geq 4) \Var[U_m | R_m]$. Using the bound on the variance $\Var[U_m | R_m]$ \eqref{bound:on:the:variance}, we have to control four terms. We do so below. Denote
$$
E := \sum_{m \in [d]} \omega_m^2 \|q_{\Pi,A_m}(m)\|_2^2\mathbbm{1}(\sigma_m \geq 4)
$$

The first term we need to control is 
\begin{align*}
\MoveEqLeft \sum_{m \in [d]} \sigma^2_m \omega^2_m \mathbbm{1}(\sigma_m \geq 4) \frac{ \|q_{A_m}(m) - q_{\Pi,A_m}(m)\|_2^2\|q_{\Pi,A_m}(m)\|_2}{\sigma_m} \\
&\leq \sqrt{\bigg(\sum_{m \in [d]} \omega_m^2 \|q_{\Pi,A_m}(m)\|_2^2\mathbbm{1}(\sigma_m \geq 4)\bigg)}\sqrt{\sum_{m \in [d]} (\sigma_m \omega_m\|q_{A_m}(m) - q_{\Pi,A_m}(m)\|_2^2)^2 \mathbbm{1}(\sigma_m \geq 4)}\\
&\leq E^{1/2}\sum_{m \in [d]} \sigma_m \omega_m\|q_{A_m}(m) - q_{\Pi,A_m}(m)\|_2^2 \mathbbm{1}(\sigma_m \geq 4)\\
& = E^{1/2} \EE[T | \sigma, R],
\end{align*}
where we used Cauchy-Schwarz and the monotonicity of $L_p$ norms. The second term is 
\begin{align*}
\MoveEqLeft \sum_{m \in [d]} \sigma^2_m \omega^2_m \mathbbm{1}(\sigma_m \geq 4) \frac{ \|q_{A_m}(m) - q_{\Pi,A_m}(m)\|_2^3}{\sigma_m} = \sum_{m \in [d]} \sqrt{\frac{\omega_m}{\sigma_m}}\sigma^{3/2}_m \omega^{3/2}_m \mathbbm{1}(\sigma_m \geq 4)  \|q_{A_m}(m) - q_{\Pi,A_m}(m)\|_2^3 \\
& \leq \sum_{m \in [d]} (\sigma_m \omega_m \mathbbm{1}(\sigma_m \geq 4)  \|q_{A_m}(m) - q_{\Pi,A_m}(m)\|_2^2)^{3/2}\\
& \leq \bigg(\sum_{m \in [d]} \sigma_m \omega_m \mathbbm{1}(\sigma_m \geq 4)  \|q_{A_m}(m) - q_{\Pi,A_m}(m)\|_2^2\bigg)^{3/2}\\
& = \EE[T | \sigma, R]^{3/2},
\end{align*}
where we used that $\omega_m \leq \sigma_m$ by definition and the monotonicity of the $L_p$ norms. The third term is
\begin{align*}
\MoveEqLeft \sum_{m \in [d]} \sigma^2_m \omega^2_m \mathbbm{1}(\sigma_m \geq 4) \frac{ \|q_{\Pi,A_m}(m)\|_2^2}{\sigma^2_m} = E.
\end{align*}
Finally the fourth term is 
\begin{align*}
\MoveEqLeft \sum_{m \in [d]} \sigma^2_m \omega^2_m \mathbbm{1}(\sigma_m \geq 4) \frac{ \|q_{A_m}(m) - q_{\Pi,A_m}(m)\|_2^2}{\sigma^2_m} \leq \sum_{m \in [d]} \omega_m\sigma_m \mathbbm{1}(\sigma_m \geq 4) \|q_{A_m}(m) - q_{\Pi,A_m}(m)\|_2^2\\
& = \EE[T | \sigma, R]. 
\end{align*}
We conclude that
\begin{align}\label{varT:bound}
\Var[T | \sigma, R] \leq C(E + (E^{1/2} + 1)\EE[T | \sigma, R] + \EE[T | \sigma, R]^{3/2}).
\end{align}
Now we will show that $\EE[E | \sigma] = O(\min(d, N))$. We start by analyzing the expectation of one term from $E$ below.
\begin{align*}
\EE[\omega_m^2 \|q_{\Pi,A_m}(m)\|_2^2\mathbbm{1}(\sigma_m \geq 4) | \sigma_m] = \omega_m^2\mathbbm{1}(\sigma_m \geq 4)\EE[ \|q_{\Pi,A_m}(m)\|_2^2 | \sigma_m] \leq  \frac{\omega_m^2\mathbbm{1}(\sigma_m \geq 4) }{(1 + t_{1,m})(1 + t_{2,m})},
\end{align*}
where we applied \eqref{bound:on:l2:norm:qpi}. Recall that $t_{i,m} = \min((\sigma_m - 4)/4, \ell_i)$ and $\omega^2_m = \min(\sigma_m, \ell_1)\min(\sigma_m, \ell_2)$. Thus 
$$
\frac{\omega_m^2}{(1 + t_{1,m})(1 + t_{2,m})} \leq O(1).
$$
We conclude that 
\begin{align}\label{E:given:sigma:bound}
\EE[E | \sigma] = \EE[\sum_{m \in [d]} \omega_m^2 \|q_{\Pi,A_m}(m)\|_2^2\mathbbm{1}(\sigma_m \geq 4) | \sigma_m] \leq O(1)\sum_{m \in [d]} \mathbbm{1}(\sigma_m \geq 4)  \leq O(1)\min(d,N).
\end{align}

We have the following result

\begin{lemma}\label{variance:bound:flattening}Suppose $\inf_{q \in \cP_{0, [0,1]}'}\|p_{X,Y,Z} - q\|_1 \geq \varepsilon$, where $\varepsilon \geq  3\frac{L}{d}$ and it satisfies conditions \eqref{epsilon:condition:one}, \eqref{epsilon:condition:two}, \eqref{epsilon:condition:three} and \eqref{epsilon:condition:four}. Then with probability at least $19/20$ over $\sigma, R$ we have $\EE [T | \sigma, R] = \Omega(\sqrt{\zeta d})$ and 
\begin{align}\label{variance:bound:flattening:equation}
\Var[T | \sigma, R] \leq O(d + (\sqrt{d} + 1)\EE[T | \sigma, R] + \EE[T | \sigma, R]^{3/2}).
\end{align}
\end{lemma}

\begin{proof} Let 
$$
D = \sum_{m \in [d]}\sigma_m \omega_m\mathbbm{1}(\sigma_m \geq 4) \frac{\varepsilon^2_m}{\ell_1\ell_2}.
$$
We first showed that $\EE[T | \sigma, R] \geq D$ for all $\sigma, R$ \eqref{Tm:lower:bound}. We also derived that $\Var[D] \leq O(\EE[D]/(\ell_1\ell_2))$ \eqref{D:variance:bound}, and that for the selected regimes of sample size $\EE[D] \gtrsim \sqrt{\zeta d}$. Therefore we have 
\begin{align*}
\PP_{\sigma, R}\bigg(\EE[T | \sigma, R] \leq \kappa \sqrt{\zeta d}\bigg)\leq \PP_{\sigma, R}(D \leq O(\EE[D])) \leq O\bigg(\frac{\Var[D]}{(\EE[D])^2}\bigg) = O(1/(\sqrt{\zeta d}\ell_1 \ell_2)) \leq 1/40,
\end{align*}
for some small enough absolute constant $\kappa$. For the second statement we will use bound \eqref{varT:bound}. By \eqref{E:given:sigma:bound} we have  
$$
\EE[\EE[E | \sigma]] \leq O(1)\EE \min(d,N) \leq O(1) \min(d, n). 
$$
Thus by Markov's inequality $E \leq 200 \EE[E] = O(1) d$ with probability at least $39/40$. Therefore
$$
\PP_{\sigma,R}(\Var[T | \sigma, R] \geq \kappa'(d + (\sqrt{d} +1)\EE[T | \sigma, R] + \EE[T | \sigma, R]^{3/2})) \leq 1/40.
$$
A union bound over the two events completes the proof. 
\end{proof}

We now turn to bound the expectation and variance under the null hypothesis. 

\begin{lemma}\label{second:variance:bound:flattening}  Suppose $p_{X,Y,Z} \in \cP_{0,[0,1], \chi^2}'(L)$. Let further $\ell_1 \geq \ell_2$ be such that $\ell_1 d \lesssim n$ (here $\lesssim$ means smaller up to an absolute constant). Then with probability at least $19/20$ we have $\EE[T | \sigma, R] \leq C n\frac{L^2}{\sqrt{\ell_1\ell_2}d^2}$ and the variance $\Var[T | \sigma, R]$ satisfies \eqref{variance:bound:flattening:equation}. 
\end{lemma}

\begin{proof}
Let us start with bounding $\EE[T | \sigma,R]$ from above. Recall that 
$$
\EE[T | \sigma, R] = \sum_{m \in [d]} \sigma_m \omega_m \|q_{A_m}(m) - q_{\Pi,A_m}(m)\|_2^2\mathbbm{1}(\sigma_m \geq 4)
$$
We will now control $\EE[\|q_{A_m}(m) - q_{\Pi,A_m}(m)\|_2^2 | \sigma_m]$. Recall that
\begin{align}\label{to:upper:bound}
 \|q_{A_m}(m) - q_{\Pi,A_m}(m)\|_2^2 & = \sum_{x,y} \frac{(q_{xy}(m) - q_{x\cdot}(m) q_{\cdot y}(m))^2}{1 + a^m_{xy}}.
\end{align}
Since by \eqref{flattening:expectation:bound} we have 
$$
\EE_{A_m} \frac{1}{1 + a^m_{xy}} \leq \frac{1}{(1 + t_{1,m})(1 + t_{2,m})q_{x\cdot}(m)q_{\cdot y}(m)} = \frac{O(1)}{ \omega_m^2 q_{x\cdot}(m)q_{\cdot y}(m)},
$$
we have that 
\begin{align}\label{intermediate:bound:null:hypothesis:lemma}
\EE_{A_m}  \|q_{A_m}(m) - q_{\Pi,A_m}(m)\|_2^2 \leq \frac{O(1)}{\omega_m^2} \sum_{x,y}  \frac{(q_{xy}(m) - q_{x\cdot}(m) q_{\cdot y}(m))^2}{q_{x\cdot}(m)q_{\cdot y}(m)}.
\end{align}
We will now focus on controlling the RHS. Using the fact that under the null hypothesis $p_{x,y|z = z} = p_{X|Z = z} p_{Y|Z = z}$ we have
\begin{align*}
\MoveEqLeft \sum_{x,y}  \frac{(q_{xy}(m) - q_{x\cdot}(m) q_{\cdot y}(m))^2}{q_{x\cdot}(m)q_{\cdot y}(m)} \\
& = \sum_{x,y} \frac{(\int (p_{X|Z}(x|z) - \int p_{X|Z}(x | z) d \tilde P(z)) (p_{Y|Z}(y|z) - \int p_{Y|Z}(y|z) d\tilde P(z)) d \tilde P(z))^2}{\int p_{X|Z}(x|z) d\tilde P(z)\int p_{Y|Z}(y|z) d\tilde P(z)}\\
& \leq \sum_{x} \frac{\int  p^2_{X|Z}(x | z) d\tilde P(z) - (\int  p_{X|Z}(x | z) d \tilde P(z))^2}{\int  p_{X|Z}(x | z) d\tilde P(z)} \sum_y \frac{\int p^2_{Y|Z}(y|z) d\tilde P(z) - (\int p_{Y|Z}(y|z) d\tilde P(z))^2}{\int p_{Y|Z}(y|z) d\tilde P(z)}.
\end{align*}
We now handle the first term on the RHS, the second one being analogous. 
\begin{align*}
\sum_{x} \frac{\int p^2_{X|Z}(x | z) d\tilde P(z) - (\int p_{X|Z}(x | z) d \tilde P(z))^2}{\int p_{X|Z}(x | z) d\tilde P(z)} & = \sum_{x} \frac{\int p^2_{X|Z}(x | z) d\tilde P(z)}{\int p_{X|Z}(x | z) d\tilde P(z)} - 1 \\
& = \int \sum_{x} \frac{p^2_{X|Z}(x | z)}{\int p_{X|Z}(x | z)d\tilde P(z)} d \tilde P(z) - 1.
\end{align*}
We now show that the function $z \mapsto \sum_{x} \frac{p^2_{X|Z}(x | z)}{\int p_{X|Z}(x | z) d\tilde P(z)}$ is continuous. This follows from the fact that each of $z \mapsto p_{X|Z}(x | z)$ is continuous. To see why $p_{X|Z}(x | z)$ is continuous first recall that $\|p_{X|Z= z} - p_{X|Z = z'}\|^2_1 \leq d_{\chi^2}(p_{X|z}, p_{X|z'}) \leq L |z - z'|$, which shows the continuity of each of $p_{X|Z}(x|z)$ for all $x$. Hence by the mean value theorem we have that 
\begin{align*}
\sum_{x} \frac{\int p^2_{X|Z}(x | z) d\tilde P(z)}{\int p_{X|Z}(x | z) d\tilde P(z)} - 1 = \int \sum_{x} \frac{p^2_{X|Z}(x | z)}{\int p_{X|Z}(x | z) d\tilde P(z)} d \tilde P(z) - 1= \sum_{x} \frac{p^2_{X|Z}(x | \tilde z)}{\int p_{X|Z}(x | z) d\tilde P(z)} -1 ,
\end{align*}
for some $\tilde z \in C_m$. Next since $x \mapsto \frac{1}{x}$ is convex on the positive reals, by Jensen's inequality we have 
\begin{align*}
\sum_{x} \frac{p^2_{X|Z}(x | \tilde z)}{\int p_{X|Z}(x | z) d\tilde P(z)} -1 & \leq \int \sum_{x} \frac{p^2_{X|Z}(x | \tilde z)}{ p_{X|Z}(x | z) } -1 d\tilde P(z)= \int d_{\chi^2}(p_{x|\tilde z}, p_{x|z}) d\tilde P(z)\\
&  \leq L \diam(C_m) = \frac{L}{d}. 
\end{align*}
We get a similar bound on the second term which implies that
\begin{align*}
\sum_{x,y}  \frac{(q_{xy}(m) - q_{x\cdot}(m) q_{\cdot y}(m))^2}{q_{x\cdot}(m)q_{\cdot y}(m)} \leq \frac{L^2}{d^2}.
\end{align*}
Hence, combining the last observation with \eqref{intermediate:bound:null:hypothesis:lemma} we have
\begin{align*}
\EE[T | \sigma] & = \sum_{m \in [d]} \sigma_m \omega_m\mathbbm{1}(\sigma_m \geq 4) \EE_{A_m} \|q_{A_m}(m) - q_{\Pi,A_m}(m)\|_2^2 \\
&\leq C  \frac{L^2}{d^2} \sum_{m \in [d]} \frac{\sigma_m}{\omega_m}\mathbbm{1}(\sigma_m \geq 4). 
\end{align*}
Recall that $\omega_m = \sqrt{\min (\sigma_m, \ell_1)\min(\sigma_m, \ell_2)}$. Recall that we are supposing (without loss of generality) $\ell_1 \geq \ell_2$. We therefore have
\begin{align*}
\sum_{m \in [d]} \frac{\sigma_m}{\omega_m}\mathbbm{1}(\sigma_m \geq 4) & \leq \sum_{m : \sigma_m \leq \ell_2} \mathbbm{1}(\sigma_m \geq  4) + \sum_{m : \ell_2 < \sigma_m \leq \ell_1} \frac{\sqrt{\sigma_m}}{\sqrt{\ell_2}}\mathbbm{1}(\sigma_m \geq  4) + \sum_{m : \sigma_m > \ell_1} \frac{\sigma_m}{\sqrt{\ell_1 \ell_2}} \mathbbm{1}(\sigma_m \geq  4)\\
& \leq d + d \sqrt{\frac{\ell_1}{\ell_2}} + \frac{N}{\sqrt{\ell_1 \ell_2}}
\end{align*}

Taking expectation it follows that 
\begin{align*}
\EE[T] \leq C \frac{L^2}{d^2} \bigg(d + d \sqrt{\frac{\ell_1}{\ell_2}} + \frac{n}{\sqrt{\ell_1 \ell_2}}\bigg) \leq C \frac{L^2}{d^2} \bigg(2d \sqrt{\frac{\ell_1}{\ell_2}} + \frac{n}{\sqrt{\ell_1 \ell_2}}\bigg) \leq \frac{C' n}{d^2 \sqrt{\ell_1 \ell_2}},
\end{align*}
for some constant $C'$ which depends on $L$, and we used the fact that $d \ell_1 \lesssim n$. It follows from Markov's inequality that
$$
\PP_{\sigma, R} \bigg(\EE[T | \sigma,R]  \geq 40 \frac{C' n}{d^2 \sqrt{\ell_1 \ell_2}}\bigg) \leq \frac{1}{40}. 
$$
The second part follows directly by Lemma \ref{variance:bound:flattening}. 
\end{proof}

\noindent \textbf{Putting Things Together.} For what follows suppose that $d$ is selected so that 
\begin{align}\label{d:n:condition}
\frac{n}{d^2\sqrt{\ell_1\ell_2}} \asymp \sqrt{d}.
\end{align}

\begin{lemma}\label{null:hypothesis:lemma} If $p_{X,Y,Z} \in \cP'_{0,[0,1], \chi^2}(L)$ and that \eqref{d:n:condition} holds. Then for a sufficiently large absolute constant $\alpha$ we have
\begin{align*}
\PP\bigg(T \geq  (\alpha + 1) \frac{C' n}{d^2 \sqrt{\ell_1\ell_2}}\bigg) \leq \frac{1}{10}.
\end{align*}
\end{lemma}

\begin{proof}
Let $T' = (T | \sigma, R)$. Denote the event from Lemma \ref{second:variance:bound:flattening} with $\cE$. Then we have 
\begin{align*}
\PP\bigg(T \geq  (\alpha + 1) \frac{C' n}{d^2 \sqrt{\ell_1\ell_2}}\bigg) & = \PP\bigg(T'  \geq  (\alpha + 1)\frac{C' n}{\sqrt{\ell_1\ell_2}d^2}\bigg) \\
& \leq  \PP\bigg(T' \geq (\alpha + 1) \frac{C' n}{\sqrt{\ell_1\ell_2} d^2} \bigg | \cE\bigg) + \PP(\cE^c). 
\end{align*}
Now we have
\begin{align*}
\PP\bigg(T' \geq  (\alpha + 1) \frac{C' n}{\sqrt{\ell_1\ell_2} d^2} \bigg| \cE\bigg) & \leq \PP\bigg(T' - \EE[T |\sigma, R] \geq \alpha \frac{C'n}{\sqrt{\ell_1\ell_2} d^2} \bigg| \cE\bigg) \leq \frac{\Var[T' | \cE]}{(\alpha \frac{C'n}{\sqrt{\ell_1\ell_2}d^2})^2}\\
& \leq \frac{O(d + (\sqrt{d} + 1)\EE[T | \sigma, R] + \EE[T | \sigma, R]^{3/2})}{(\alpha \frac{C'n}{\sqrt{\ell_1\ell_2}d^2})^2}\\
& \leq \frac{O(d + (\sqrt{d} + 1)(\frac{C' n}{\sqrt{\ell_1\ell_2}d^2}) + (\frac{C'n}{\sqrt{\ell_1\ell_2}d^2})^{3/2})}{(\alpha \frac{C'n}{\sqrt{\ell_1\ell_2}d^2})^2}\\
& \leq \frac{1}{20}
\end{align*}
where the above holds when $\frac{n}{\sqrt{\ell_1\ell_2}d^2} \asymp \sqrt{d}$ for a large enough $\alpha$.
\end{proof}

\begin{lemma}\label{alternative:hypothesis:lemma}
If $p_{X,Y,Z}$ is such that $\inf_{q \in \cP_{0,[0,1]}'}\|p_{X,Y,Z}-q\|_1 \geq \varepsilon$, and the conditions of Lemma \ref{variance:bound:flattening} hold. Then for a small enough absolute constant $\kappa$ we have that
\begin{align*}
\PP(T \leq \kappa \sqrt{\zeta d}) \leq \frac{1}{10}. 
\end{align*} 
\end{lemma}
\begin{proof}
We apply Chebyshev's inequality to $T' = (T | \sigma, R)$. Let $\cE$ be the event of Lemma \ref{variance:bound:flattening}. Set $\tau = \kappa \sqrt{\zeta d}$ for some small enough absolute constant $\kappa$. 
\begin{align*}
\PP(T \leq \tau) &= \PP(T' \leq \tau) \leq \PP(T'\leq \tau| \cE) + \PP(\cE^c). \\
&\leq \PP\bigg(|T' - \EE[T|\sigma, R]| \geq \frac{1}{2} \EE[T | \sigma, R] | \cE\bigg) + \frac{1}{20}.
\end{align*}
Next
\begin{align*}
\PP\bigg(|T' - \EE[T|\sigma, R]| \geq \frac{1}{2} \EE[T | \sigma, R] | \cE\bigg) \leq O\bigg(\frac{d + (\sqrt{d} + 1)\EE[T | \sigma, R] + \EE[T | \sigma, R]^{3/2}}{\EE[T | \sigma, R]^2}\bigg)\\
= O(1/\zeta^{1/2}) \leq \frac{1}{20},
\end{align*}
for a large enough value of $\zeta$. 
\end{proof}

Combining Lemmas \ref{null:hypothesis:lemma} and \ref{alternative:hypothesis:lemma} we have that if $\kappa \sqrt{\zeta d} \geq (\alpha + 1) n C'/(\sqrt{\ell_1\ell_2}d^2)\asymp (\alpha +1) C' \sqrt{d}$, there will be a gap between the values under the null and the alternative hypothesis. This happens when $\zeta$ is large enough. Notice that $\frac{n C'}{\sqrt{\ell_1\ell_2} d^2} \asymp \sqrt{d}$ is equivalent to $d \asymp \frac{C'^{2/5} n^{2/5}}{(\ell_1\ell_2)^{1/5}}$. 
Plugging this in all the inequalities \eqref{epsilon:condition:one}, \eqref{epsilon:condition:two}, \eqref{epsilon:condition:three} and \eqref{epsilon:condition:four}, results in new inequalities that need to hold. We list those below, and in addition we recall that $d \ell_1 \lesssim n$. Condition \eqref{epsilon:condition:one} is equivalent to
\begin{align*}
\varepsilon \gtrsim \frac{(\ell_1 \ell_2)^{1/5}}{n^{2/5}}. 
\end{align*}
Taking the second term of \eqref{epsilon:condition:two}, and using the assumption that $d \ell_1 \lesssim n$  we have
\begin{align*}
\varepsilon \gtrsim  \frac{(\ell_1 \ell_2)^{1/5}}{n^{2/5}},
\end{align*}
since $\sqrt{\frac{d \ell_1 \sqrt{\ell_2}}{n^{3/2}}} = \sqrt[4]{\frac{d \ell_1}{n}}\sqrt{\frac{\sqrt{d \ell_1 \ell_2}}{n}}$. Taking \eqref{epsilon:condition:three}, we have
\begin{align*}
\varepsilon \gtrsim \bigg(\frac{\sqrt{\ell_1 \ell_2}}{n}\bigg)^{7/10},
\end{align*}
which is of smaller order than $\frac{(\ell_1 \ell_2)^{1/5}}{n^{2/5}}$ whenever $\frac{\sqrt{\ell_1 \ell_2}}{n} \lesssim 1$. Finally for \eqref{epsilon:condition:four} we have
\begin{align*}
\varepsilon \gtrsim \bigg(\frac{\sqrt{\ell_1\ell_2}}{n}\bigg)^{13/20} \frac{1}{\sqrt[4]{\ell_1 \ell_2}},
\end{align*}
which is also of smaller order than $\frac{(\ell_1 \ell_2)^{1/5}}{n^{2/5}}$ whenever $\frac{\sqrt{\ell_1 \ell_2}}{n} \lesssim 1$. This completes the proof. 
\end{proof}


\subsection{Proofs from Section \ref{cont:case:upper:bounds:section}}

We start this section with showing the following result:

\begin{lemma}There exists a constant $C$ depending only on $(s, L)$ such that if $p_{X,Y|Z}(x,y|z) \in \cH^{2,s}(L)$ it follows that $p_{X|Z}(x|z)p_{Y|Z}(y|z) \in \cH^{2,s}(C)$.
\end{lemma}

\begin{proof} We will start by showing that $p_{X|Z}(x|z)$ is H\"{o}lder smooth in the $x$ dimension. We have, by the Leibniz rule and Jensen's inequality
\begin{align*}
\bigg|\frac{\partial^{\lfloor s \rfloor}}{\partial x^{\lfloor s \rfloor}}p_{X|Z}(x|z) - \frac{\partial^{\lfloor s \rfloor}}{\partial x^{\lfloor s \rfloor}}p_{X|Z}(x'|z)\bigg| \leq \int_{0}^1 \bigg|\frac{\partial^{\lfloor s \rfloor}}{\partial x^{\lfloor s \rfloor}} p_{XY |Z}(x,y|z) - \frac{\partial^{\lfloor s \rfloor}}{\partial x^{\lfloor s \rfloor}} p_{XY |Z}(x',y|z) \bigg| dy  \leq L |x - x'|^{s - \lfloor s \rfloor}.
\end{align*}
In addition any derivative of a lower or equal to $0 \leq k \leq \lfloor s \rfloor$ order is bounded since (by the Leibniz rule and Jensen's inequality)
\begin{align*}
\bigg|\frac{\partial^{k}}{\partial x^{k}}p_{X|Z}(x|z) \bigg| \leq \int_0^1 \bigg|\frac{\partial^{k}}{\partial x^{k}} p_{XY |Z}(x,y|z)\bigg | dy \leq L.
\end{align*}
By symmetry the same statement holds for $p_{Y|Z}(y|z)$. Next we show that the product $p_{X|Z}(x|z)p_{Y|Z}(y|z) \in \cH^{2,s}(C)$ for a sufficiently large $C$. First we argue that the derivatives are bounded. We have
\begin{align*}
\bigg|\frac{\partial^k}{\partial x^k}\frac{\partial^{\lfloor s \rfloor - k}}{\partial y^{\lfloor s \rfloor - k}} p_{X|Z}(x|z)p_{Y|Z}(y|z) \bigg|= \bigg| \frac{\partial^k}{\partial x^k}p_{X|Z}(x|z)\bigg| \bigg|\frac{\partial^{\lfloor s \rfloor - k}}{\partial y^{\lfloor s \rfloor - k}}p_{Y|Z}(y|z) \bigg| \leq L^2. 
\end{align*}
Next, take any $0 < k < \lfloor s \rfloor$. We have
\begin{align*}
\MoveEqLeft \bigg|\frac{\partial^k}{\partial x^k}\frac{\partial^{\lfloor s \rfloor - k}}{\partial y^{\lfloor s \rfloor - k}} p_{X|Z}(x|z)p_{Y|Z}(y|z) - \frac{\partial^k}{\partial x^k}\frac{\partial^{\lfloor s \rfloor - k}}{\partial y^{\lfloor s \rfloor - k}} p_{X|Z}(x'|z)p_{Y|Z}(y'|z)\bigg| \\
& \leq \bigg|\frac{\partial^k}{\partial x^k} p_{X|Z}(x|z) - \frac{\partial^k}{\partial x^k} p_{X|Z}(x'|z)\bigg| \bigg|\frac{\partial^{\lfloor s \rfloor - k}}{\partial y^{\lfloor s \rfloor - k}}p_{Y|Z}(y|z)\bigg| \\
& + \bigg|\frac{\partial^k}{\partial x^k} p_{X|Z}(x'|z)\bigg|\bigg|\frac{\partial^{\lfloor s \rfloor - k}}{\partial y^{\lfloor s \rfloor - k}}p_{Y|Z}(y|z)- \frac{\partial^{\lfloor s \rfloor - k}}{\partial y^{\lfloor s \rfloor - k}}p_{Y|Z}(y'|z)\bigg|\\
& \leq 2L^2 (1 \wedge |x - x'|) + 2L^2 (1 \wedge |y - y'|)\\
& \leq 2L^2 (|x-x'|^{s -\lfloor s \rfloor } + |y - y'|^{s - \lfloor s \rfloor})\\
& \leq L^2 4 \bigg(\frac{|x-x'|^{2} + |y-y'|^{2}}{2}\bigg)^{\frac{s - \lfloor s \rfloor}{2}},
\end{align*}
where the last inequality follows by the generalized means inequality and the third inequality follows because we can bound
\begin{align*}
\bigg|\frac{\partial^k}{\partial x^k} p_{X|Z}(x|z) - \frac{\partial^k}{\partial x^k} p_{X|Z}(x'|z)\bigg| \leq 2L \wedge L |x - x'|\leq 2L (1 \wedge |x-x'|).
\end{align*}
 Similarly when $k = \lfloor s \rfloor$ we have 
\begin{align*}
\MoveEqLeft \bigg|\frac{\partial^{\lfloor s \rfloor}}{\partial x^{\lfloor s \rfloor}} p_{X|Z}(x|z)p_{Y|Z}(y|z) - \frac{\partial^{\lfloor s \rfloor}}{\partial x^{\lfloor s \rfloor}}  p_{X|Z}(x'|z)p_{Y|Z}(y'|z)\bigg| \\
& \leq \bigg|\frac{\partial^{\lfloor s \rfloor}}{\partial x^{\lfloor s \rfloor}} p_{X|Z}(x|z) -\frac{\partial^{\lfloor s \rfloor}}{\partial x^{\lfloor s \rfloor}}  p_{X|Z}(x'|z)\bigg| | p_{Y|Z}(y|z)| +  |p_{Y|Z}(y|z)- p_{Y|Z}(y'|z)\bigg|\frac{\partial^{\lfloor s \rfloor}}{\partial x^{\lfloor s \rfloor}}  p_{X|Z}(x'|z) \bigg|\\
& \leq L^2 |x-x'|^{s -\lfloor s \rfloor} + 2L^2(1 \wedge |y - y'|)\\
& \leq 2L^2 (|x-x'|^{s -\lfloor s \rfloor } + |y - y'|^{s - \lfloor s \rfloor})\\
& \leq L^2 4 \bigg(\frac{|x-x'|^{2} + |y-y'|^{2}}{2}\bigg)^{\frac{s - \lfloor s \rfloor}{2}}.
\end{align*}
Similar logic shows that the same holds when $k = 0$. This completes the proof. 
\end{proof}

We now formulate a Poissonized version of Theorem \ref{continuous:case:upper:bound:theorem}.

\begin{theorem}[Continuous $X,Y,Z$ Upper Bound]\label{continuous:case:upper:bound:theorem:Poi} Set $d = \lceil n^{2s/(5s + 2)} \rceil$, $d' = \lceil d^{1/s}\rceil$ and set the threshold $\tau = \sqrt{\zeta d}$ for a sufficiently large $\zeta$ (depending on $L$). Define $\cH_0(s) = \cP_{0,[0,1]^3,\TV}(L) \cup \cP_{0,[0,1]^3,\chi^2}(L)$ when $s \geq 1$ and $\cH_0(s) = \cP_{0,[0,1]^3,\chi^2}(L)$ when $s < 1$. Then, for a sufficiently
 large absolute constant $c$ (depending on $\zeta, L$),
 when
 $\varepsilon \geq c n^{-2s/(5s+2)}$,
we have that
\begin{align*}
\sup_{p \in \cH_0(s)} \sum_{k = 0}^\infty \PP(N = k)\EE_p[\psi_\tau(\cD'_k)] & \leq \frac{1}{10},\\ 
\sup_{p \in \{p \in \cQ_{0,[0,1]^3, \TV}(L,s): \inf_{q \in \cP_{0,[0,1]^3}} \|p - q\|_1 \geq \varepsilon\}} \sum_{k = 0}^\infty \PP(N = k)\EE_p[1 - \psi_\tau(\cD'_k)] & \leq \frac{1}{10}.
\end{align*}
\end{theorem}

\begin{proof}[Proof of Theorem \ref{continuous:case:upper:bound:theorem}] The proof is identical to that of Theorem \ref{main:theorem:finite:discrete:XY}. 
\end{proof}

\begin{proof}[Proof of Theorem \ref{continuous:case:upper:bound:theorem:Poi}]
To prove this theorem we will assume that $N \sim Poi(n)$ instead of $Poi(n/2)$ for convenience. Since this changes only constant factors, we can do this WLOG. 

Denote by
$$
q_{ij}(m) := \PP(X \in C'_i, Y \in C'_j | Z \in C_m) = \int_{C_m} \int_{C'_i \times C'_j} p_{X,Y|Z}(x,y|z) dx dy d\tilde P(z),
$$
where $d\tilde P(z) = d P(z)/\PP(Z \in C_m)$. Denote by 
$$
q_{i\cdot}(m) := \sum_{j \in [d]} q_{ij}(m) = \int_{C_m} \sum_{j \in [d]}\int_{C'_i \times C'_j} p_{X,Y|Z}(x,y|z) dx dy d\tilde P(z)  = \int_{C_m} \int_{C'_i \times [0,1]} p_{X,Y|Z}(x,y|z) dx dy d\tilde P(z).
$$
Similarly define $q_{\cdot j}(m)$. Inspection of the proof of Theorem \ref{main:theorem:scaling:discrete:XY}, reveals that we need to re-prove several facts and the proof will hold.

First we need to upper bound \eqref{to:upper:bound} or \eqref{intermediate:bound:null:hypothesis:lemma}, where we now index by $i, j$ instead of $x,y$ for convenience.  We have the following two results which control this expression for each of the two null hypothesis respectively. 
\begin{lemma} 
\label{lem:aa}
Suppose that $p_{X,Y,Z} \in \cP_{0,[0,1]^3, \TV}(L)$. Then
\begin{align*}
\|q_{A_m} - q_{\Pi,A_m}\|_2^2 \leq \bigg(\frac{L}{d}\bigg)^4.
\end{align*}
\end{lemma}
\begin{proof} Since $a_{ij}^m \geq 0$ we have that $\|q_{A_m} - q_{\Pi,A_m}\|_2^2 \leq \|q_{} - q_{\Pi}\|_2^2$. Next by the monotonicity of the $L_p$ norms we have
$$
\|q_{} - q_{\Pi}\|_2^2 = \sum_{i,j} (q_{ij}(m) - q_{i\cdot}(m) q_{\cdot j}(m))^2 \leq \bigg(\sum_{i,j} |q_{ij}(m) - q_{i\cdot}(m) q_{\cdot j}(m)|\bigg)^2.
$$
Recall that  $p \in \cP_{0,[0,1]^3, \TV}(L)$ implies that $\|p_{X|Z = z} - p_{X| Z = z'}\|_1 \leq L|z - z'|$ and $\|p_{Y| Z = z} - p_{Y|Z = z'}\|_1 \leq L|z - z'|$. We have
\begin{align*}
\MoveEqLeft \sum_{i,j}\bigg|\int_{C_m} \int_{C'_i \times C'_j} p_{X|Z}(x|z)p_{Y|Z}(y|z)dx dy d\tilde P(z) - \int_{C_m} \int_{C'_i} p_{X|Z}(x|z) dxd\tilde P(z) \int_{C_m} \int_{C'_j} p_{Y|Z}(y|z)dy d\tilde P(z) \bigg|\\
&\leq \int_{C_m} \sum_i\bigg|\int_{C'_i } p_{X|Z}(x|z) dx- \int_{C_m} \int_{C'_i} p_{X|Z}(x|z) dx d\tilde P(z)\bigg| \times \\
&\sum_{j}\bigg|\int_{C'_j} p_{Y|Z}(y|z)dy-\int_{C_m} \int_{C'_j} p_{Y|Z}(y|z)dy d\tilde P(z)\bigg|d\tilde P(z) 
\end{align*}
Take the first summation, and apply Jensen's inequality to conclude
\begin{align*}
\MoveEqLeft \sum_i\bigg|\int_{C'_i } p_{X|Z}(x|z) dx- \int_{C_m} \int_{C'_i} p_{X|Z}(x|z) dx d\tilde P(z)\bigg| \\
&\leq \int_{C_m}\sum_i\bigg|\int_{C'_i } p_{X|Z}(x|z) dx-  \int_{C'_i} p_{X|Z}(x|z') dx \bigg|d\tilde P(z') \\
&\leq \int_{C_m}\sum_i\int_{C'_i } |p_{X|Z}(x|z) - p_{X|Z}(x|z')| dx d\tilde P(z') \\
&\leq \int L|z- z'| d \tilde P(z') \leq L\diam(C_m) = \frac{L}{d}.
\end{align*}
Using the same strategy for the second summation completes the proof, i.e. we establish
\begin{align*}
\sum_{i,j} (q_{ij}(m) - q_{i\cdot}(m) q_{\cdot j}(m))^2 \leq \bigg(\frac{L}{d}\bigg)^4.
\end{align*}
\end{proof}

\begin{lemma}
\label{lem:bb}
Suppose that $p_{X,Y,Z} \in \cP_{0, [0,1]^3, \chi^2}(L)$. Then
\begin{align*}
\EE_{A_m} \|q_{A_m} - q_{\Pi,A_m}\|_2^2 \leq \frac{O(1)}{\omega_m^2} \bigg(\frac{L}{d}\bigg)^2
\end{align*}
\end{lemma}
\begin{proof}
Using inequality \eqref{intermediate:bound:null:hypothesis:lemma} it suffices to directly control the quantity:
$$
\sum_{i,j} \frac{(q_{ij}(m) - q_{i\cdot}(m) q_{\cdot j}(m))^2}{q_{i\cdot}(m) q_{\cdot j}(m)}.
$$
By definition we have
\begin{align*}
\MoveEqLeft \sum_{i,j} \frac{(q_{ij}(m) - q_{i\cdot}(m) q_{\cdot j}(m))^2}{q_{i\cdot}(m) q_{\cdot j}(m)} \\
& = \sum_{i,j} \frac{\bigg(\int_{C_m} \int_{C'_i \times C'_j} p_{X,Y|Z}(x,y|z) dx dy d\tilde P(z) - \int_{C_m} \int_{C'_i} p_{X|Z}(x|z) dx d\tilde P(z)\int_{C_m} \int_{C'_j} p_{Y|Z}(y|z) dy d\tilde P(z)\bigg)^2}{\int_{C_m} \int_{C'_i} p_{X|Z}(x|z) d x d\tilde P(z) \int_{C_m} \int_{C'_j} p_{Y|Z}(y|z) d y d\tilde P(z)}
\end{align*}
Using the fact that $p_{X,Y|Z}(x,y|z) = p_{X,Z}(x|z) p_{Y|Z}(y|z)$ and Cauchy-Schwarz we obtain
\begin{align*}
\MoveEqLeft \sum_{i,j} \frac{(q_{ij}(m) - q_{i\cdot}(m) q_{\cdot j}(m))^2}{q_{i\cdot}(m) q_{\cdot j}(m)} \\
& \leq \bigg(\sum_{i} \frac{\int_{C_m} \bigg(\int_{C'_i} p_{X|Z}(x|z) d x\bigg)^2 d \tilde P(z)}{\int_{C_m} \int_{C'_i} p_{X|Z}(x|z) d x d\tilde P(z)} - 1\bigg)\bigg(\sum_{j} \frac{\int_{C_m} \bigg(\int_{C'_j} p_{Y|Z}(y|z) d y\bigg)^2 d \tilde P(z)}{\int_{C_m} \int_{C'_j} p_{Y|Z}(y|z) d y d\tilde P(z)} - 1\bigg).
\end{align*}
We will handle each of these terms individually. We first note that the function $\int_{C_i} p(x|z) d x$ is continuous in $z$. To see this recall that (by Cauchy-Schwarz)
$$
\bigg(\sum_{i}\int_{C'_i} |p_{X|Z}(x|z) - p_{X|Z}(x|z')| dx\bigg)^2 \leq \int_{[0,1]} \frac{p^2_{X|Z}(x|z')}{p_{X|Z}(x|z)} dx - 1 \leq L |z - z'|. 
$$
It follows that $\bigg|\int_{C'_i} p_{X|Z}(x|z) dx -  \int_{C'_i} p_{X|Z}(x|z') dx \bigg| \leq \sqrt{L |z - z'|}$. It therefore follows by the mean value theorem (as in the discrete case) that for some $z' \in C_m$:
\begin{align*}
\MoveEqLeft \sum_{i} \frac{\int_{C_m} \bigg(\int_{C'_i} p_{X|Z}(x|z) d x\bigg)^2 d \tilde P(z)}{\int_{C_m} \int_{C'_i} p_{X|Z}(x|z) d x d\tilde P(z)} - 1 = \sum_{i} \frac{ \bigg(\int_{C'_i} p_{X|Z}(x|z') d x\bigg)^2}{\int_{C_m} \int_{C'_i} p_{X|Z}(x|z) d x d\tilde P(z)} - 1 \\
& \leq \int_{C_m} \bigg(\sum_{i} \frac{ \bigg(\int_{C'_i} p_{X|Z}(x|z') d x\bigg)^2}{ \int_{C'_i} p_{X|Z}(x|z) d x } - 1 \bigg)d\tilde P(z),
\end{align*}
where we used Jensen's inequality and the fact that $x\mapsto 1/x$ is convex in the last inequality. Now by Cauchy-Schwarz we have
$$
\frac{ \bigg(\int_{C'_i} p_{X|Z}(x|z') d x\bigg)^2}{ \int_{C'_i} p_{X|Z}(x|z) d x }  \leq \int_{C'_i} \frac{p^2_{X|Z}(x|z')}{p_{X|Z}(x|z)} dx. 
$$
Hence we conclude
$$
\sum_{i} \frac{\int_{C_m} \bigg(\int_{C'_i} p_{X|Z}(x|z) d x\bigg)^2 d \tilde P(z)}{\int_{C_m} \int_{C'_i} p_{X|Z}(x|z) d x d\tilde P(z)} - 1\leq  \int_{C_m} \int_{[0,1]} \frac{p^2_{X|Z}(x|z')}{p_{X|Z}(x|z)} dx - 1 d \tilde P(z) \leq L\diam(C_m) = \frac{L}{d},
$$
by assumption. Handling the second term in the same way warrants the desired conclusion. 
\end{proof}

Next we need to lower bound \eqref{expression:to:lower:bound}. For this it suffices to lower bound the distance $d_{\TV}(q(m), q_\Pi(m))$, where $q(m)$ is the distribution with ``density'' $q_{ij}(m)$ while $q_{\Pi}(m)$ is the distribution with density $ q_{i\cdot}(m)q_{\cdot j}(m)$. We have
\begin{lemma} 
\label{lem:cc}
There exist constants $b_1, b_2 > 0$ depending only on $(s,L)$, such that the following bound holds
\begin{align*}
\|q(m) - q_\Pi(m)\|_1 \geq b_1 \sup_{z \in C_m} \int_{[0,1]^2} |p_{X,Y|Z}(x,y|z) - p_{X|Z}(x  |z) p_{Y|Z}( y | z)| dx dy - \frac{b_2}{d}.
\end{align*}

\end{lemma}

\begin{proof}

Take a point $z^* \in C_m$ such that it maximizes the function $z \mapsto \int_{[0,1]^2} |p_{X,Y|Z}(x,y|z) - p_{X|Z}(x  |z) p_{Y|Z}( y | z)| dx dy$. If such a point does not exist, take a sequence of points that converge to the supremum. By the triangle inequality we have 
\begin{align*}
\MoveEqLeft \sum_{i,j} |q_{ij}(m) - q_{i\cdot}(m)q_{\cdot j}(m)| \\
& \geq \sum_{i,j} \bigg|\int_{C'_i \times C'_j} p_{X,Y|Z}(x,y|z^*) dx dy - \int_{C'_i \times [0,1]} p_{X,Y|Z}(x,y|z^*) dx dy \int_{[0,1] \times C'_j} p_{X,Y|Z}(x,y|z^*) dx dy\bigg|\\
& - \sum_{i,j} \bigg| \int_{C_m} \int_{C'_i \times C'_j} p_{X,Y|Z}(x,y|z) - p_{X,Y|Z}(x,y|z^*) dx dy d \tilde P(z) \bigg|\\
& - \sum_{i,j} \bigg|q_{i\cdot}(m) \bigg(\int_{C_m} \int_{[0,1] \times C'_j} p_{X,Y|Z}(x,y|z) - p_{X,Y|Z}(x,y|z^*) dx dy d\tilde P(z)\bigg)\bigg|\\
& - \sum_{i,j} \bigg|\int_{[0,1] \times C'_j} p_{X,Y|Z}(x,y|z^*) dx dy \bigg(\int_{C_m} \int_{C'_i \times [0,1]} p_{X,Y|Z}(x,y|z) - p_{X,Y|Z}(x,y|z^*) dx dy d\tilde P(z)\bigg)\bigg|
\end{align*}
We will first handle the last three terms, starting with the first one. Next we obtain
\begin{align*}
\MoveEqLeft \sum_{i,j} \bigg| \int_{C_m} \int_{C'_i \times C'_j} p_{X,Y|Z}(x,y|z) - p_{X,Y|Z}(x,y|z^*) dx dy d \tilde P(z) \bigg|\\
& \leq  \int_{C_m} \sum_{i,j} \int_{C'_i \times C'_j} |p_{X,Y|Z}(x,y|z) - p_{X,Y|Z}(x,y|z^*)| dx dy d \tilde P(z) \\
& \leq \int_{C_m} L| z - z^*| d \tilde P(z) \leq \frac{L}{d}.
\end{align*}
For the second term we have
\begin{align*}
\MoveEqLeft \sum_{i,j} \bigg|q_{i\cdot}(m) \bigg(\int_{C_m} \int_{[0,1] \times C'_j} p_{X,Y|Z}(x,y|z) - p_{X,Y|Z}(x,y|z^*) dx dy d\tilde P(z)\bigg)\bigg|\\
& \leq \int_{C_m} \sum_{i} q_{i\cdot}(m)\sum_{j} \int_{[0,1] \times C'_j} |p_{X,Y|Z}(x,y|z) - p_{X,Y|Z}(x,y|z^*) | dx dy d\tilde P(z)\\
& \leq \int_{C_m} L| z - z^*| d \tilde P(z) \leq \frac{L}{d}.
\end{align*}
The analysis of the third term is the same. We will finally need a lower bound on 
\begin{align}\label{cont:term:in:question}
\sum_{i,j} \bigg|\int_{C'_i \times C'_j} p_{X,Y|Z}(x,y|z^*) dx dy - \int_{C'_i \times [0,1]} p_{X,Y|Z}(x,y|z^*) dx dy \int_{[0,1] \times C'_j} p_{X,Y|Z}(x,y|z^*) dx dy\bigg|
\end{align}
We will compare this term to the following 
\begin{align}\label{term:to:compare:to}
\int_{[0,1]^2} |p_{X,Y|Z}(x,y|z^*) - p_{X|Z}(x |z^*) p_{Y|Z}( y | z^*)| dx dy,
\end{align}
where $p_{X|Z}(x |z^*) = \int_{[0,1]} p_{X,Y|Z}(x,y|z^*) d y$ and similarly for $p_{Y|Z}( y | z^*)$. Note that in terms of this notation we may rewrite the term in question as
\begin{align*}
 \MoveEqLeft \sum_{i,j} \bigg|\int_{C'_i \times C'_j} p_{X,Y|Z}(x,y|z^*) dx dy - \int_{C'_i} p_{X|Z}(x|z^*) dx \int_{C'_j} p_{Y|Z}(y|z^*) dy\bigg| = \\
& \sum_{i,j} \bigg|\int_{C'_i \times C'_j} [p_{X,Y|Z}(x,y|z^*) - p_{X|Z}(x|z^*) p_{Y|Z}(y|z^*)] dx dy \bigg|.
\end{align*}

For what follows let $\mu$ denote the Lebesgue measure on $[0,1]$. 
We now need the following Lemma
\begin{lemma}\label{super:crucial:smoothing:lemma}Suppose that $h \in \cH^{2,s}(L)$. For an integer $\kappa \in \NN$, take a decomposition of $[0,1]^2 = \prod_{i, j \in [\kappa]} B_i \times B_j$ where the length of each $\mu(B_i) = \kappa^{-1}$ and $B_i$ form a decomposition of $[0,1]$. Let $W_{\kappa} h = \sum_{i,j \in [\kappa]} \kappa^2 (\int_{B_i \times B_j} h )\mathbbm{1}_{B_i \times B_j}$. There exist constants $b_1, b_2 > 0$ depending only on $(s, L)$ such that 
\begin{align*}
\|W_{\kappa}(h)\|_1 \geq b_1 \|h\|_1 - b_2 \kappa^{-s}, ~~~~~ \forall h \in \cH^{2,s}(L)
\end{align*}
\end{lemma}

\begin{proof} The proof of this Lemma is extremely similar to the proof of Lemma 3 of \cite{ery2018remember} modulo changes from $L_2$ distance to $L_1$ distance which do not modify the arguments. Hence we omit the details. We also remark that the proof of this Lemma in the one-dimensional case can be found in Proposition 2.16 of \cite{ingster2012nonparametric}. 
\end{proof}

In order to use Lemma \ref{super:crucial:smoothing:lemma}, we note that the difference $p_{X,Y|Z}(x,y|z^*) - p_{X|Z}(x|z^*) p_{Y|Z}(y|z^*)$ belongs to the H\"{o}lder class $\cH^{2,s}(2L)$ by definition.
Using Lemma \ref{super:crucial:smoothing:lemma}, we can write the following bound 

\begin{align*}
\MoveEqLeft \sum_{i,j} \bigg|\int_{C'_i \times C'_j} [p_{X,Y|Z}(x,y|z^*) - p_{X|Z}(x|z^*) p_{Y|Z}(y|z^*)] dx dy \bigg| \geq \\
& b_1 \int_{[0,1]^2} |p_{X,Y|Z}(x,y|z^*) - p_{X|Z}(x |z^*) p_{Y|Z}( y | z^*)| dx dy - b_2 \frac{1}{(d^{1/s})^s}
\end{align*}
Putting everything together we conclude that
\begin{align*}
\sum_{i,j} |q_{ij}(m) - q_{i\cdot}(m)q_{\cdot j}(m))| & \geq b_1 \sup_{z \in C_m} \int_{[0,1]^2} |p_{X,Y|Z}(x,y|z) - p_{X|Z}(x  |z) p_{Y|Z}( y | z)| dx dy - \frac{b_2}{d},
\end{align*}
completing the proof. 
\end{proof}

After having these three results, the remaining details of the proof follow closely that of Theorem \ref{main:theorem:scaling:discrete:XY} so we omit the details. 
\end{proof}

\section{Unbounded $Z$, Discrete $X,Y$ Non-Scaling}


Define the set of distributions $\cE'_{}$ as distributions whose generating mechanism of the triple $(X,Y,Z)$ supported on $\RR^3$ is as follows: first a $Z$ from the distribution $p_Z$ (which is absolutely continuous with respect to the Lebesgue measure) with potentially unbounded support is generated. Next, $X$ and $Y$ are generated from the distribution $p_{X,Y|Z}$ which is supported on $[\ell_1] \times [\ell_2]$ for (almost) all $Z$. Denote by $\cP_{0}' \subset \cE'_{}$ the set of null distributions (i.e. distributions such that $X \independent Y | Z$) and let $\cQ_{0}' = \cE'_{} \setminus \cP_{0}'$.

\begin{definition}[Null Lipschitzness] \leavevmode 
\begin{enumerate}
\item {Null TV Lipschitzness: } Let $\cP_{0, \TV}'(L) \subset \cP_{0}'$
be the collection of distributions $p_{X,Y,Z}$ such that for all $z,z'$ we have:
\begin{align*}
\|p_{X | Z = z} - p_{X | Z = z'}\|_1 \leq L |z - z'| \mbox{ and } \|p_{Y | Z = z} - p_{Y | Z = z'}\|_1 \leq L |z - z'|,
\end{align*}
where $p_{X| Z = z}$ and $p_{Y | Z = z}$ denote the conditional distributions of $X | Z = z$ and $Y | Z = z$ under $p_{X,Y,Z}$  respectively. 
\item {Null $\chi^2$ Lipschitzness: } Let $\cP_{0, \chi^2}'(L) \subset \cP_{0}'$ 
be the collection of distributions $p_{X,Y,Z}$ such that for all $z,z' $ we have:
\begin{align*}
d_{\chi^2}(p_{X | Z = z}, p_{X | Z = z'}) \leq L |z - z'| \mbox{ and } d_{\chi^2}(p_{Y | Z = z}, p_{Y | Z = z'}) \leq L |z - z'|,
\end{align*}
where $p_{X| Z = z}$ and $p_{Y | Z = z}$ denote the conditional distributions of $X | Z = z$ and $Y | Z = z$ under $p_{X,Y,Z}$ respectively. The distance $d_{\chi^2}(p_{X | Z = z}, p_{X | Z = z'})$ is considered $\infty$ if $p_{X | Z = z} \ll p_{X | Z = z'}$ is violated. 
\item {Null H\"{o}lder Lipschitzness: } Let $\cP_{0, \TV^2}'(L) \subset \cP_{0}'$ be the collection of distributions $p_{X,Y,Z}$ such that for all $z,z'$ we have:
\begin{align*}
\|p_{X | Z = z} - p_{X | Z = z'}\|_1 \leq \sqrt{L |z - z'|} \mbox{ and } \|p_{Y | Z = z} - p_{Y | Z = z'}\|_1 \leq \sqrt{L |z - z'|},
\end{align*}
where $p_{X| Z = z}$ and $p_{Y | Z = z}$ denote the conditional distributions of $X | Z = z$ and $Y | Z = z$ under $p_{X,Y,Z}$  respectively. 
\end{enumerate}
\end{definition}
\begin{definition}[Alternative TV Lipschitzness]Let $\cQ_{0, \TV}'(L) \subset \cQ_{0}'$ be the collection of distributions 
$p_{X,Y,Z}$ such that for all $z, z'$ we have
\begin{align*}
\|p_{X,Y | Z= z} - p_{X,Y | Z = z'}\|_1 \leq L |z - z'|, 
\end{align*}
where $p_{X,Y| Z = z}$ denotes the conditional distribution of $X,Y | Z = z$ under $p_{X,Y,Z}$. 
\end{definition}

First it is straightforward to generalize the results of Section \ref{upper:bound:section}, to the case where $Z$ is supported on an interval $[A, B]$ with $A < B$. The proof needs to be modified to respect the fact that the bins have a size of $\frac{B-A}{d}$ instead of size $d$. The final rate that one obtains is $\frac{(B-A)^{1/5}}{n^{2/5}} \vee \frac{(B-A)^{7/15}}{n^{8/15}}$.

\begin{lemma}\label{bounding:Z:lemma} Suppose that $\inf_{p \in \cP_0'} \|q_{X,Y,Z} - p\|_1 = \varepsilon$. Define the distribution $\tilde q_{X,Y,Z} = q_{X,Y|Z} \tilde q_Z$ where $\tilde q_Z$ is the restriction of $q_Z$ on an interval $S_\eta$ with the property that $\PP(Z \in S_\eta) \geq 1 - \eta$. Then  $\inf_{p \in \cP_0'} \|\tilde q_{X,Y,Z} - p\|_1 \geq \frac{1}{6}\frac{\varepsilon - 2\eta}{1 - \eta}$.
\end{lemma}

\begin{proof}[Proof of Lemma \ref{bounding:Z:lemma}] By Lemma \ref{TV:proj:vs:all} (whose proof does not depend on the bounded $Z$ support) we know that 
\begin{align*}
\inf_{p \in \cP_0'} \|q_{X,Y,Z} - p\|_1 \geq \|\tilde q_{X,Y,Z} - \tilde q_{X | Z} \tilde q_{Y|Z} \tilde q_{Z}\|_1/6
\end{align*}
Next
\begin{align*}
\|\tilde q_{X,Y,Z} - \tilde q_{X | Z} \tilde q_{Y|Z} \tilde q_{Z}\|_1 = \int_{S_\eta} \sum_{x,y} |q_{X,Y|Z}(x,y|z) - q_{X|Z}(x|z)q_{Y|Z}(y|z)|  \frac{q_Z(z)}{\PP(Z \in S_\eta)} dz
\end{align*}
On the other hand
\begin{align*}
\varepsilon \leq \| q_{X,Y,Z} -  q_{X | Z}  q_{Y|Z}  q_{Z}\|_1 & = \int_{S_\eta} \sum_{x,y} |q_{X,Y|Z}(x,y|z) - q_{X|Z}(x|z)q_{Y|Z}(y|z)|  q_Z(z)dz\\
& +  \int_{S_\eta^c} \sum_{x,y} |q_{X,Y|Z}(x,y|z) - q_{X|Z}(x|z)q_{Y|Z}(y|z)|  q_Z(z)dz\\
& \leq \int_{S_\eta} \sum_{x,y} |q_{X,Y|Z}(x,y|z) - q_{X|Z}(x|z)q_{Y|Z}(y|z)|  q_Z(z)dz\\
& + 2\PP(Z \in S_\eta^c)\\
\end{align*}

We conclude that 
\begin{align*}
\inf_{p \in \cP_0'} \|q_{X,Y,Z} - p\|_1  \geq \frac{1}{6}\frac{\varepsilon - 2\eta}{1 - \eta},
\end{align*}
as claimed.

\end{proof}

Next suppose we split the sample into two samples of equal size. Assume for simplicity that initially we had $2n$ samples. Take the second data, and construct the shortest interval $\hat S_\eta$ containing $n(1-\eta) + C\sqrt{n\log n}$ of the samples $\{Z_1, \ldots, Z_n\}$, where $C$ is an absolute constant. We have the following result.

\begin{lemma} \label{VC:ineq:lemma} For a sufficiently large $C$ we have
\begin{align*}
\PP(Z \in \hat S_\eta) \geq 1 - \eta,
\end{align*}
with probability at least $1- \frac{1}{n}$. In addition $\mu(\hat S_\eta) \leq \mu(S_{\eta/2})$ with probability at least $1- \frac{1}{n}$, where $\mu$ is the Lebesgue measure, $\eta \geq 4 C \sqrt{\frac{\log n}{n}}$, and $S_{\eta/2}$ is the shortest interval such that $\PP(Z \in S_{\eta/2}) \geq 1 - \eta/2$. 
\end{lemma}

\begin{proof}[Proof of Lemma \ref{VC:ineq:lemma}] By the VC inequality \cite[cf. Theorem 12.5]{devroye2013probabilistic}, and the fact that closed intervals on the real line have VC dimension equal to 2, we have that
\begin{align*}
\PP\bigg(\bigg|\PP(Z \in \hat S_{\eta}) - \bigg((1-\eta) + C\sqrt{\frac{\log n}{n}}\bigg)\bigg| \geq t\bigg) \leq 8 (\frac{ne}{2})^2 \exp(-32n t^2)
\end{align*}

Now set $t = C\sqrt{\frac{\log n}{n}}$. For large enough $C$, the above probability is bounded by $\frac{1}{n}$ which is what we claimed. 

By the VC inequality 

\begin{align*}
\PP\bigg(\bigg|\PP(Z \in S_{\eta/2}) - n^{-1}\sum_{n} \mathbbm{1}(Z_i \in S_{\eta/2}) \bigg)\bigg| \geq t\bigg) \leq 8 (\frac{ne}{2})^2 \exp(-32n t^2)
\end{align*}

Thus with a choice of $t = C \sqrt{\frac{\log n}{n}}$ we obtain 
\begin{align*}
n^{-1}\sum_{n} \mathbbm{1}(Z_i \in S_{\eta/2})  \geq 1 - \eta/2 - C \sqrt{\frac{\log n}{n}} > 1 - \eta + C\sqrt{\frac{\log n}{n}},
\end{align*}
with probability $\frac{1}{n}$, under our assumption on $\eta$. Thus if $\mu(S_{\eta/2}) \leq \mu(\hat S_{\eta})$ that would be a contradiction. 

\end{proof}

Next we describe the testing procedure. Suppose we have $2n$ observations. Split the data equally, and estimate $\hat S_{\eta}$, where $1 \geq \eta \geq C\sqrt{\frac{\log n}{n}}$ on the second half as discussed above. 

Draw $N \sim Poi(\frac{n}{2})$. If $N > n$ accept the null hypothesis. If $N \leq n$ take arbitrary $N$ out of the $n$ samples from the first half of the data and work with them. The next step is to discretize the variable $Z$ into $d$ bins of equal size. Denote those bins with $\{C_1, \ldots, C_d\}$, so that $\cup_{i \in [d]} C_i = \hat S_{\eta}$, and each $C_i$ is an interval of length $\frac{\mu(\hat S_{\eta})}{d}$. Next construct the datasets $\cD_{m} = \{(X_i, Y_i) : Z_i \in C_m, i \in [N]\}$. Let $\sigma_m = |\cD_m|$ be the sample size in each set $\cD_m$, so that $\sum_{m \in [d]} \sigma_m = N$. For bins $\cD_m$ with at least $\sigma_m \geq 4$ observations, let for brevity $U_m = U(\cD_m)$. Each $U_m$ can be thought of as a local test of independence within the bin $C_m$ --- if the value of $U_m$ is close to $0$ then intuitively independence holds within that bin, while if the value of $U_m$ is large, independence is potentially violated within that bin. In order to combine these different statistics we follow \citet{canonne2018testing} and consider the following test statistic
\begin{align*}
T = \sum_{m \in [d]} \mathbbm{1}(\sigma_m \geq 4) \sigma_m U_m. 
\end{align*}
We will prove that under the null hypothesis the value of $T$ is likely to be below a threshold $\tau$ (to be specified), while under the alternative hypothesis $T$ will likely exceed the value $\tau$. Define the test 
$$\psi_\tau(\cD_N) = \mathbbm{1}(T \geq \tau),$$ 
where $\cD_N = \{(X_1,Y_1,Z_1), \ldots, (X_N,Y_N,Z_N)\}$. Recall the definitions of the null Lipschitzness classes defined above
$\cP_{0,\TV}'(L), \cP_{0,\chi^2}'(L), \cP_{0,\TV^2}'(L)$ and the alternative Lipschitzness classes $\cQ_{0,\TV}'(L)$. We have
\begin{theorem}[Finite Discrete $X$, $Y$ Upper Bound] \label{main:theorem:finite:discrete:XY:unbounded} Set $d = \lceil \mu(\hat S_{\eta})^{4/5} n^{2/5} \rceil \wedge \lceil \mu(\hat S_{\eta})^{8/15} n^{8/15} \rceil$ and let $\tau = \zeta \sqrt{d}$ for a sufficiently large absolute constant $\zeta$ (depending on $L$). 
Finally, suppose that $\varepsilon \geq c \max\bigg(\frac{\mu(S_{\eta/2})^{1/5}}{n^{2/5}}, \frac{\mu(S_{\eta/2})^{7/15}}{n^{8/15}}, \eta\bigg)$, for a sufficiently large constant $c$ (depending on $\zeta$, $L$, $\ell_1,\ell_2$).
Then we have that 
\begin{align*}
\sup_{p \in \cP_{0,\TV^2}'(L) \cup \cP_{0,\TV}'(L) \cup \cP_{0,\chi^2}'(L)} \EE_p[\psi_\tau(\cD_N)] & \leq \frac{1}{10},\\ 
\sup_{p \in \{p \in \cQ_{0,\TV}'(L): \inf_{q \in \cP'_{0}} \|p - q\|_1 \geq \varepsilon\}} \EE_p[ 1- \psi_\tau(\cD_N)] & \leq \frac{1}{10} + \exp(-n/8).
\end{align*}
\end{theorem}

\begin{proof} The proof of this result follows directly from the previous discussion and the proof of Theorem \ref{main:theorem:finite:discrete:XY}. We omit the details. 
\end{proof}

\begin{remark}
In addition we point out that for sub-Gaussian distributions, the above strategy can lead to rates which coincide with the rates described in the bounded support case up to logarithmic factors. To see this it suffices to select $S_{\eta} = [Z_{(1)}, Z_{[n]}]$ where $Z_{(i)}$ denote the order statistics on a second dataset; by sub-Gaussianity one has with high probability that $|Z_{(n)}|, |Z_{(1)}| \lesssim \sqrt{\log n}$. Furthermore $\EE \PP(Z \geq Z_{(n)}) = \EE F_Z(Z)^n = \int_{0}^1 u^n du = \frac{1}{n + 1}$. Thus by Markov's inequality $\PP(\PP(Z \geq Z_{(n)}) \geq a_n) \leq \frac{1}{(n + 1)a_n}$. Thus if $a_nn \rightarrow \infty$ with high probability $\PP(Z \geq Z_{(n)}) \leq a_n$. By symmetry the same argument is valid for $\PP(Z \leq Z_{(1)}) \leq a_n$. Thus it suffices that the critical radius scales as
\begin{align*}
\varepsilon \gtrsim \frac{(\sqrt{\log n})^{1/5}}{n^{2/5}}. 
\end{align*}
\end{remark}


%
%
%
%

\section{Proofs from Section \ref{examples:section}}

\begin{proof}[Proof of Lemma \ref{continuous:lemma:lipschitz:TV:smooth}] Note that condition \eqref{xy:Lipschitzness:condition} ensures the second part of the definition of the set $\cQ_{0,[0,1]^3,\TV}$ hence we only need to prove that $\|p_{X,Y|Z = z} - p_{X,Y|Z = z'}\|_1$ is bounded by $(e^L-1)|z - z'|$.
\begin{align*}
\MoveEqLeft \int \int |p_{X,Y|Z}(x,y|z) - p_{X,Y|Z}(x,y|z')|  dx dy \\
& = \int \int  \bigg(\frac{\max(p_{X,Y|Z}(x,y|z), p_{X,Y|Z}(x,y|z'))}{\min(p_{X,Y|Z}(x,y|z), p_{X,Y|Z}(x,y|z'))} - 1\bigg) \min(p_{X,Y|Z}(x,y|z), p_{X,Y|Z}(x,y|z'))dx dy\\
& \leq  \int \int  \bigg(\frac{\max(p_{X,Y|Z}(x,y|z), p_{X,Y|Z}(x,y|z'))}{\min(p_{X,Y|Z}(x,y|z), p_{X,Y|Z}(x,y|z'))} - 1\bigg) p_{X,Y|Z}(x,y|z) dx dy.
\end{align*}
By the Lipschitz property of $\log p_{X,Y|Z}(x,y|z)$ we have
\begin{align*}
\frac{\max(p_{X,Y|Z}(x,y|z), p_{X,Y|Z}(x,y|z'))}{\min(p_{X,Y|Z}(x,y|z), p_{X,Y|Z}(x,y|z'))} - 1 \leq \exp(L|z - z'|) - 1.
\end{align*}
From here the proof can continue as in the proof of Lemma \ref{log:Lipschitz:functions:lemma}. 

The fact that $p_{X,Y|Z}(x,y|z)$ is H\"{o}lder in $x,y$ for every $z$ with $s = 1$ is clear since $C |x - x'| + |y - y'| \leq \sqrt{2}C ((x - x')^2 + (y - y')^2)^{1/2}$. 

This completes the proof.
\end{proof}

\section{Proofs from Section \ref{numerical:experiments:section}}
\label{app:exp}

In this section we prove that the the generation mechanism of Section \ref{numerical:experiments:section} is indeed ``TV smooth''. We first start by the example $X = \frac{U_1 + Z}{2}$ and $Y = \frac{U_2 + Z}{2}$. We want to show that $\|p_{X| Z = z}- p_{X|Z = z'}\|_1 \leq L|z - z'|$ for an appropriate $L$ (and similarly for $Y$). In this example it is simple to verify that 
\begin{align*}
\|p_{X| Z = z} - p_{X|Z = z'}\|_1 = \int_0^{1} 2|\mathbbm{1}_{[\frac{z}{2}, \frac{z + 1}{2}]}(x) - \mathbbm{1}_{[\frac{z'}{2}, \frac{z' + 1}{2}]}(x)| dx = 2 |z - z'|.
\end{align*}
This implies that the so generated data belongs to the set $\cP_{0,[0,1]^3, \TV}(2)$. 

Next we consider $X = \frac{U_1 + U + Z}{3}$ and $Y = \frac{U_2 + U + Z}{3}$. We will argue that $\|p_{X,Y|Z = z} - p_{X,Y|Z = z'}\|_1 \leq L|z - z'|$ for an appropriate $L$. Let $\Sigma$ be the Borel $\sigma$-field on $[0,1]^2$. We note the following bounds
\begin{align*}
\MoveEqLeft d_{\TV}(p_{X,Y|Z = z}, p_{X,Y|Z = z'})\\
& = \sup_{A \in \Sigma} \bigg|\PP\bigg(\bigg(\frac{U_1 + U + z}{3}, \frac{U_2 + U + z}{3}\bigg) \in A\bigg) - \PP\bigg(\bigg(\frac{U_1 + U + z'}{3}, \frac{U_2 + U + z'}{3}\bigg) \in A\bigg)\bigg|\\
& =  \sup_{A \in \Sigma} \bigg|\int_{0}^{1}\PP\bigg(\bigg(\frac{U_1 + U + z}{3}, \frac{U_2 + U + z}{3}\bigg) \in A \bigg | U = u \bigg) du \\
& ~~~~~~~ - \int_{0}^{1}\PP\bigg(\bigg(\frac{U_1 + U + z'}{3}, \frac{U_2 + U + z'}{3}\bigg) \in A \bigg | U = u\bigg) du\bigg|\\
&\leq \int_{0}^{1} \sup_{A \in \Sigma}  \bigg|\PP\bigg(\bigg(\frac{U_1 + u + z}{3}, \frac{U_2 + u + z}{3}\bigg) \in A   \bigg)  - \PP\bigg(\bigg(\frac{U_1 + u + z'}{3}, \frac{U_2 + u + z'}{3}\bigg) \in A \bigg) \bigg|du\\
\end{align*}

The expression in the above integral is nothing but the total variation between the law of $(\frac{U_1 + u + z}{3}, \frac{U_2 + u + z}{3})$ and $(\frac{U_1 + u + z'}{3}, \frac{U_2 + u + z'}{3})$ which is $\frac{1}{2}$ of the $L_1$ distance. Since $U_1 $ and $U_2$ are independent it is simple to see that for a fixed $u$ the above is equivalent to

\begin{align}\label{TV:calculation:equation}
\MoveEqLeft \sup_{A \in \Sigma}  \bigg|\PP\bigg(\bigg(\frac{U_1 + u + z}{3}, \frac{U_2 + u + z}{3}\bigg) \in A   \bigg)  - \PP\bigg(\bigg(\frac{U_1 + u + z'}{3}, \frac{U_2 + u + z'}{3}\bigg) \in A \bigg) \bigg| \nonumber \\
& = \frac{9}{2} \int_0^1 \int_0^1\bigg| \mathbbm{1}_{[\frac{u + z}{3},\frac{1 + u + z}{3}]}(u_1)\mathbbm{1}_{[\frac{u + z}{3},\frac{1 + u + z}{3}]}(u_2) -  \mathbbm{1}_{[\frac{u + z'}{3},\frac{1 + u + z'}{3}]}(u_1)\mathbbm{1}_{[\frac{u + z'}{3},\frac{1 + u + z'}{3}]}(u_2) \bigg| du_1 du_2.
 \end{align}

\begin{figure}[H]
\centering
\includegraphics[scale=.2]{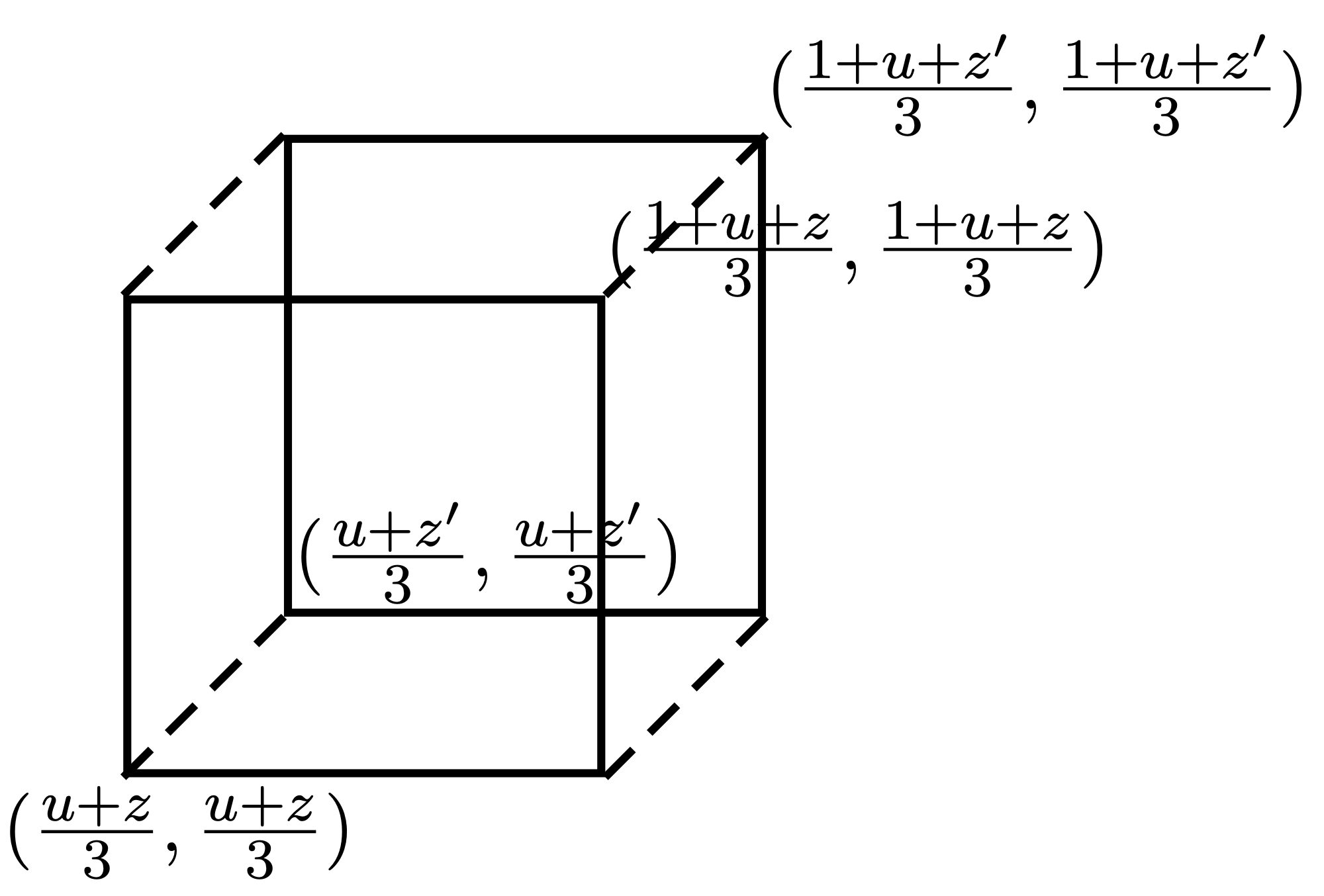}
\caption{The two rectangles given in equation \eqref{TV:calculation:equation} are shown above. The area of the symmetric difference is bounded by 4 times the area of a parallelogram. }\label{simple:figure:area:difference}
\end{figure}

Using Figure \ref{simple:figure:area:difference} we can bound the integral in \eqref{TV:calculation:equation} by $4$ times the area of a parallelogram, which is further bounded by $\sqrt{2}/9 |z - z'|$. We conclude that 
\begin{align*}
\frac{1}{2} \|p_{X,Y|Z = z} - p_{X,Y|Z = z'}\|_1 = d_{\TV}(p_{X,Y|Z = z}, p_{X,Y|Z = z'}) \leq 2 \sqrt{2} |z - z'|,
\end{align*}
which is what we wanted to show. 

The fact that the density $p_{X,Y|Z}(x,y|z)$ is Lipschitz in $x,y$ follows by a straightforward but tedious calculation which is omitted. The idea is to calculate the cdf by conditioning on $U$ and then take derivatives with respect to $x$ and $y$. 

\section{Extension to a Multivariate $Z$} \label{multivariate:Z:extension}

In this section, we give lower bounds for the multivariate $Z$ case. We also provide matching upper bounds for dimensions $d_Z = 1, 2$. Giving matching upper bounds for the case when $d_Z > 2$ is an open problem. 

\subsection{Discrete $X,Y$ Case with Fixed Many Categories}

Let $Z \in [0,1]^{d_Z}$. Define $\cP'_{0, [0,1]^{d_Z}}$ and $\cQ'_{0, [0,1]^{d_Z}}$ analogously to the sets $\cP'_{0, [0,1]}$ and $\cQ'_{0, [0,1]}$. Analogously to the Null and Alternative TV Lipschitzness define the following classes. 

\begin{definition}[Null Lipschitzness] \leavevmode 
\begin{enumerate}
\item {Null TV Lipschitzness: } Let $\cP_{0,[0,1]^{d_Z}, \TV}'(L) \subset \cP_{0, [0,1]^{d_Z}}'$ be the collection of distributions $p_{X,Y,Z}$ such that for all $z,z' \in [0,1]^{d_Z}$ we have:
\begin{align*}
\|p_{X | Z = z} - p_{X | Z = z'}\|_1 \leq L \|z - z'\|_2 \mbox{ and } \|p_{Y | Z = z} - p_{Y | Z = z'}\|_1 \leq L \|z - z'\|_2,
\end{align*}
where $p_{X| Z = z}$ and $p_{Y | Z = z}$ denote the conditional distributions of $X | Z = z$ and $Y | Z = z$ under $p_{X,Y,Z}$  respectively. 
\item {Null $\chi^2$ Lipschitzness: } Let $\cP_{0,[0,1]^{d_Z}, \chi^2}'(L) \subset \cP_{0, [0,1]^{d_Z}}'$ be the collection of distributions $p_{X,Y,Z}$ such that for all $z,z' \in [0,1]^{d_Z}$ we have:
\begin{align*}
d_{\chi^2}(p_{X | Z = z}, p_{X | Z = z'}) \leq L \|z - z'\|_2 \mbox{ and } d_{\chi^2}(p_{Y | Z = z}, p_{Y | Z = z'}) \leq L \|z - z'\|_2,
\end{align*}
where $p_{X| Z = z}$ and $p_{Y | Z = z}$ denote the conditional distributions of $X | Z = z$ and $Y | Z = z$ under $p_{X,Y,Z}$ respectively. The distance $d_{\chi^2}(p_{X | Z = z}, p_{X | Z = z'})$ is considered $\infty$ if $p_{X | Z = z} \ll p_{X | Z = z'}$ is violated. 
\item {Null H\"{o}lder Lipschitzness: } Let $\cP_{0,[0,1]^{d_Z}, \TV^2}'(L) \subset \cP_{0, [0,1]^{d_Z}}'$ be the collection of distributions $p_{X,Y,Z}$ such that for all $z,z' \in [0,1]^{d_Z}$ we have:
\begin{align*}
\|p_{X | Z = z} - p_{X | Z = z'}\|_1 \leq \sqrt{L \|z - z'\|_2} \mbox{ and } \|p_{Y | Z = z} - p_{Y | Z = z'}\|_1 \leq \sqrt{L \|z - z'\|_2},
\end{align*}
where $p_{X| Z = z}$ and $p_{Y | Z = z}$ denote the conditional distributions of $X | Z = z$ and $Y | Z = z$ under $p_{X,Y,Z}$  respectively. 
\end{enumerate}
\label{def:nullsmoothness:extension}
\end{definition}
Under the alternative we consider slightly different classes in the discrete and continuous cases. Formally, we define the following class for the discrete $X$ and $Y$ setting:
\begin{definition}[Alternative TV Lipschitzness]\label{alternative:TV:smoothness:def:extension} Let $\cQ_{0,[0,1]^{d_Z}, \TV}'(L) \subset \cQ_{0,[0,1]^{d_Z}}'$ be the collection of distributions 
$p_{X,Y,Z}$ such that for all $z, z' \in [0,1]^{d_Z}$ we have
\begin{align*}
\|p_{X,Y | Z= z} - p_{X,Y | Z = z'}\|_1 \leq L \|z - z'\|_2, 
\end{align*}
where $p_{X,Y| Z = z}$ denotes the conditional distribution of $X,Y | Z = z$ under $p_{X,Y,Z}$. 
\end{definition}

\begin{theorem}[Critical Radius Lower Bound]\label{first:lower:bound:app} Let $\overline \cH_0 = \cP_{0,[0,1]^{d_Z}}'$. Suppose that $\cH_0$ is either of $\cP_{0,[0,1]^{d_Z}, \TV}'(L) $, $\cP_{0,[0,1]^{d_Z}, \TV^2}'(L) $ or $\cP_{0,[0,1]^{d_Z}, \chi^2}'(L)$, while $\cH_1 = \cQ_{0,[0,1]^{d_Z}, \TV}'(L)$ for some fixed $L \in \RR^+$. Then for some absolute constant $c_0 > 0$ the critical radius defined in \eqref{critical:radius} is bounded as
\begin{align*}
\varepsilon_n(\cH_0,\overline \cH_0, \cH_1) \geq c_0 \bigg(\frac{(\ell_1 \ell_2)^{1/(d_Z + 4)}}{n^{2/(d_Z + 4)}} \wedge 1\bigg).
\end{align*}
\end{theorem}

\begin{proof} The proof is identical with the one of Theorem \ref{first:lower:bound}. For completeness we provide full details below. 
To derive a lower bound we will first show how to obtain multiple distributions which are far from independent by perturbing the uniform discrete distribution. Suppose for simplicity that $\ell_1 = 2 \ell_1'$ and $\ell_2 = 2\ell_2'$ for some integers $\ell_1'$ and $\ell_2'$ (although this simplifies our calculation we will remark how to fix the calculation for the odd case as well).  
 
We first construct a single null hypothesis distribution. Suppose that $Z \sim U[0,1]$. Let the basic null distribution be given by the density $p_{X,Y | Z} (x,y|z) = \frac{1}{\ell_1 \ell_2}$ for each $x, y \in [\ell_1] \times [\ell_2]$ and $z \in [0,1]^{d_Z}$. Clearly, this distribution belongs to all three sets of null distributions $\cP_{0,[0,1]^{d_Z}, \TV}'(L) $, $\cP_{0,[0,1]^{d_Z}, \TV^2}'(L)$ and $\cP_{0,[0,1]^{d_Z}, \chi^2}'(L)$.  

Next we will perturb this null distribution $p$ to obtain alternative distributions. Let $\Delta = (\delta_{xy})_{x \in [\ell_1'], y \in [\ell_2']}$ be a matrix of $\pm 1$ numbers $\delta_{xy}$. We create the $\ell_1 \times \ell_2$ matrix $\tilde \Delta$ so that 
\begin{align*}
\tilde \delta_{xy}  & = \delta_{xy} \mbox { for } x,y \in [\ell_1']\times[\ell_2'], \\
\tilde \delta_{xy} & = -\delta_{(x -\ell_1')y} \mbox { for } x > \ell_1', y \in [\ell_2'], \\
\tilde \delta_{xy} & = -\delta_{x(y -\ell_2')} \mbox { for } x \in [\ell_1'], y  > \ell_2', \\
\tilde \delta_{xy} & = \delta_{(x-\ell_1')(y -\ell_2')} \mbox { for } x  > \ell_1', y  > \ell_2'.
\end{align*}

\begin{figure}
\centering
\includegraphics[scale=.2]{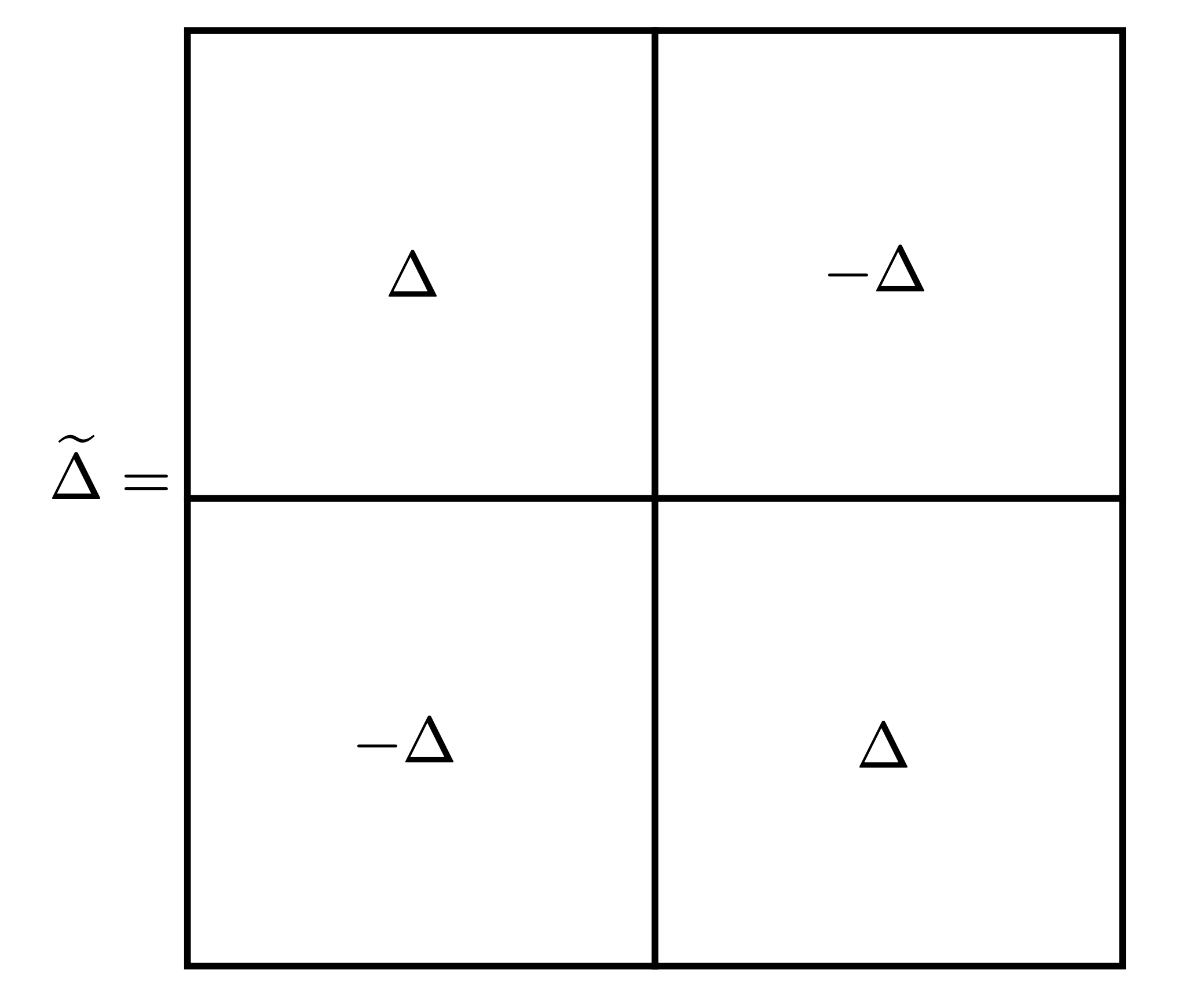}
\caption{The construction of the matrix $\tilde\Delta$ from $\Delta$.}\label{tilde:Delta:figure:appendix}
\end{figure}

It is simple to check that all row sums and column sums of the matrix $\tilde \Delta$ are 0 (see also Figure \ref{tilde:Delta:figure:appendix}). We perturb the null distribution $p$ using the following procedure: for $x,y \in [\ell_1] \times [\ell_2]$ we take $q_{X,Y|Z}(x,y|z) = \frac{1}{\ell_1 \ell_2} + \tilde \delta_{xy} \eta_\nu(z)$, where 
$$
\eta_{\nu}(z) = \rho \sum_{j \in [d]} \nu_{j_1,\ldots j_{d_Z}}  \prod_{k = 1}^{d_Z} h_{j_k, d}(z_k),
$$
where $\rho > 0$ is a constant, $d \in \NN$, $\nu_{j_1,\ldots j_{d_Z}} \in \{-1,+1\}$, and $h_{j, d}(z) = \sqrt{d}h(d z - j + 1)$ for $z \in [(j-1)/d, j/d]$, and $h$ is an infinitely differentiable function supported on $[0,1]$ 
such that $\int h(z) dz = 0$ and $\int h^2(z) dz = 1$. Since the row sums and column sums of $\tilde \Delta$ are $0$, it is simple to verify that the marginals of the distribution remain unchanged under this perturbation, i.e. $q_{X | Z}(x|z) = \sum_{y} q_{X,Y| Z}(x,y|z) = p_{X|Z}(x|z) = \frac{1}{\ell_1}$ and similarly $q_{Y | Z}(y|z) =\sum_{x} q_{X,Y| Z}(x,y|z) = p_{Y|Z}(y|z) = \frac{1}{\ell_2}$. 
We note that in the case that one of $\ell_1$ or $\ell_2$ or both is odd, the fix is to add one row and/or column to the matrix $\Delta$ to be fixed, and reason as in the even case. 

When perturbing, in order to ensure that we create valid probability distributions, we need to satisfy the conditions that 
$$
\frac{1}{\ell_1 \ell_2} - \rho (\sqrt{d})^{d_Z} \|h\|^{d_Z}_{\infty} \geq 0,
$$
and 
$$
\frac{1}{\ell_1 \ell_2} + \rho (\sqrt{d})^{d_Z} \|h\|^{d_Z}_{\infty} \leq 1.
$$
We will ensure this by our choices of $\rho$ and $d$. Next, we need to verify that $q_{X,Y,Z} \in \cQ_{0,[0,1]^{d_Z}, \TV}'(L)$. We start by showing that $\|q_{X,Y |Z = z} - q_{X,Y|Z = z'}\|_1 \leq L \|z - z'\|_2$. We have
\begin{align*}
\|q_{X,Y | Z= z} - q_{X,Y| Z= z'}\|_1 & = \ell_1\ell_2 |\eta_{\nu}(z) - \eta_\nu(z')|. 
\end{align*}
Using telescoping arguments we can bound
\begin{align*}
\|q_{X,Y |Z= z} - q_{X,Y | Z=z'}\|_1  \lesssim \ell_1\ell_2 d_Z \rho d (\sqrt{d})^{d_Z} \|h'\|_{\infty}\|h\|^{d_Z-1}_{\infty} \leq L.
\end{align*}

We let $Z \sim U([0,1])$. Next we will show that the constructed distributions $q_{X,Y,Z}$ are $\varepsilon$ far from being independent, that is we will show that 
\begin{align*}
\inf_{p \in \cP_{0,[0,1]}'} \|q_{X,Y,Z} - p\|_1 \geq \varepsilon,
\end{align*}
for some $\varepsilon > 0$. 


Next, suppose that $h$ satisfies $\int |h(z)|dz = c$ for some $0 < c < 1$. We then have that the $L_1$ distance between $q_{X,Y,Z}$ and $q_{X|Z}q_{Y |Z}q_Z$ satisfies
$$
\varepsilon := \EE_Z \ell_1\ell_2 |\eta_\nu(Z)| = \ell_1\ell_2 \rho \sum_{\bj \in [d]^{d_Z}} \int  \prod_{k = 1}^{d_Z} |h_{j_k, d}(z_k)| d z_k, = \ell_1 \ell_2 \rho (\sqrt{d})^{d_Z}c^{d_Z},
$$
where in the above we used that $h_{j,d}$ have disjoint support. By Lemma \ref{TV:proj:vs:all} this shows that $q_{X,Y,Z}$ is at least $\varepsilon/6$ from any conditionally independent distribution in $L_1$ distance.

Next, we put uniform priors over $\nu$ and $\Delta$, i.e., the random variables $(\nu_i)_{i \in [d]}$ and $(\delta_{xy})_{x,y \in [\ell_1]\times[\ell_2]}$ are taken as i.i.d. Rademachers. The likelihood ratio is 
\begin{align*}
W = \EE_{\nu, \Delta} \prod_{i = 1}^n (1 + \tilde \delta_{X_i, Y_i} \bar \eta_\nu(Z_i))
\end{align*}
where $\bar \eta_\nu(z) = \ell_1 \ell_2 \eta_\nu(z)$ and the expectation is taken over all Rademacher sequences $\nu$ and $\Delta$.  By a standard argument \cite{ery2018remember, balakrishnan2017hypothesis,ingster2003nonparametric}  the risk of the likelihood ratio (which is the optimal test by Neyman-Pearson's Lemma) is bounded from below by $1 - \frac{1}{2}\sqrt{\Var_0 W}$. Hence it suffices to study $\EE_0(W^2) - 1$ (here $\EE_0$ is the expectation under the null hypothesis). We have
\begin{align*}
\EE_0 W^2 & = \EE_{\nu,\nu',\Delta,\Delta'} \prod_{i = 1}^n \EE_0 (1 + \tilde \delta_{X_i, Y_i}\bar \eta_\nu(Z_i))(1 + \tilde \delta'_{X_i, Y_i}\bar \eta_{\nu'}(Z_i))\\
& =\EE_{\nu,\nu',\Delta,\Delta'}\prod_{i = 1}^n \EE_{Z_i} \sum_{(x,y) \in [\ell_1]\times [\ell_2]}  \frac{(1 + \tilde \delta_{xy}\bar \eta_\nu(Z_i))(1 + \tilde \delta'_{xy}\bar \eta_{\nu'}(Z_i))}{\ell_1\ell_2}\\
& =  \EE_{\nu,\nu',\Delta,\Delta'}\prod_{i = 1}^n  \bigg(  1 + \frac{\EE_{Z_i} \bar \eta_\nu(Z_i)\bar \eta_{\nu'}(Z_i)}{\ell_1 \ell_2}\sum_{(x,y) \in [\ell_1]\times [\ell_2]} \tilde \delta_{xy}\tilde \delta'_{xy} \bigg)\\
&= \EE_{\nu,\nu',\Delta,\Delta'}\prod_{i = 1}^n (1 + 4\ell_1 \ell_2 \rho^2\langle\nu, \nu'\rangle \langle\Delta, \Delta' \rangle)\\
& \leq\EE_{\nu,\nu',\Delta,\Delta'} \exp(4n \ell_1 \ell_2 \rho^2\langle\nu, \nu'\rangle \langle\Delta, \Delta' \rangle)
\end{align*}
In the above $\langle \Delta, \Delta'\rangle = \Tr(\Delta\T \Delta')$ is the standard matrix dot product, while $\langle \nu, \nu' \rangle$ is the standard vector dot product and $\EE_{\nu,\nu',\Delta,\Delta'}$ is the expectation with respect to independent Rademacher draws of $\nu, \nu', \Delta, \Delta'$. Thus,
\begin{align*}
\EE W^2 & \leq \EE_{\Delta, \Delta'} \EE_{\nu, \nu'} [\exp(4n \ell_1 \ell_2 \rho^2\langle\nu, \nu'\rangle \langle\Delta, \Delta' \rangle)] = \EE_{\Delta, \Delta'} \cosh(4n \ell_1 \ell_2 \rho^2 \langle\Delta, \Delta' \rangle)^d \\
& \leq \EE_{\Delta, \Delta'} \exp((4n \ell_1 \ell_2 \rho^2 \langle\Delta, \Delta' \rangle)^2d^{d_Z}/2), 
\end{align*}
where we used the inequality $\cosh(x) \leq \exp(x^2/2)$, which can be verified by a Taylor expansion. Next, since when we condition on one value of $\Delta'$ all values of $\langle\Delta, \Delta'\rangle$ happen with the same probability as if we conditioned on any other value of $\Delta'$ we have the identity
\begin{align}
 \EE_{\Delta, \Delta'} \exp((4n \ell_1 \ell_2 \rho^2 \langle\Delta, \Delta' \rangle)^2d^{d_Z}/2) =  \EE_{\Delta} \exp((4n \ell_1 \ell_2 \rho^2 \sum_{xy \in [\ell_1']\times[\ell_2']}\delta_{xy})^2d^{d_Z}/2)\label{exponential:identity:lower:bound:appendix}
\end{align}
Note that if one has i.i.d. Rademacher random variables $\delta_{xy}$ and i.i.d. standard normal variables $W_{xy}$ for any nonnegative integers $a_{xy}$ for  $ x,y \in [\ell_1']\times [\ell_2']$ one has
$$
\EE \prod_{xy \in [\ell_1']\times [\ell_2']} \delta^{a_{xy}}_{xy}\leq \EE \prod_{xy \in [\ell_1']\times [\ell_2']} W^{a_{xy}}_{xy}
$$
Expanding the exponential function in \eqref{exponential:identity:lower:bound:appendix} one can control all moments of $\delta_{xy}$ using the inequality above with corresponding moments of $W_{xy}$. Thus we conclude
\begin{align}
\EE_{\Delta} \exp((4n \ell_1 \ell_2 \rho^2 \sum_{xy \in [\ell_1']\times[\ell_2']}\delta_{xy})^2d^{d_Z}/2)\leq \EE_{\bW} \exp((4n \ell_1 \ell_2 \rho^2 \sum_{xy \in [\ell_1']\times[\ell_2']}W_{xy})^2d^{d_Z}/2)
\end{align}
The random variable $\sum_{xy \in [\ell_1']\times[\ell_2']}W_{xy} \sim N(0, \ell_1'\ell_2')$ and therefore $(\sum_{xy \in [\ell_1']\times[\ell_2']}W_{xy})^2/\ell_1'\ell_2' := \chi^2$ has a $\chi^2(1)$ distribution. We have 
\begin{align}
 \EE_{\bW} \exp((4n \ell_1 \ell_2 \rho^2 \sum_{xy \in [\ell_1']\times[\ell_2']}W_{xy})^2d^{d_Z}/2) \leq  \EE_{\chi^2} \exp((4n \ell_1 \ell_2 \rho^2)^2\ell_1'\ell_2'\chi^2d^{d_Z}/2).
\end{align}
Suppose now that $(4n \ell_1 \ell_2 \rho^2)^2d^{d_Z}\ell_1'\ell_2' < 1$. The above is the mgf of a chi-squared random variable and hence equals to
\begin{align*}
\sqrt{\frac{1}{1 - (4n \ell_1 \ell_2 \rho^2)^2d^{d_Z}\ell_1'\ell_2'}}.
\end{align*}
This quantity can be made arbitrarily close to $1$ provided that $(4n \ell_1 \ell_2 \rho^2)^2d^{d_Z}\ell_1'\ell_2' $ is small. Based on this select $\frac{1}{d} \asymp \frac{(\ell_1\ell_2)^{1/(d_Z + 4)}}{n^{2/(4 + d_Z)}} \wedge 1$, $\rho \asymp \frac{1}{\ell_1 \ell_2 d^{d_Z/2 + 1}}$ for some sufficiently small constants. This ensures $\frac{1}{\ell_1 \ell_2} - \rho \sqrt{d} \|h\|_{\infty} \geq 0$ and $\ell_1 \ell_2 \rho d^{d_Z/2}d  \leq C$. With these choices we obtain that the critical radius is bounded from below by a constant times $\frac{1}{d} \asymp  \frac{(\ell_1\ell_2)^{1/(d_Z + 4)}}{n^{2/(4 + d_Z)}} \wedge 1$. 
\end{proof}

Next we show a matching upper bound, only in the case when $d_Z \leq 2$, and in the case when $\ell_1$ and $\ell_2$ are not allowed to scale with $n$. 

\begin{theorem}[Finite Discrete $X$, $Y$ Upper Bound] \label{main:theorem:finite:discrete:XY:app} Set $d = \lceil n^{2/(4 + d_Z)} \rceil$ and let $\tau = \zeta d^{d_Z/2}$ for a sufficiently large absolute constant $\zeta$ (depending on $L$), where $d_Z \leq 2$ (this latter  condition is required in order to ensure that sufficiently many bins will have at least 4 samples).  Finally, suppose that $\varepsilon \geq c n^{-2/(4 + d_Z)}$, for a sufficiently large constant $c$ (depending on $\zeta$, $L$, $\ell_1,\ell_2$).
Then we have that 
\begin{align*}
\sup_{p \in \cP_{0,[0,1]^{d_Z},\TV^2}'(L) \cup \cP_{0,[0,1]^{d_Z},\TV}'(L) \cup \cP_{0,[0,1]^{d_Z},\chi^2}'(L)} \EE_p[\psi_\tau(\cD_N)] & \leq \frac{1}{10},\\ 
\sup_{p \in \{p \in \cQ_{0,[0,1]^{d_Z},\TV}'(L): \inf_{q \in \cP'_{0,[0,1]^{d_Z}}} \|p - q\|_1 \geq \varepsilon\}} \EE_p[ 1- \psi_\tau(\cD_N)] & \leq \frac{1}{10} + \exp(-n/8).
\end{align*}
\end{theorem}

\begin{proof} The proof follows very closely that of Theorem \ref{main:theorem:finite:discrete:XY} so we omit most details. We only note that in the variance calculations the $d$ has to be substituted with $d^{d_Z}$ and the threshold needs to be substituted from $\sqrt{d}$ to $\sqrt{d}^{d_Z}$. This yields the following two rates 
\begin{align*}
\sqrt{\frac{\sqrt{d}^{d_Z}}{n}} \vee \frac{1}{d},
\end{align*}
and 
\begin{align*}
\frac{\sqrt[4]{\sqrt{d}^{d_Z}d^{3d_Z}}}{n} \vee \frac{1}{d},
\end{align*}
It turns out that when $d \asymp n^{2/(4 + d_Z)}$, and $d_Z \leq 2$ the first rate dominates, hence the proof is complete. 
\end{proof}

\subsection{Continuous Case}

Define $\cP_{0, [0,1]^{2 + d_Z}}$ and $\cQ_{0,[0,1]^{2 + d_Z}}$ analogously to $\cP_{0, [0,1]^{3}}$ and $\cQ_{0,[0,1]^{3}}$. 

\begin{definition}[Null Lipschitzness] \leavevmode 
Let $\cP_{0,[0,1]^{2 + d_Z}, \TV}(L) \subset \cP_{0, [0,1]^{2 + d_Z}}$ 
be the collection of distributions $p_{X,Y,Z}$ such that for all $z,z' \in [0,1]$ we have:
\begin{align*}
\|p_{X | Z = z} - p_{X | Z = z'}\|_1 \leq L \|z - z'\|_2 \mbox{ and } \|p_{Y | Z = z} - p_{Y | Z = z'}\|_1 \leq L \|z - z'\|_2,
\end{align*}
where $p_{X| Z = z}$ and $p_{Y | Z = z}$ denote the conditional distributions of $X | Z = z$ and $Y | Z = z$ under $p_{X,Y,Z}$  respectively. 
\end{definition}

\begin{definition}[Alternative Lipschitzness]\label{alternative:smoothness:holder:app} Let $\cQ_{0,[0,1]^{2 + d_Z}, \TV}(L,s) \subset \cQ_{0,[0,1]^{2 + d_Z}}$ be the collection of distributions $p_{X,Y,Z}$ 
such that for all $z, z' \in [0,1]^{d_Z}$ we have
\begin{align*}
\|p_{X,Y | Z= z} - p_{X,Y | Z = z'}\|_1 \leq L \|z - z'\|_2, 
\end{align*}
where $p_{X,Y| Z = z}$ denotes the conditional distribution of $X,Y | Z = z$ under $p_{X,Y,Z}$. In addition we assume that for all $z,x,y \in [0,1]$: $p_{X,Y|Z}(x,y|z) \in \cH^{2,s}(L)$. 
\end{definition}

 \begin{theorem}[Critical Radius Lower Bound]\label{lower:bound:continuous:case:app} Let $\overline \cH_0 = \cP_{0,[0,1]^{2 + d_Z}}$. Suppose that $\cH_0$ is $\cP_{0,[0,1]^{2 + d_Z}, \TV}(L)$, and $\cH_1 = \cQ_{0,[0,1]^{2 + d_Z}, \TV}(L,s)$ for some fixed $L, s\in \RR^+$. Then we have that for some absolute constant $c_0 > 0$,
 \begin{align*}
\varepsilon_n(\cH_0,\overline \cH_0, \cH_1) \geq \frac{c_0}{n^{2s/((4 + d_Z)s + 2)}}.
\end{align*}
 \end{theorem}

\begin{proof}[Proof of Theorem \ref{lower:bound:continuous:case:app}]  Suppose $(X,Y,Z) \in [0,1]^{2 + d_Z}$ are variables with a joint density with respect to the Lebesgue measure in $[0,1]^{2 + d_Z}$. Under the null hypothesis we specify the distribution as $p_{X,Y,Z}(x,y,z) = 1$ for all $(x,y,z) \in [0,1]^{2 + d_Z}$, or in other words the variables have independent uniform distributions on $[0,1]$. Clearly this distribution belongs to the sets $\cP_{0,[0,1]^{2 + d_Z},\TV}(L)$ and $\cP_{0,[0,1]^{2 + d_Z},\chi^2}(L)$. Under the alternative hypothesis we specify the distribution as
\begin{align*}
q_{X,Y|Z}(x,y|z) = 1 + \gamma_{\Delta}(x,y)\eta_{\nu}(z),
\end{align*}
where as in the proof of Theorem \ref{first:lower:bound} 
$$
\eta_{\nu}(z) = \rho \sum_{j \in [d]} \nu_{j_1,\ldots j_{d_Z}}  \prod_{k = 1}^{d_Z} h_{j_k, d}(z_k),
$$
where $\rho > 0$ is a constant, $d \in \NN$, $\nu_{j_1,\ldots j_{d_Z}} \in \{-1,+1\}$ ,and $h_{j, d}(z) = \sqrt{d}h(d z - j + 1)$ for $z \in [(j-1)/d, j/d]$, and $h$ is an infinitely differentiable function supported on $[0,1]$ 
such that $\int h(z) dz = 0$ and $\int h^2(z) dz = 1$. Furthermore we take
$$
\gamma_{\Delta}(x,y) =  \rho^2 \sum_{j \in [d']}\sum_{i \in [d']} \delta_{ij} h_{i, d'}(x)h_{j, d'}(y),\footnotemark
$$
\footnotetext{Here $\Delta = \{\delta_{ij}\}_{i,j\in[d]}$.}and we let the marginal distribution of $Z$ be uniform on $[0,1]^{d_Z}$ (i.e. we let $q_Z = p_Z \equiv 1$). In order for this perturbation to be meaningful we need that $1 \geq \sqrt{(d')^2}\sqrt{d}^{d_Z} \|h\|_{\infty}^{2 + d_Z} \rho^3$. It is simple to check that $\int_{[0,1]^2} q_{X,Y|Z}(x,y|z) dx dy = 1$. Let us now check what are the marginals of such a distribution conditioned on $z$. We have
$$
q_{X|Z}(x|z) = \int_{[0,1]} q_{X,Y | Z}(x,y|z) dy = 1 +  \eta_{\nu}(z)\int_{[0,1]} \gamma_{\Delta}(x,y) dy = 1.
$$
Similarly $\int_{[0,1]} q_{X,Y|Z}(x,y|z) dx = 1$. It is therefore clear that $q_{X|Z}(x|z) q_{Y|Z}(y|z)  \in \cH^{2,s}(L)$ for any $L$. We now check how far away is the distribution $q_{X,Y,Z}(x,y,z)= q_{X,Y|Z}(x,y|z)q_Z(z) = q_{X,Y|Z}(x,y|z)p_Z(z) = q_{X,Y|Z}(x,y|z)$ with respect to $p_{X,Y,Z}(x,y,z)$ (note that $p_{X,Y,Z} = q_{X | Z} q_{Y|Z} q_Z$) in total variation. We have
\begin{align*}
\MoveEqLeft \|q_{X,Y|Z}q_Z - p_{X,Y,Z}\|_1 = \int_{[0,1]^{2 + d_Z}} |\gamma_{\Delta}(x,y)\eta_{\nu}(z)| dx dy dz \\
&= \int_{[0,1]}  \rho \sum_{j \in [d']} |h_{j,d'}(x)| dx \int_{[0,1]}\rho \sum_{j \in [d']} |h_{j,d'}(y)|  dy \int_{[0,1]^{d_Z}} \rho \sum_{j \in [d]}  \prod_{k = 1}^{d_Z} |h_{j_k, d}(z_k)| \prod_{k = 1}^{d_Z} dz_k.
\end{align*} 
Calculating each of the above integrals and multiplying them yields 
\begin{align*}
\|q_{X,Y|Z}p_Z - p_{X,Y,Z}\|_1 = \|q_{X,Y,Z} - q_{X | Z} q_{Y|Z} q_Z\|_1 = \sqrt{d}^{d_Z} \sqrt{(d')^2} \rho^3 c^{2 + d_Z},
\end{align*}
where $c = \int_{[0,1]} |h(x)| dx$. Using Lemma \ref{TV:proj:vs:all} (here we use this lemma with a slight abuse of notation since the lemma is only valid for discrete $X,Y$ and continuous $Z$, but the same proof extends to the continuous case) we have that 
$$\inf_{p \in \cP_{0,[0,1]^{2 + d_Z}}}\|q - p\|_1 \geq \frac{\sqrt{d}^{d_Z} \sqrt{(d')^2} \rho^3 c^{2 + d_Z}}{6}.$$ Next we check that the TV between the distributions $(X,Y | Z = z)$ and $(X,Y| Z = z')$ is Lipschitz.
\begin{align*}
\int_{[0,1]^2}|\gamma_{\Delta}(x,y)| |\eta_{\nu}(z) - \eta_{\nu}(z')| dx dy = \sqrt{(d')^2} \rho^2 c^2  |\eta_{\nu}(z) - \eta_{\nu}(z')|. 
\end{align*}
It is simple to see that by telescoping and triangle inequality we have
\begin{align*}
\bigg|\rho \sum_{j \in [d]} \nu_{j_1,\ldots j_{d_Z}}  \prod_{k = 1}^{d_Z} h_{j_k, d}(z_k) - \rho \sum_{j \in [d]} \nu_{j_1,\ldots j_{d_Z}}  \prod_{k = 1}^{d_Z} h_{j_k, d}(z'_k)\bigg|\leq \rho \sqrt{d}^{d_Z} \|h\|^{d_Z-1}_{\infty} d \|h'\|_{\infty} \sqrt{d_Z} \|\zb - \zb'\|_2,
\end{align*}
and therefore we need $\sqrt{d}^{d_Z}d \sqrt{(d')^2}\rho^3$ to be smaller than some constant. Next we check that $q(x, y | z)$ belongs to the H\"{o}lder class in $x$ and $y$. We have that 
\begin{align*}
\MoveEqLeft \bigg|\frac{\partial^k}{\partial x^k} \frac{\partial^{\lfloor s \rfloor - k}}{\partial y^{\lfloor s \rfloor - k}}\gamma_{\Delta}(x,y)\eta_{\nu}(z) - \frac{\partial^k}{\partial x^k} \frac{\partial^{\lfloor s \rfloor - k}}{\partial y^{\lfloor s \rfloor - k}}\gamma_{\Delta}(x',y')\eta_{\nu}(z)\bigg| \\
& \leq \sqrt{d}\rho \|h\|_{\infty} \bigg|\frac{\partial^k}{\partial x^k} \frac{\partial^{\lfloor s \rfloor - k}}{\partial y^{\lfloor s \rfloor - k}}\gamma_{\Delta}(x,y) - \frac{\partial^k}{\partial x^k} \frac{\partial^{\lfloor s \rfloor - k}}{\partial y^{\lfloor s \rfloor - k}}\gamma_{\Delta}(x',y')\bigg|\\
& = \sqrt{d}\rho \|h\|_{\infty}\rho^2 (\sqrt{d'})^2 (d')^{\lfloor s \rfloor} \bigg| \sum_{i,j \in [d']} \delta_{ij}\bigg[ h^{(k)}(d'x - i + 1) h^{(\lfloor s \rfloor -k)}(d'y - j + 1) \\
& -h^{(k)}(d'x' - i + 1) h^{(\lfloor s \rfloor -k)}(d'y' - j + 1)  \bigg]\bigg|.
\end{align*}
Suppose now that $x \in \bigg[\frac{i_x - 1}{d'}, \frac{i_x}{d'}\bigg], y \in \bigg[\frac{j_y - 1}{d'}, \frac{j_y}{d'}\bigg]$, and $x' \in \bigg[\frac{i_{x'} - 1}{d'}, \frac{i_{x'}}{d'}\bigg], y' \in \bigg[\frac{j_{y'} - 1}{d'}, \frac{j_{y'}}{d'}\bigg]$. Therefore the above summation can be bounded as
\begin{align*}
& \sqrt{d}\rho \|h\|_{\infty}\rho^2 (\sqrt{d'})^2 (d')^{\lfloor s \rfloor} \bigg[\bigg| h^{(k)}(d'x - i_x + 1) h^{(\lfloor s \rfloor -k)}(d'y - j_y + 1) -h^{(k)}(d'x' - i_x + 1) h^{(\lfloor s \rfloor -k)}(d'y' - j_y + 1)\bigg|\\
&+  \bigg| h^{(k)}(d'x - i_{x'} + 1) h^{(\lfloor s \rfloor -k)}(d'y - j_{y'} + 1) -h^{(k)}(d'x' - i_{x'} + 1) h^{(\lfloor s \rfloor -k)}(d'y' - j_{y'} + 1)\bigg|  \bigg].
\end{align*}
Next we will handle the first expression in the bracket above:
\begin{align*}
& \leq |h^{(k)}(d'x - i_x + 1) -h^{(k)}(d'x' - i_x + 1)| |h^{(\lfloor s \rfloor -k)}(d'y - j_y + 1)| \\
& +  |h^{(k)}(d'x' - i_x + 1)| |h^{(\lfloor s \rfloor -k)}(d'y - j_y + 1) - h^{(\lfloor s \rfloor -k)}(d'y' - j_y + 1)|\\
& \leq d' \|h^{(k + 1)}\|_\infty |x - x'| \|h^{(\lfloor s \rfloor -k)}\|_{\infty} \wedge (2 \|h^{(k)}\|_\infty\|h^{(\lfloor s \rfloor -k)}\|_{\infty})  \\
& + d' \|h^{(\lfloor s \rfloor -k) + 1}\|_\infty |y - y'| \|h^{k}\|_{\infty} \wedge (2\|h^{(\lfloor s \rfloor -k)}\|_\infty \|h^{k}\|_{\infty})\\
& \leq C(1 \wedge d' \sqrt{(x - x')^2 + (y - y')^2})\\
& \leq C (d' \sqrt{(x - x')^2 + (y - y')^2})^{s - \lfloor s \rfloor},
\end{align*}
where in the last inequality we used that $(1 \wedge u)^a \leq u^a$ for $u > 0$ and $0 \leq a \leq 1$. We can handle the second expression in the bracket above in a similar way. 

In addition it is clear that any lower order $k \leq \lfloor s \rfloor$ partial derivatives with respect to $x$ and $y$ of $q(x,y|z)$ are bounded by $(\sqrt{d})^{d_Z}\rho^3 \|h\|^{d_Z}_{\infty} \sqrt{d'}^2 (d')^{k} C$ for some constant $C$ which will depend on the function $h$.

It therefore suffices that $\sqrt{d}^{d_Z}\rho^3 (\sqrt{d'})^2 (d')^s$ to be smaller than a constant and we will have both conditions satisfied. Now we write down the likelihood ratio between the null and the alternative mixing over all choices of Rademacher vector and matrix $\nu, \Delta$:
\begin{align*}
W = \EE_{\nu,\Delta} \prod_{i = 1}^n (1 + \gamma_{\Delta}(X_i, Y_i) \eta_{\nu}(Z_i)).
\end{align*}

The second moment of $W$ is 
\begin{align*}
\EE W^2 & = \EE_{\nu, \nu',\Delta,\Delta'}  \prod_{i = 1}^n \EE_0(1 + \gamma_{\Delta}(X_i, Y_i) \eta_{\nu}(Z_i))(1 +\gamma_{\delta'}(X_i, Y_i) \eta_{\nu'}(Z_i))\\
& = \EE_{\nu, \nu',\Delta,\Delta'}  \prod_{i = 1}^n (1 + \EE_0 \gamma_{\Delta}(X_i, Y_i) \eta_{\nu}(Z_i)\gamma_{\delta'}(X_i, Y_i) \eta_{\nu'}(Z_i)),
\end{align*}
where the above follows from the fact that $\EE_0 \eta_{\nu}(Z_i) = 0$ (and that $X_i$ and $Y_i$ are independent of $Z_i$ under the null hypothesis). Continuing the identities yields
\begin{align*}
\EE W^2 & =  \EE_{\nu, \nu',\Delta,\Delta'}  \prod_{i = 1}^n (1 + \EE_0 \gamma_{\Delta}(X_i, Y_i) \gamma_{\delta'}(X_i, Y_i) \EE_0\eta_{\nu}(Z_i) \eta_{\nu'}(Z_i))\\
& = \EE_{\nu, \nu',\Delta,\Delta'}  \prod_{i = 1}^n (1 + \rho^6 \langle\Delta, \Delta'\rangle\langle \nu,\nu'\rangle),
\end{align*}
where $\langle\Delta,\Delta' \rangle= \Tr(\Delta\T \Delta')$. From here the proof can continue as in Theorem \ref{first:lower:bound}. The final expression that needs to be smaller than a constant is $(n \rho^6)^2 d^{d_Z} (d')^2$. Set $d \asymp n^{2s/((4 + d_Z)s + 2)}$, $d' = d^{1/s}$, $\rho^3 \asymp d^{-(1 + d_Z/2 + 1/s)}$. This results in a rate $\asymp 1/d \asymp n^{-2s/((4 + d_Z)s + 2)}$. 
\end{proof}

We will now sketch the details of the following theorem:

\begin{theorem}[Continuous $X,Y,Z$ Upper Bound]\label{continuous:case:upper:bound:theorem:app} Suppose $s \geq 1$ and $d_Z$ is $\leq 2$ (this latter  condition is required in order to ensure that sufficiently many bins will have at least 4 samples). Set $d = \lceil n^{2s/((4 + d_Z)s + 2)} \rceil$, $d' = d^{1/s}$ and set the threshold $\tau = \zeta(\sqrt{ d})^{d_Z}$ for a sufficiently large $\zeta$ (depending on $L$). Then, for a sufficiently
 large absolute constant $c$ (depending on $\zeta, L$),
 when
 $\varepsilon \geq c n^{-2s/((4 + d_Z)s + 2)}$,
we have that
\begin{align*}
\sup_{p \in \cP_{0,[0,1]^{2 + d_Z},\TV}(L)} \EE_p[\psi_\tau(\cD'_k)] & \leq \frac{1}{10},\\ 
\sup_{p \in \{p \in \cQ_{0,[0,1]^{2 + d_Z}, \TV}(L,s): \inf_{q \in \cP_{0,[0,1]^{2 + d_Z}}} \|p - q\|_1 \geq \varepsilon\}} \EE_p[1 - \psi_\tau(\cD'_k)] & \leq \frac{1}{10} + \exp(-n/8).
\end{align*}
\end{theorem}

\begin{proof}[Proof of Theorem \ref{continuous:case:upper:bound:theorem:app}]

For what follows set for brevity $\ell_1 = \ell_2 = d^{1/s}$. We only need minor modifications in the variance calculation to track the dimension of $Z$. 

We will now select a threshold at the level of $\zeta (\sqrt{d})^{d_Z}$, and will give conditions on the minimum sample size for each of the cases. We will use $\gtrsim$ in the sense bigger up to an absolute constant. We will assume that $\frac{\varepsilon}{2} -  \frac{C}{d} \geq \frac{\varepsilon}{4}$ so that $\eta \geq \frac{\varepsilon}{4}$. 
\begin{itemize}
\item In the first sub-case we have to satisfy $\frac{n \eta^2}{\sqrt{\ell_1\ell_2}} \gtrsim \zeta(\sqrt{ d})^{d_Z}$. This is ensured when 
\begin{align}\label{epsilon:condition:one:app}
\varepsilon \gtrsim \sqrt{\frac{(\sqrt{d})^{d_Z}\sqrt{\ell_1 \ell_2}}{n}} \vee \frac{1}{d}.
\end{align}
\item In the second sub-case we have $\frac{n^{3/2}\eta^2}{ \ell_1\sqrt{\ell_2}\frac{ \sqrt{n\ell_1}}{\ell_2}} \gtrsim \zeta(\sqrt{ d})^{d_Z}$ or $\frac{n^{3/2}\eta^2}{ \ell_1\sqrt{\ell_2}\sqrt{d^{d_Z}}} \gtrsim \zeta(\sqrt{ d})^{d_Z},$
This is implied when 
\begin{align}\label{epsilon:condition:two:app}
\varepsilon \gtrsim \min\bigg(\sqrt{ \frac{(\sqrt{ d})^{d_Z}\ell_1^{3/2}}{\sqrt{\ell_2}n}}, \sqrt{\frac{ \sqrt{d}^{d_Z} \sqrt{d^{d_Z}} \ell_1 \sqrt{\ell_2}}{n^{3/2}}}, \frac{d^{d_Z} \ell_1}{n} \bigg) \vee \frac{1}{d}.
\end{align}
\item In the third sub-case case we need $\frac{n^2 \eta^2}{36 \ell_1\ell_2d^{d_Z}} \gtrsim (\sqrt{ d})^{d_Z}$ 
which happens when 
\begin{align}\label{epsilon:condition:three:app}
\varepsilon \gtrsim \min\bigg(\frac{ \zeta^{1/4} \sqrt{d^{d_Z} \sqrt{d}^{d_Z}\ell_1\ell_2}}{n}, \frac{d^{d_Z}\ell_2}{n}\bigg) \vee \frac{1}{d},
\end{align}
where the last condition enforces when this case is not feasible. 
\item Finally in the second case we need $\frac{n^4 \eta^4 }{16\ell_1 \ell_2  d^{3d_Z}} \gtrsim \sqrt{\zeta d^{d_Z}}$, which is implied when 
\begin{align}\label{epsilon:condition:four:app}
\varepsilon \gtrsim \min\bigg(\frac{\zeta^{1/8} (d^{3d_Z} \sqrt{d}^{d_Z}\ell_1\ell_2)^{1/4}}{n}, \frac{d^{d_Z}}{n}\bigg)\vee \frac{1}{d}
\end{align}
\end{itemize}

\noindent \textbf{Analysis of the Variance.} 

Now we derive a bound on the variance of the statistic 
$$
\sum_{m \in [d^{d_Z}]}\sigma_m \omega_m\mathbbm{1}(\sigma_m \geq 4) \frac{\varepsilon^2_m}{\ell_1\ell_2}
$$
Recall now that $\sigma_m \sim \operatorname{Poi}(\alpha_m)$ are independent and therefore
\begin{align}
\Var\bigg[\sum_{m \in [d^{d_Z}]}\sigma_m \omega_m\mathbbm{1}(\sigma_m \geq 4) \frac{\varepsilon^2_m}{\ell_1\ell_2}\bigg] & = \sum \Var\bigg[\sigma_m \omega_m\mathbbm{1}(\sigma_m \geq 4) \frac{\varepsilon^2_m}{\ell_1\ell_2}\bigg] \nonumber \\
& = \sum_{m \in [d^{d_Z}]}\frac{\varepsilon^4_m}{\ell_1^2\ell^2_2} \Var (\sigma_m \omega_m\mathbbm{1}(\sigma_m \geq 4)) \nonumber \\
& \leq \frac{C'}{\ell_1\ell_2} \sum_{m \in [d^{d_Z}]}\frac{\varepsilon^2_m}{\ell_1\ell_2} \EE (\sigma_m \omega_m\mathbbm{1}(\sigma_m \geq 4)) \nonumber\\
& \leq \frac{C'}{\ell_1\ell_2} \EE\bigg[\sum_{m \in [d^{d_Z}]}\sigma_m \omega_m\mathbbm{1}(\sigma_m \geq 4) \frac{\varepsilon^2_m}{\ell_1\ell_2}\bigg],\label{D:variance:bound:app}
\end{align}
where in the next to last inequality we used Claim 2.2. of \cite{canonne2018testing}, and $C'$ is an absolute constant described in that claim. 

We now bound the variance of the statistic $T$ (recall the definition \eqref{T:stat:def}). Since $T_m$ (recall definition \ref{Tm:def}) are independent given $\sigma_m, R_m$ we have that
$$
\Var[T | \sigma, R] = \sum_{m \in [d^{d_Z}]} \Var[T_m | \sigma_m, R_m].
$$
Next, by definition of $T_m$ we have that $ \Var[T_m | \sigma_m, R_m] = \sigma^2_m \omega^2_m \mathbbm{1}(\sigma_m \geq 4) \Var[U_m | R_m]$. Using the bound on the variance $\Var[U_m | R_m]$ \eqref{bound:on:the:variance}, we have to control four terms. We do so below. Denote
$$
E := \sum_{m \in [d^{d_Z}]} \omega_m^2 \|q_{\Pi,A_m}(m)\|_2^2\mathbbm{1}(\sigma_m \geq 4)
$$

The first term we need to control is 
\begin{align*}
\MoveEqLeft \sum_{m \in [d]} \sigma^2_m \omega^2_m \mathbbm{1}(\sigma_m \geq 4) \frac{ \|q_{A_m}(m) - q_{\Pi,A_m}(m)\|_2^2\|q_{\Pi,A_m}(m)\|_2}{\sigma_m} \\
&\leq \sqrt{\bigg(\sum_{m \in [d]} \omega_m^2 \|q_{\Pi,A_m}(m)\|_2^2\mathbbm{1}(\sigma_m \geq 4)\bigg)}\sqrt{\sum_{m \in [d]} (\sigma_m \omega_m\|q_{A_m}(m) - q_{\Pi,A_m}(m)\|_2^2)^2 \mathbbm{1}(\sigma_m \geq 4)}\\
&\leq E^{1/2}\sum_{m \in [d^{d_Z}]} \sigma_m \omega_m\|q_{A_m}(m) - q_{\Pi,A_m}(m)\|_2^2 \mathbbm{1}(\sigma_m \geq 4)\\
& = E^{1/2} \EE[T | \sigma, R],
\end{align*}
where we used Cauchy-Schwarz and the monotonicity of $L_p$ norms. The second term is 
\begin{align*}
\MoveEqLeft \sum_{m \in [d^{d_Z}]} \sigma^2_m \omega^2_m \mathbbm{1}(\sigma_m \geq 4) \frac{ \|q_{A_m}(m) - q_{\Pi,A_m}(m)\|_2^3}{\sigma_m} = \sum_{m \in [d^{d_Z}]} \sqrt{\frac{\omega_m}{\sigma_m}}\sigma^{3/2}_m \omega^{3/2}_m \mathbbm{1}(\sigma_m \geq 4)  \|q_{A_m}(m) - q_{\Pi,A_m}(m)\|_2^3 \\
& \leq \sum_{m \in [d^{d_Z}]} (\sigma_m \omega_m \mathbbm{1}(\sigma_m \geq 4)  \|q_{A_m}(m) - q_{\Pi,A_m}(m)\|_2^2)^{3/2}\\
& \leq \bigg(\sum_{m \in [d^{d_Z}]} \sigma_m \omega_m \mathbbm{1}(\sigma_m \geq 4)  \|q_{A_m}(m) - q_{\Pi,A_m}(m)\|_2^2\bigg)^{3/2}\\
& = \EE[T | \sigma, R]^{3/2},
\end{align*}
where we used that $\omega_m \leq \sigma_m$ by definition and the monotonicity of the $L_p$ norms. The third term is
\begin{align*}
\MoveEqLeft \sum_{m \in [d^{d_Z}]} \sigma^2_m \omega^2_m \mathbbm{1}(\sigma_m \geq 4) \frac{ \|q_{\Pi,A_m}(m)\|_2^2}{\sigma^2_m} = E.
\end{align*}
Finally the fourth term is 
\begin{align*}
\MoveEqLeft \sum_{m \in [d^{d_Z}]} \sigma^2_m \omega^2_m \mathbbm{1}(\sigma_m \geq 4) \frac{ \|q_{A_m}(m) - q_{\Pi,A_m}(m)\|_2^2}{\sigma^2_m} \leq \sum_{m \in [d^{d_Z}]} \omega_m\sigma_m \mathbbm{1}(\sigma_m \geq 4) \|q_{A_m}(m) - q_{\Pi,A_m}(m)\|_2^2\\
& = \EE[T | \sigma, R]. 
\end{align*}
We conclude that
\begin{align}\label{varT:bound:app}
\Var[T | \sigma, R] \leq C(E + (E^{1/2} + 1)\EE[T | \sigma, R] + \EE[T | \sigma, R]^{3/2}).
\end{align}
Now we will show that $\EE[E | \sigma] = O(\min(d^{d_Z}, N))$. We start by analyzing the expectation of one term from $E$ below.
\begin{align*}
\EE[\omega_m^2 \|q_{\Pi,A_m}(m)\|_2^2\mathbbm{1}(\sigma_m \geq 4) | \sigma_m] = \omega_m^2\mathbbm{1}(\sigma_m \geq 4)\EE[ \|q_{\Pi,A_m}(m)\|_2^2 | \sigma_m] \leq  \frac{\omega_m^2\mathbbm{1}(\sigma_m \geq 4) }{(1 + t_{1,m})(1 + t_{2,m})},
\end{align*}
where we applied \eqref{bound:on:l2:norm:qpi}. Recall that $t_{i,m} = \min((\sigma_m - 4)/4, \ell_i)$ and $\omega^2_m = \min(\sigma_m, \ell_1)\min(\sigma_m, \ell_2)$. Thus 
$$
\frac{\omega_m^2}{(1 + t_{1,m})(1 + t_{2,m})} \leq O(1).
$$
We conclude that 
\begin{align}\label{E:given:sigma:bound:app}
\EE[E | \sigma] = \EE[\sum_{m \in [d^{d_Z}]} \omega_m^2 \|q_{\Pi,A_m}(m)\|_2^2\mathbbm{1}(\sigma_m \geq 4) | \sigma_m] \leq O(1)\sum_{m \in [d^{d_Z}]} \mathbbm{1}(\sigma_m \geq 4)  \leq O(1)\min(d^{d_Z},N).
\end{align}

We have the following result

\begin{lemma}\label{variance:bound:flattening:app}Suppose $\inf_{q \in \cP_{0, [0,1]}'}\|p_{X,Y,Z} - q\|_1 \geq \varepsilon$, where $\varepsilon \geq  C\frac{L}{d}$ and it satisfies conditions \eqref{epsilon:condition:one:app}, \eqref{epsilon:condition:two:app}, \eqref{epsilon:condition:three:app} and \eqref{epsilon:condition:four:app}. Then with probability at least $19/20$ over $\sigma, R$ we have $\EE [T | \sigma, R] = \Omega(\sqrt{\zeta d^{d_Z}})$ and 
\begin{align}\label{variance:bound:flattening:app:equation}
\Var[T | \sigma, R] \leq O(d^{d_Z} + (\sqrt{d^{d_Z}} + 1)\EE[T | \sigma, R] + \EE[T | \sigma, R]^{3/2}).
\end{align}
\end{lemma}

\begin{proof} Let 
$$
D = \sum_{m \in [d^{d_Z}]}\sigma_m \omega_m\mathbbm{1}(\sigma_m \geq 4) \frac{\varepsilon^2_m}{\ell_1\ell_2}.
$$
We first showed that $\EE[T | \sigma, R] \geq D$ for all $\sigma, R$ \eqref{Tm:lower:bound}. We also derived that $\Var[D] \leq O(\EE[D]/(\ell_1\ell_2))$ \eqref{D:variance:bound:app}, and that for the selected regimes of sample size $\EE[D] \gtrsim \sqrt{\zeta d^{d_Z}}$. Therefore we have 
\begin{align*}
\PP_{\sigma, R}\bigg(\EE[T | \sigma, R] \leq \kappa \sqrt{\zeta d^{d_Z}}\bigg)\leq \PP_{\sigma, R}(D \leq O(\EE[D])) \leq O\bigg(\frac{\Var[D]}{(\EE[D])^2}\bigg) = O(1/(\sqrt{\zeta d^{d_Z}}\ell_1 \ell_2)) \leq 1/40,
\end{align*}
for some small enough absolute constant $\kappa$. For the second statement we will use bound \eqref{varT:bound:app}. By \eqref{E:given:sigma:bound:app} we have  
$$
\EE[\EE[E | \sigma]] \leq O(1)\EE \min(d^{d_Z},N) \leq O(1) \min(d^{d_Z}, n). 
$$
Thus by Markov's inequality $E \leq 200 \EE[E] = O(1) d^{d_Z}$ with probability at least $39/40$. Therefore
$$
\PP_{\sigma,R}(\Var[T | \sigma, R] \geq \kappa'(d^{d_Z} + (\sqrt{d^{d_Z}} +1)\EE[T | \sigma, R] + \EE[T | \sigma, R]^{3/2})) \leq 1/40.
$$
A union bound over the two events completes the proof. 
\end{proof}

We now turn to bound the expectation and variance under the null hypothesis. 

\begin{lemma}\label{seconD:variance:bound:app:flattening}  Suppose $p_{X,Y,Z} \in \cP_{0,[0,1]^3, \TV}(L)$. Then with probability at least $19/20$ we have $\EE[T | \sigma, R] \leq C \frac{L^2n}{d^2\sqrt{\ell_1\ell_2}}$ and the variance $\Var[T | \sigma, R]$ satisfies \eqref{variance:bound:flattening:app:equation}. 
\end{lemma}

\begin{proof}
Let us start with bounding $\EE[T | \sigma,R]$ from above. Recall that 
$$
\EE[T | \sigma, R] = \sum_{m \in [d]} \sigma_m \omega_m \|q_{A_m}(m) - q_{\Pi,A_m}(m)\|_2^2\mathbbm{1}(\sigma_m \geq 4)
$$
We will now control $\EE[\|q_{A_m}(m) - q_{\Pi,A_m}(m)\|_2^2 | \sigma_m]$. Recall that
\begin{align}\label{to:upper:bound:new:app}
 \|q_{A_m}(m) - q_{\Pi,A_m}(m)\|_2^2 & = \sum_{x,y} \frac{(q_{xy}(m) - q_{x\cdot}(m) q_{\cdot y}(m))^2}{1 + a^m_{xy}}\\
& \leq \sum_{x,y} (q_{xy}(m) - q_{x\cdot}(m) q_{\cdot y}(m))^2\nonumber
\end{align}
Using the strategy of Lemma \ref{lem:aa} we can bound the above by $C/d^4$ for some constant depending on $L$. 
Therefore,
\begin{align*}
\EE[T | \sigma] & = \sum_{m \in [d^{d_Z}]} \sigma_m \omega_m\mathbbm{1}(\sigma_m \geq 4) \EE_{A_m} \|q_{A_m}(m) - q_{\Pi,A_m}(m)\|_2^2 \\
&\leq C/d^4   \sum_{m \in [d^{d_Z}]} \sigma_m \omega_m\mathbbm{1}(\sigma_m \geq 4) \leq C/d^4   N\sqrt{\ell_1\ell_2}. 
\end{align*}
Hence $\EE T \leq C \frac{n\sqrt{\ell_1\ell_2}}{d^4} $. By Markov's inequality we therefore have
\begin{align*}
\PP_{\sigma, R} \bigg(\EE[T | \sigma,R]  \geq 40 C \frac{n\sqrt{\ell_1\ell_2}}{d^4}\bigg) \leq \frac{1}{40}. 
\end{align*}
Since $\ell_1 = \ell_2 = d^{1/s}$ and $s \geq 1 $we have that the above is smaller than $ C \frac{n\sqrt{\ell_1\ell_2}}{d^4} \leq  C \frac{n}{d^2\sqrt{\ell_1\ell_2}}$, which completes the proof. 
\end{proof}

\noindent \textbf{Putting Things Together.} For what follows suppose that $d$ is selected so that 
\begin{align}\label{d:n:condition:app}
\frac{n}{d^2\sqrt{\ell_1\ell_2}} \asymp (\sqrt{d})^{d_Z}.
\end{align}

\begin{lemma}\label{null:hypothesis:lemma:app} If $p_{X,Y,Z} \in \cP_{0,[0,1]^3, \TV}(L)$ and that \eqref{d:n:condition:app} holds. Then for a sufficiently large absolute constant $\alpha$ we have
\begin{align*}
\PP\bigg(T \geq  (\alpha + 1) \frac{C' n}{d^2 \sqrt{\ell_1\ell_2}}\bigg) \leq \frac{1}{10}.
\end{align*}
\end{lemma}

\begin{proof}
Let $T' = (T | \sigma, R)$. Denote the event from Lemma \ref{seconD:variance:bound:app:flattening} with $\cE$. Then we have 
\begin{align*}
\PP\bigg(T \geq  (\alpha + 1) \frac{C' n}{d^2 \sqrt{\ell_1\ell_2}}\bigg) & = \PP\bigg(T'  \geq  (\alpha + 1)\frac{C' n}{\sqrt{\ell_1\ell_2}d^2}\bigg) \\
& \leq  \PP\bigg(T' \geq (\alpha + 1) \frac{C' n}{\sqrt{\ell_1\ell_2} d^2} \bigg | \cE\bigg) + \PP(\cE^c). 
\end{align*}
Now we have
\begin{align*}
\PP\bigg(T' \geq  (\alpha + 1) \frac{C' n}{\sqrt{\ell_1\ell_2} d^2} \bigg| \cE\bigg) & \leq \PP\bigg(T' - \EE[T |\sigma, R] \geq \alpha \frac{C'n}{\sqrt{\ell_1\ell_2} d^2} \bigg| \cE\bigg) \leq \frac{\Var[T' | \cE]}{(\alpha \frac{C'n}{\sqrt{\ell_1\ell_2}d^2})^2}\\
& \leq \frac{O(d^{d_Z} + (\sqrt{d^{d_Z}} + 1)\EE[T | \sigma, R] + \EE[T | \sigma, R]^{3/2})}{(\alpha \frac{C'n}{\sqrt{\ell_1\ell_2}d^2})^2}\\
& \leq \frac{O(d^{d_Z} + (\sqrt{d^{d_Z}} + 1)(\frac{C' n}{\sqrt{\ell_1\ell_2}d^2}) + (\frac{C'n}{\sqrt{\ell_1\ell_2}d^2})^{3/2})}{(\alpha \frac{C'n}{\sqrt{\ell_1\ell_2}d^2})^2}\\
& \leq \frac{1}{20}
\end{align*}
where the above holds when $\frac{n}{\sqrt{\ell_1\ell_2}d^2} \asymp (\sqrt{d})^{d_Z}$ for a large enough $\alpha$.
\end{proof}

\begin{lemma}\label{alternative:hypothesis:lemma}
If $p_{X,Y,Z}$ is such that $\inf_{q \in \cP_{0,[0,1]}'}\|p_{X,Y,Z}-q\|_1 \geq \varepsilon$, and the conditions of Lemma \ref{variance:bound:flattening:app} hold. Then for a small enough absolute constant $\kappa$ we have that
\begin{align*}
\PP(T \leq \kappa \zeta(\sqrt{ d})^{d_Z}) \leq \frac{1}{10}. 
\end{align*} 
\end{lemma}
\begin{proof}
We apply Chebyshev's inequality to $T' = (T | \sigma, R)$. Let $\cE$ be the event of Lemma \ref{variance:bound:flattening:app}. Set $\tau = \kappa \zeta(\sqrt{ d})^{d_Z}$ for some small enough absolute constant $\kappa$. 
\begin{align*}
\PP(T \leq \tau) &= \PP(T' \leq \tau) \leq \PP(T'\leq \tau| \cE) + \PP(\cE^c). \\
&\leq \PP\bigg(|T' - \EE[T|\sigma, R]| \geq \frac{1}{2} \EE[T | \sigma, R] | \cE\bigg) + \frac{1}{20}.
\end{align*}
Next
\begin{align*}
\PP\bigg(|T' - \EE[T|\sigma, R]| \geq \frac{1}{2} \EE[T | \sigma, R] | \cE\bigg) \leq O\bigg(\frac{d^{d_Z} + (\sqrt{d^{d_Z}} + 1)\EE[T | \sigma, R] + \EE[T | \sigma, R]^{3/2}}{\EE[T | \sigma, R]^2}\bigg)\\
= O(1/\zeta^{1/2}) \leq \frac{1}{20},
\end{align*}
for a large enough value of $\zeta$. 
\end{proof}

Combining Lemmas \ref{null:hypothesis:lemma:app} and \ref{alternative:hypothesis:lemma} we have that if $\kappa \zeta(\sqrt{ d})^{d_Z} \geq (\alpha + 1) n C'/(\sqrt{\ell_1\ell_2}d^2)\asymp (\alpha +1) C' (\sqrt{d})^{d_Z}$, there will be a gap between the values under the null and the alternative hypothesis. This happens when $\zeta$ is large enough. Notice that $\frac{n C'}{\sqrt{\ell_1\ell_2} d^2} \asymp (\sqrt{d})^{d_Z}$ is equivalent to $d \asymp n^{\frac{2s}{(4+ d_Z)s + 2}}$. 
Plugging this in all the inequalities \eqref{epsilon:condition:one:app}, \eqref{epsilon:condition:two:app}, \eqref{epsilon:condition:three:app} and \eqref{epsilon:condition:four:app}, results in new inequalities that need to hold. 
Condition \eqref{epsilon:condition:one:app} is equivalent to
\begin{align*}
\varepsilon \gtrsim n^{-\frac{2s}{(4 + d_Z)s + 2}}. 
\end{align*}
Since the first term of \eqref{epsilon:condition:two:app} is of the same order as the term in \eqref{epsilon:condition:one:app}, the rate will remain unchanged. 
Next, taking \eqref{epsilon:condition:three:app}, we have
$\frac{\sqrt{d^{d_Z} \sqrt{d}^{d_Z}\ell_1\ell_2}}{n} = \frac{d^{3d_Z/4 + 1/s}}{n}$.
Finally the first term in \eqref{epsilon:condition:four:app} equals to $\frac{d^{7d_Z/8 + 1/s}}{n}$. It can be checked that is dominated by the first term in \eqref{epsilon:condition:two:app}, at the value of $d = n^{\frac{2s}{(4 + d_Z)s + 2}}$ whenever $\frac{5s d_Z + 4}{2((4+d_Z)s + 2)}\leq 1$, which holds when $d_Z \leq 2$.

The rest of the proof can continue as in the proof of Theorem \ref{continuous:case:upper:bound:theorem} and Theorem \ref{continuous:case:upper:bound:theorem:Poi}. We omit the details.  
\end{proof}

\end{document}